\documentclass[reqno,11pt]{amsart}
\usepackage{setspace} 
\makeatletter

\usepackage[T1]{fontenc}
\usepackage{lmodern}

\usepackage{fullpage}
\usepackage{textcmds}  
\usepackage{amsmath, amssymb, amsfonts, amstext, verbatim, amsthm, mathrsfs, stmaryrd}
\usepackage{appendix}
\usepackage{mathdots}
\usepackage{ytableau}

\usepackage[all,cmtip,2cell]{xy}
\usepackage{pgf,tikz,pgfplots}
\pgfplotsset{compat=1.15}
\usetikzlibrary{arrows}
\usetikzlibrary{positioning}
\usetikzlibrary{quotes}
\usepackage{tikz-cd}
\usepackage{quiver}
\usepackage{mathtools}
\usepackage{xcolor}

\usepackage{graphicx}
\graphicspath{ {images/} }

\tikzset{
  symbol/.style={
    draw=none,
    every to/.append style={
      edge node={node [sloped, allow upside down, auto=false]{$#1$}}}
  }
}

\usepackage{enumerate}
\usepackage{enumitem}
\setlist{topsep=1em, itemsep=1em}
\usepackage[colorlinks=true,linkcolor=blue,citecolor=blue,urlcolor=blue,citebordercolor={0 0 1},urlbordercolor={0 0 1},linkbordercolor={0 0 1}]{hyperref} 
\usepackage[shortalphabetic]{amsrefs} 
\usepackage{enotez}
\setenotez{backref=true}

\def\mydefc#1{\expandafter\def\csname c#1\endcsname{\mathcal{#1}}}
\def\mydefallc#1{\ifx#1\mydefallc\else\mydefc#1\expandafter\mydefallc\fi}
\mydefallc ABCDEFGHIJKLMNOPQRSTUVWXYZ\mydefallc

\def\mydefb#1{\expandafter\def\csname b#1\endcsname{\mathbf{#1}}}
\def\mydefallb#1{\ifx#1\mydefallb\else\mydefb#1\expandafter\mydefallb\fi}
\mydefallb ABCDEFGHIJKLMNOPQRSTUVWXYZ\mydefallb

\def\mydeffrac#1{\expandafter\def\csname frac#1\endcsname{\mathscr{#1}}}
\def\mydeffracall#1{\ifx#1\mydeffracall\else\mydeffrac#1\expandafter\mydeffracall\fi}
\mydeffracall ABCDEFGHIJKLMNOPQRSTUVWXYZ\mydeffracall

\usepackage[nameinlink, capitalize]{cleveref}
\theoremstyle{plain}
\newtheorem{thm}{Theorem}[section]
\newtheorem{cor}[thm]{Corollary}
\newtheorem{lem}[thm]{Lemma}

\newtheorem{conj}{Conjecture}
\newtheorem{prop}[thm]{Proposition}
\newtheorem{quest}[thm]{Question}

\theoremstyle{definition}
\newtheorem{rem}[thm]{Remark}
\newtheorem{defn}[thm]{Definition}
\newtheorem{hyp}[thm]{Hypotheses}
\newtheorem{setup}[thm]{Setup}
\newtheorem{notn}[thm]{Notation}

\newtheorem{ex}[thm]{Example}

\newtheorem{introthm}{Theorem}

\newtheorem{introconj}{Conjecture}

\def\rm{\mathrm}

\DeclareMathOperator{\rk}{rank}

\newcommand{\DCoh}{\mathcal{D}^{b}}

\DeclareMathOperator{\Coh}{Coh}

\DeclareMathOperator{\Cone}{Cone}

\DeclareMathOperator{\End}{End}
\DeclareMathOperator{\ev}{ev}
\DeclareMathOperator{\Ext}{Ext}

\DeclareMathOperator{\GL}{GL}

\DeclareMathOperator{\Hom}{Hom}
\newcommand{\id}{\mathrm{id}}

\newcommand{\pt}{\mathrm{pt}}

\DeclareMathOperator{\Forg}{Forg}

\newcommand{\RHom}{\mathbf{R}{\Hom}} 

\newcommand{\norm}[1]{\left\lVert#1\right\rVert}

\DeclareMathOperator{\Stab}{Stab}
\DeclareMathOperator{\Sym}{Sym}

\DeclareMathOperator{\Pic}{Pic}

\DeclareMathOperator{\Ob}{Ob}

\DeclareMathOperator{\im}{im}

\DeclareMathOperator{\coker}{coker}

\DeclareMathOperator{\Bl}{Bl}
\DeclareMathOperator{\pr}{pr}

\DeclareMathOperator{\Hilb}{Hilb}
\DeclareMathOperator{\rep}{rep}

\DeclareMathOperator{\logZ}{logZ}

\def\bf{\mathbf}

\def\Astab{\mathcal{A}\Stab}

\newcommand{\CC}{\mathbf{C}}

\renewcommand{\H}{{\rm H}}
\newcommand{\Gr}{{\rm {Gr}}}

\newcommand{\ko}{\mathcal{O}}

\newcommand{\Ku}{\mathrm{Ku}}

\newcommand{\diag}{{\mathrm{diag}}}

\newcommand{\ch}{\mathrm{ch}}
\newcommand{\Ch}{\mathrm{Ch}}

\usepackage{pbox}
\usepackage[normalem]{ulem}

\numberwithin{equation}{section}

\makeatletter

\usepackage{babel}

\begin{document}

\title[Fano NMMP]{toward the noncommutative minimal model program for Fano varieties}

\author[T. Karube]{Tomohiro Karube}
\address{\parbox{0.9\textwidth}{The University of Tokyo,
Graduate School of Mathematical Sciences\\[1pt]
Meguro-ku, Tokyo, 153-8914,Japan.\vspace{1mm}}}
\email{\url{karube-tomohiro803@g.ecc.u-tokyo.ac.jp}}
\author[A. Robotis]{Antonios-Alexandros Robotis}
\address{\parbox{0.9\textwidth}{Columbia University, Department of Mathematics\\  
2990 Broadway, New York, NY, USA, 10027}
\vspace{1mm}}
\email{\url{a.robotis@columbia.edu}}

\author[V. Zuliani]{Vanja Zuliani}
\address{\parbox{0.9\textwidth}{Universit\'e Paris-Saclay, CNRS, Laboratoire de Math\'ematiques d'Orsay\\[1pt]
Rue Michel Magat, B\^at. 307, 91405 Orsay, France
\vspace{1mm}}}
\email{\url{vanja.zuliani@universite-paris-saclay.fr}}

\begin{abstract}
\linespread{1.0}\selectfont
   We study the noncommutative minimal model program, as proposed by Halpern-Leistner, for Fano varieties. We construct lifts of Iritani's quantum cohomology central charge in the following examples: Grassmannians, smooth quadrics, and smooth cubic threefolds and fourfolds. Moreover, we verify that these lifted paths are quasi-convergent and give rise to the expected semiorthogonal decompositions of the bounded derived category. We also construct geometric stability conditions in the examples above and observe that, after a suitable isomonodromic deform\-ation of the quantum cohomology central charge, the quasi-convergent paths for Grassmannians and quadrics can be chosen to start in the geometric region. 
\end{abstract}

\maketitle

\linespread{1.0}\selectfont
\tableofcontents

\addtocontents{toc}{\protect\setcounter{tocdepth}{2}}

\linespread{1.15}\selectfont

\section{Introduction}

Stability conditions on triangulated categories were introduced by Bridgeland \cite{Br07} to formalize the notion of $\Pi$-stability in theoretical physics \cite{DouglasPiStability}. Over the last twenty years, stability conditions have been the subject of intense study both in algebraic geometry \cites{arcara2013minimal,Bayer2013MMPFM, Bayer2011BridgelandSC, Feyzbakhsh2025StabilityCO, Nuer2019MMPVW, Li2015StabilityCO} and related fields \cites{Barbierirep,HKK}. From the perspective of mirror symmetry, when $X$ is a Calabi-Yau or Fano variety, the space $\Stab(X)$ of stability conditions on its derived category $\DCoh(X)$ is of particular interest. In the Calabi-Yau case, $\Stab(X)$ is expected to be closely related to the physicists' stringy K\"ahler moduli space -- see \cite{bridgelandspacesstabilityconditions}. Meanwhile, in the Fano case it is expected that $\Stab(X)$ is related to moduli spaces of Landau-Ginzburg models \cite{DKKCompactifications}, and that certain paths in $\Stab(X)$ should give rise to semiorthogonal decompositions of $\DCoh(X)$ \cites{NMMP,HLJR}.

The present work studies stability conditions on $\DCoh(X)$ when $X$ is a smooth Fano variety. In \Cref{S:constructionofstab}, we construct geometric stability conditions on $\DCoh(X)$, when $X$ is a finite product of Grassmannians, a smooth quadric hypersurface $Q\subset \bf{P}^n$, or a smooth cubic of dimension $3$ or $4$. As a consequence of these results, we construct ``almost geometric'' stability conditions on $\Hilb^n(\bf{P}^2)$, $\Hilb^n(\bf{P}^1\times \bf{P}^1)$ for all $n\ge 1$, and weighted projective stacks of the form $\bf{P}(a_0,\ldots, a_n)$ for $\gcd(a_0,\ldots, a_n) = 1$. 

The rest of the paper applies some of these existence results in the context of the \emph{Non\-commutative Minimal Model Program} (NMMP) of Halpern-Leistner \cite{NMMP}. In \Cref{S:nMMPFano}, we revisit some of the conjectures of the NMMP in the case where $X$ is a smooth Fano variety. We also establish results connecting Iritani's quantum cohomology central charge \cite{iritani_integral} to the space of stability conditions $\Stab(X)$ and show that the veracity of the Gamma II conjecture of \cite{GGI16} implies part of the conjectures of \cite{NMMP} when $X$ is a Grassmannian or quadric. In \Cref{S:examples}, we use these results to construct canonical paths in $\Stab(X)$ for Grassmannians, quadrics, and cubic threefolds and fourfolds.

The main results of the paper are summarized at the end of this introduction, written as Theorems \ref{T:introgeomstab}, \ref{intro:TB}, \ref{intro:TC}, \ref{T:conjecture2cases}, \ref{T:introconj2fullcases}. The reader interested exclusively in the construction of geometric stability conditions is invited to jump directly to \Cref{S:constructionofstab}, which can be read independently.

\subsubsection*{Recollection of the NMMP}
Recently, groundbreaking work of Katzarkov--Kontsevich--Pantev--Yu \cite{KKPY} has profitably used the analytic structure of quantum cohomology \cite{KontsevichManin} to define new birational invariants of varieties, called \emph{atoms}. Parallel to these developments, Halpern-Leistner \cite{NMMP} has proposed the \emph{Noncommutative Minimal Model Program} (NMMP), which aims to associate to a contraction $f:X\to Y$ of a smooth projective variety $X$ a canonical semiorthogonal decomp\-osition of $\DCoh(X)$, unique up to mut\-ation. In the present work, we consider the case where $X$ is Fano and $Y = \pt$. The fundamental idea is that one should associate to $X$ a family of paths in $\Stab(X)$, whose central charges are determined by solutions to the quantum differential equation of $X$. We present here a sketch of the relevant ideas, writing $\H^\bullet(X) = \H^\bullet(X,\bf{C})$ throughout. 

\smallskip

\textbf{Step 1.} By Bridgeland's deformation theorem \cite{Br07}, deformation of a stability condition $\sigma = (Z,\cP) \in \Stab(X)$ is controlled by deform\-ation of its central charge $Z \in \Hom(\H_{\rm{alg}}^\bullet(X),\bf{C})$. An intrinsic way to deform the central charge comes from the quantum differential equation (QDE) of $X$. This is encoded by the flat ``quantum'' connection $\nabla$ on the trivial $\H^\bullet(X)$-bundle, $\cH_X\to B\times \bf{P}^1$, where $B \subseteq \H^\bullet(X)$ is a suitable subspace on which the quantum product converges. Galkin--Golyshev--Iritani \cite{GGI16} study solutions to the \emph{quantum differential equation} at $\tau \in B$, whose solutions are flat sections of the following connection
    \begin{equation}
    \label{E:qde}
        \nabla_{w\partial_w}^\tau = w\frac{\partial}{\partial w} - \frac{1}{w}\cE_\tau\star_\tau(-) + \mu
    \end{equation}
    where $w$ is a holomorphic coordinate on $\bf{C} \subset \bf{P}^1$, $\cE$ is the ``Euler'' section of $\cH_X$, and $\mu \in \End(\H^\bullet(X))$ is the grading operator (see \Cref{SS:quantumcohomology}). 
    There is a canonical fund\-amental solution $\Phi_w^{\tau}:\bf{C}^*\to \End(\H^\bullet(X))$ of \eqref{E:qde}, and following Iritani \cite{iritani_integral} one can define a class of putative central charges
    \[
        \cZ_w^{\tau}(-) = (2\pi w)^{\dim X/2} \int_X \Phi_w^{\tau}(-),
    \]
    called \emph{quantum cohomology central charges} -- see \Cref{defn_perturbed_can_sol}. 

    \smallskip

    \textbf{Step 2.} Next, one lifts $\cZ_w^{\tau} \in \Hom(\H^\bullet_{\rm{alg}}(X),\bf{C})$ to a family of stability conditions $\sigma_w^\tau$ in $\Stab(X)$, for $w$ in an open sector $\mathscr{S} \subseteq \bf{C}^*$ near the origin. Constructing this lift involves proving a strong existence result for stability conditions on $X$, a difficult open problem. 

    \smallskip
    
    \textbf{Step 3.} Having constructed $\sigma_w^\tau$ in $\Stab(X)$, one chooses a generic ray $\bf{R}_{>0} \cdot e^{\mathtt{i}\varphi} \subset \mathscr{S}$ and sets $w(t) = t e^{\mathtt{i}\varphi}$ for $t\in \bf{R}_{\ge 0}$. This gives a path $\sigma_t^\tau := \sigma_{w(t)}^\tau$ in $\Stab(X)$ defined on $(0,a)$ for some $a>0$. One then verifies that as $t\to 0$ the path $\sigma_t^\tau$ is quasi-convergent in the sense of \cite{HLJR}.\footnote{An alternative formulation is that the corresponding path in $\Stab(X)/\bf{C}$ converges to a boundary point of the space of augmented stability conditions $\Astab(X)$ as introduced in \cite{augmented}.} Then, the results \emph{ibid.} give rise to a (polarized) semiorthogonal decomposition $\DCoh(X) = \langle \cD_1,\ldots, \cD_n\rangle$ whose factors are generated by ``limit semistable'' objects $E$ such that $\Im(\log\cZ_t^{\tau}(E))$ satisfies certain asymptotic conditions.

    \smallskip
    
    \textbf{Step 4.} It is expected that different generic choices of canonical fundamental solution $\Phi_w^\tau$ for $\tau \in B$, sectors $\mathscr{S}$, and lifts $\sigma_w^{\tau}$ should result in mutation equivalent semiorthogonal decompositions. Thus, as ab\-stract categories, without a preferred choice of embedding, the factors $\cD_1,\ldots, \cD_n$ are intrinsically attached to $\DCoh(X)$. This is to be contrasted with the Jordan-H\"older property for semi\-orthogonal decompositions of derived categories, which is false in general \cites{JHproperty,JHpolarizable}.

\begin{figure} 
\begin{center}
\begin{tikzpicture}[scale=2, every node/.style={font=\small}]

  \def\aouter{1.5}
  \def\ainner{1.35}
  \def\b{1}
  \def\cx{1}
  \def\cut{0.75}

  \begin{scope}[xshift=0.1cm]

    \fill[gray!10] (-1.2,-1.2) rectangle (1,1.2);

    \draw[thin] (1,-1.2) -- (1,1.2);

    \fill[red!30!pink!40]
      plot[domain=-asin(\cut/\b):asin(\cut/\b), samples=200]
        ({\cx - \aouter*cos(\x)}, {\b*sin(\x)})
      -- (\cx,\cut) -- (\cx,-\cut) -- cycle;

    \fill[blue!60, opacity=0.4]
      plot[domain=-asin(\cut/\b):asin(\cut/\b), samples=200]
        ({\cx - \ainner*cos(\x)}, {\b*sin(\x)})
      -- plot[domain=asin(\cut/\b):-asin(\cut/\b), samples=200]
        ({\cx - \aouter*cos(\x)}, {\b*sin(\x)}) -- cycle;

    \draw[dashed, thin]
      ({\cx - \aouter*cos(asin(\cut/\b))},  \cut) -- (\cx,  \cut);
    \draw[dashed, thin]
      ({\cx - \aouter*cos(asin(\cut/\b))}, -\cut) -- (\cx, -\cut);
    \draw[dashed, thin]
      plot[domain=-asin(\cut/\b):asin(\cut/\b), samples=200]
        ({\cx - \aouter*cos(\x)}, {\b*sin(\x)});

    \draw[thick, black]
      plot [smooth, tension=1]
        coordinates {
          ({\cx - \ainner*cos(25)}, .5)
          (.05, -.05)
          (0.55, 0.1)
          (1, 0.15)
        };

    \node[black, scale=1, rotate=38] at (0.30,-.03) {>};

    \node[black, font=\normalsize] at (0.6,0.22) {$\sigma_t$};

  \end{scope}

  \begin{scope}[xshift=1.3cm, yshift=0.3cm]
    \fill[red!30!pink!40] (0,0) rectangle (0.2,0.15);
    \node[right] at (0.25,0.075) {non-geom. glued region};

    \fill[blue!60, opacity=0.4] (0,-0.25) rectangle (0.2,-0.1);
    \node[right] at (0.25,-0.175) {geometric glued region};
  \end{scope}

\end{tikzpicture}
\begin{caption}
    {The path of stability conditions constructed in this paper. $\sigma_t$ is defined for $t$ in an interval $(0,a)$. The arrow-head indicates the direction as $t\to 0$ and the horizontal line can be interpreted as boundary points in the partial compactification $\Astab(X)$ constructed in \cite{augmented}.}
\end{caption}
\end{center}
\label{F:pathdiagram}
\end{figure}

\medskip

The strategy outlined above was first considered by Halpern-Leistner \cite{NMMP} in the case of $\bf{P}^1$, where complete results were obtained. The subsequent work \cite{Vanjapaper} describes quasi-convergent paths in $\Stab(\bf{P}^n)$ beginning in the geometric region with central charge of the form $\cZ_t^{\tau}$ as above, with limit as $t\to 0$ giving rise to a mutation of the standard Beilinson collection \cite{Beilinson1978}. These ideas were also used profitably in \cite{karube2024noncommutative} in the case of surfaces blown-up in a point to construct quasi-convergent paths recovering the canonical blow-up semiorthogonal decomposition of Orlov \cite{Orlovmonoidal}. While the full picture of the NMMP is still emergent, the present work makes the first steps toward an understanding of the Fano case.

\subsubsection*{Overview} Next, we give a conceptual overview of the paper. As stated above, one of the main ideas of the NMMP is that there should be a connection between paths of stability conditions and solutions to the quantum differential equation of $X$. The idea that quantum cohomology and spaces of stability conditions should be related is not new, and has appeared in work of Bridgeland \cite{BridgelandNCCY3} and Iritani \cite{iritani_integral}.

The main objective of \Cref{S:nMMPFano} is to explain a direct relationship between solutions to the quantum differential equation for $\tau \in \H^\bullet(X)$, and paths in $\Stab(X)$. This is elucidated through our elaboration of \cite{NMMP}*{Proposal III} in the case of Fano varieties, which we now explain. For the first part of this discussion, we constrain ourselves to the small quantum cohomology locus, i.e. when $\tau \in \H^2(X)$.

The quantum connection $\nabla^\tau$ at $\tau \in \H^2(X)$ has a regular singularity at $w = \infty$ and an irregular singularity at $w = 0$. The key player in determining the behavior of solutions as $w\to 0$ is the Euler operator $\cE_\tau\star_\tau(-)$ in \eqref{E:qde}. Note that in the small quantum cohomology locus, $\cE_\tau = c_1(X)$ -- see \eqref{E:eulerfield}. Denote by $\sigma(\cE_\tau)$ the multi-set of eigenvalues of the Euler operator and by $\lvert \sigma(\cE_\tau)\rvert$ the underlying set. Work of Sanda--Shamoto \cite{SandaShamoto}, identifies a so-called \emph{A-model mutation system} which consists of a decomposition of vector spaces
\begin{equation}
\label{E:Amodelmutation}
    \H^\bullet(X) = \bigoplus_{\lambda \in \lvert \sigma(\cE_\tau)\rvert} \cA_\lambda
\end{equation}
plus additional linear algebraic data. By \cite{SandaShamoto}*{Lem. 3.5}, $\cA_\lambda$ can be characterized as the set of cohomology classes $\alpha$ such that $\lVert e^{\lambda/w}\Phi_w^\tau(\alpha)\rVert \le O(\lvert w\rvert^{-m})$ as $w\to 0$, for some $m\in \bf{Z}_{\ge 0}$, and for any choice of norm $\lVert \:\cdot\:\rVert$ on $\H^\bullet(X)$.\footnote{This condition is sometimes called being of \emph{moderate growth}.} In \cite{SandaShamoto}, it is also shown that \eqref{E:Amodelmutation} admits a categorical lift, in a suitable sense, for smooth Fano complete intersections. Our first conjecture in the present work is that a corresponding result holds for all smooth Fano varieties, and furthermore that such decompositions arise from quasi-convergent paths in $\Stab(X)$, as developed in \cite{HLJR}:

\begin{introconj}
[ = \Cref{conj:noncommutativeGamma}\ref{conj2SOD}, simplified] For any smooth Fano variety $X$, there exist $\tau \in \H^\bullet(X)$, a sector $\mathscr{S}\subset \bf{C}^*$, $\epsilon>0$, and a map $\mathscr{S} \cap \{z\in \bf{C}^*:\lvert z\rvert < \epsilon\} \to \Stab(X)$ such that for an open dense set of $\{\varphi \in \bf{R}:\bf{R}_{>0}\cdot e^{\mathtt{i}\varphi}\subset \mathscr{S}\}$, there is a quasi-convergent path $\sigma_{t,\varphi}^\tau := (\cZ_{te^{\mathtt{i}\varphi}}^\tau,\cP_{t,\varphi})$ in $\Stab(X)$ defined as $t\to 0$. In addition, the semiorthogonal decomposition induced by $\sigma_{t,\varphi}^\tau$ as $t\to 0$ is of the form 
\[
    \DCoh(X) = \langle \cD_\lambda:\lambda \in \lvert \sigma(\cE_\tau)\rvert\rangle.
\]
Here, the ordering on $\lvert \sigma(\cE_\tau)\rvert$ is $\lambda<\mu$ if $\Im(-e^{-\mathtt{i}\varphi}\mu)>\Im(-e^{-\mathtt{i}\varphi}\lambda)$.  Furthermore, limit semistable objects $E\in \cD_\lambda$ satisfy certain asymptotic estimates \eqref{E:LSSasymptotic}.
\label{Conj:Introconj1}
\end{introconj}

The quasi-convergent path $\sigma_{t,\varphi}^\tau$ is a lift of the path obtained from $\cZ_w^\tau$ by setting $w = te^{\mathtt{i}\varphi}$. The next part of the conjectures concerns the apparent dependence of \Cref{Conj:Introconj1} on the parameters $\varphi\in \bf{R}$ and $\tau \in \H^\bullet(X)$.

\begin{introconj}
[ = \Cref{conj:noncommutativeGamma}\ref{conj2independence}, simplified]
In the context of \Cref{Conj:Introconj1}, the quasi-convergent paths $\sigma_{t,\varphi}^\tau$ depend contiuously on $(\tau,\varphi)$ and the semiorthogonal decomp\-ositions of $\DCoh(X)$ obtained from deformations of $(\tau,\varphi)$ are related by mutation.
\label{Introconj:uniqueness}
\end{introconj}

Sanda--Shamoto \cite{SandaShamoto} show that in the Fano complete intersection case, the categorical lift of \eqref{E:Amodelmutation} comes from a mutation of the Kuznetsov decomposition $\DCoh(X) = \langle \Ku(X),\cO_X,\ldots, \cO_X(n-d)\rangle$, where $X \subset \bf{P}^n$ is cut out by equations with total degree $d$. Similarly, we expect that the canonical decomposition of $X$ predicted by the NMMP should be the Kuznetsov decomposition, up to mutation -- see \Cref{conj:compl_inter}. 

The final part of the conjectures deals with the existence of \emph{geometric} stability conditions, which are stability conditions $\sigma \in \Stab(X)$ with respect to which all structure sheaves of (closed) points of $X$ are stable of the same phase. The moduli spaces $\cM_\sigma(v)$ constructed using stability conditions in the geometric region typically have geometry closely related to that of the variety $X$ itself; for example, taking $v$ to be the Chern character of the structure sheaf of a point, one recovers $X$ as a moduli space of Bridgeland semistable objects. Variation of $\sigma$ can produce interesting birational transformations of the spaces $\cM_\sigma(v)$, which are related to the minimal model program of $X$ -- cf. \cites{arcara2013minimal,XiaoleietalNefHilbschemes,Bayer2013MMPFM,TodaSurface}. 

From the perspective of homological mirror symmetry, existence of geometric stability conditions is an enticing question, since it suggests ways to intrinsically construct from the data of a (pre-)triangulated (dg-)category $\cD$ a variety $X$ and an exact equivalence $\DCoh(X) \simeq \cD$. The space of stability conditions of $\DCoh(\bf{P}^1)$ has been extensively studied by Okada \cite{Ok06} and Halpern-Leistner \cite{NMMP}*{\S 3}. In the latter work, it is shown that the quantum cohomology central charge $\cZ_t$ lifts to a quasi-convergent path $\sigma_t$ in $\Stab(X)$ for $t\in \bf{R}_{>0}$. When one sends $t\to\infty$, i.e. toward the regular singularity of $\nabla$, $\sigma_t$ travels from the glued region \cite{Collins_Polischuk_2010} associated to the Beilinson collection $\langle \cO,\cO(1)\rangle$ to the geometric region. This suggests a mechanism for finding geometric stability conditions on $\DCoh(X)$, starting from the more easily constructed glued regions.

\medskip

It was observed in \cite{Vanjapaper} that the glued regions associated to certain full exceptional collections on $\DCoh(\bf{P}^n)$ contain geometric stability conditions for all $n\ge 1$. However, this property does not seem to be invariant under mutation. The heuristics of Dubrovin's conjecture \cite{Dubrovin1998} and its reformulation as the Gamma conjectures \cite{GGI16} suggest that in order to construct canonical paths $\sigma_w$ in $\Stab(X)$ lifting solutions of the QDE, one should study solutions of the differential equations $\nabla_{w\partial_w}^\tau = 0$, where $\tau$ is allowed to vary in a region $\mathscr{B}\subset \H^\bullet(X)$. 

In fact, one should consider certain \emph{isomonodromic deformations} of the quantum connection $\nabla$ (\Cref{D:isomonodromic_deformation}). Roughly, these are extensions of $\nabla$ to a flat connection $\widetilde{\nabla}$ over the trivial $\H^\bullet(X)$-bundle over a space $M\times \bf{P}^1$, such that for any fixed $x\in M$ there is a regular singularity at $\infty$ and an irregular singularity at $0$. Here, there is a fixed holomorphic embedding $\mathscr{B}\hookrightarrow M$. The deformation is isomonodromic in that the monodromy data at $\infty$ and the Stokes data at $0$ are constant as $x\in M$ varies.

When $X$ has an open set $\mathscr{B}\subset \H^\bullet(X)$ of points near $\tau = 0$ where the quantum product $\star_\tau$ converges and is semisimple, there is a canonical isomonodromic deformation of the quantum connection to $\widetilde{\mathscr{U}}_N\times \bf{P}^1$ \cite{dubrovin1998painleve}. Here, $\mathscr{U}_N$ is the configuration space of $N$ distinct and unordered points in $\bf{C}$. In \Cref{SS:isomonodromicdeformation}, we isolate some important properties of this isomonodromic deformation and observe that the canonical fundamental solution $\Phi_w$ of \cite{GGI16} can be extended to a fundamental solution $\Phi_w^u$, where $u\in \widetilde{\mathscr{U}}_N$. In particular, fixing $u$, we have a canonical fundamental solution of $\widetilde{\nabla}_{w\partial_w}^u = 0$. Using this, we define quantum cohomology central charges 
\[
    \cZ_w^u(-) = (2\pi w)^{\dim X/2} \int_X \Phi_w^u(-)
\]
depending on $u\in \widetilde{\mathscr{U}}_N$. This allows us to state:

\begin{introconj}
[ = \Cref{conj:noncommutativeGamma}\ref{conj2isomonodromic}] For any smooth Fano variety $X$, there is an isomonodromic deformation $(\nabla^u)_{u\in U}$ of the quantum connection, a sector $\mathscr{S}\subset \bf{C}^*$, constants $\epsilon,\rho>0$, and a holomorphic map 
\[
    U\times (\mathscr{S}\cap \{w:\lvert w\rvert<\rho+\epsilon\})\to \Stab(X),\;\; w\mapsto \sigma^u_w,
\] 
where $\sigma_w^u$ has central charge $\cZ_w^u$, such that \vspace{-2mm}
\begin{enumerate}[label=(\alph*)]
    \item $\sigma_w^u$ is quasi-convergent as $w\to 0$ along any ray-segment in $\mathscr{S} \cap \{w:\lvert w\rvert <\rho+\epsilon\}$;\vspace{-2mm}
    \item the semiorthogonal decompositions obtained from $\sigma_w^u$ and a choice of ray segment are all mutation equivalent; and \vspace{-2mm}
    \item $\sigma_w^u$ is geometric for all $w$ with $\rho-\epsilon<\lvert w\rvert <\rho+\epsilon$.\vspace{-2mm}
\end{enumerate}
\label{Introconj:isomonodromic}
\end{introconj}

\subsubsection*{Results} Much of the paper is dedicated to verifying these conjectures, rephrased as \Cref{conj:noncommutativeGamma}, in several cases. We consider the semisimple cases of projective quadrics $Q\subset \bf{P}^n$ and Grassmannian varieties $\Gr(k,n)$, as well as the non-semisimple cases of cubic threefolds and fourfolds.  We include a thorough discussion in \Cref{SS:conjectures} of how the conjectures in the present work are related to the NMMP conjectures \cite{NMMP}. In particular, we verify \cite{NMMP}*{Proposal III} for all smooth quadrics and Grassmannians -- see \Cref{C:originalNMMPholds}.

Much of the technical work in the paper, contained in \Cref{S:constructionofstab}, involves the construction of geometric stability conditions in several new cases. This involves a careful analysis of the gluing construction of Collins--Polishchuk \cite{Collins_Polischuk_2010} for full exceptional collections arising from resolutions of the diagonal \cites{Beilinson1978,Kapranov85,Kapranov1988}. Accordingly, our first theorem is:

\begin{introthm}
( = part of \Cref{T:geomexistence})
\label{T:introgeomstab}
    If $X$ is a smooth quadric hypersurface in $\bf{P}^n$ or any finite product of Grass\-mannian varieties $\Gr(k,n)$, then $X$ admits geometric stability conditions.
\end{introthm}

To our knowledge, \Cref{T:introgeomstab}, which is of independent interest, gives the first examples of higher dimensional varieties $X$ besides $\bf{P}^n$ admitting geometric stability conditions with central charge factoring through $\H^{\bullet}_{\rm{alg}}(X)$. Using symmetries of the stability conditions constructed in \Cref{T:introgeomstab}, we apply the induction procedure of \cite{macri2009inducing} and its refinement in \cite{DellHengLicata} to construct \emph{almost} geometric stability conditions (\Cref{D:almostgeometric}) in some other examples:

\begin{introthm}
[ = rest of \Cref{T:geomexistence}]
    For all $n\ge 1$, the Hilbert schemes $\Hilb^{n}(\bf{P}^2)$, $\Hilb^n(\bf{P}^1\times \bf{P}^1)$, and the weighted projective stacks $\bf{P}(a_0,\ldots, a_n)$ with $\gcd(a_0,\ldots, a_n) = 1$ admit almost geometric stability conditions. (See \Cref{P:Hilbnstab} and \Cref{C:weightedprojectivestab} for more precise statements of these results.)
    \label{intro:TB}
\end{introthm}

The later parts of \Cref{S:constructionofstab} are extend the techniques of \Cref{SS:geomFEC} to cubic hypersurfaces in $\bf{P}^n$. This is achieved by establishing technical results on gluing stability conditions in the presence of a ``Kuznetsov-type'' semiorthogonal decomposition of $\DCoh(X)$.

\begin{introthm}
[ = \Cref{thm_glued_geom_stab_cubic3} + \Cref{thm_geomstab_cubic4_noplane}]\label{intro:TC}
    If $X$ is a cubic threefold or a cubic fourfold not containing a plane, then $\DCoh(X)$ admits geometric stability conditions with central charge factoring through the canonical morphism $\rm{K}_0(X) \to \rm{K}_0^{\rm{top}}(X)$.
\label{T:cubicthreeorfour}
\end{introthm}

The threefold case of \Cref{T:cubicthreeorfour} has already been obtained in \cite{Bernardara2016BridgelandSC}. However, there, the result is obtained by constructing stability conditions on a heart obtained from $\Coh(X)$ by tilting and proving a suitable threefold Bogomolov-Gieseker (BG) inequality -- see \cite{Bayer2011BridgelandSC}. This strategy is the main one used in the literature to construct geometric stability conditions; unfortunately, it seems that proving the necessary BG inequalities becomes increasingly difficult as the dimension of $X$ increases. By contrast, the construction of geometric stability conditions in the present work is independent of the BG inequality. 

The fourfold case of \Cref{T:cubicthreeorfour} seems to be completely new. In \Cref{S:examples}, we use these new geometric stability conditions to verify the conjectures in some cases: 

\begin{introthm}
[ = \Cref{C:NMMPFanoexamples} + \Cref{T:conj2.1threefold} + \Cref{T:conj2.1fourfold}]
    \Cref{Conj:Introconj1} holds when $X$ is a smooth quadric, a Grassmannian $\Gr(k,n)$, a smooth cubic threefold, or a smooth cubic fourfold not containing a plane.
\label{T:conjecture2cases}
\end{introthm}

Finally, we have: 

\begin{introthm}
[ = \Cref{T:NCgammaforGrandQ}] Conjectures \ref{Conj:Introconj1}, \ref{intro:TB}, and \ref{intro:TC} hold for smooth projective quadrics and Grass\-mannians $\Gr(k,n)$. 
\label{T:introconj2fullcases}
\end{introthm}

We are not able to prove Conjectures \ref{intro:TB} and \ref{intro:TC} for the cubic hypersurfaces considered in the present work. For instance, the theory of isomonodromic deformations of the quantum connection in the non-semisimple situation be sufficiently developed to attack \Cref{Introconj:isomonodromic}. Nevertheless, in \Cref{SS:Epilogue} we make some speculations about what might be expected.

\subsubsection*{Related work}

Since the inception of this project, there have been several related developments. As mentioned above, the recent work of Katzarkov--Kontsevich--Pantev--Yu \cite{KKPY} has established the theory of atoms, which have proven to be fine enough birational invariants of varieties to resolve long-standing rationality questions in the birational geometry of hypersurfaces. Some of the ideas present in \cite{KKPY} have been expounded upon in lectures by its authors over the last years, which informed the formulation of the NMMP \cite{NMMP} and thus the present work. 

Atoms are birational invariants of varieties that are constructed at the cohomological level, from decompositions of certain non-Archimedean bundles with flat connection and fiber $\H^\bullet(X)$, called A-model F-bundles \cite{Fbundles}. The NMMP can be regarded as an attempt to lift these decompositions to the categorical level, i.e. to semiorthogonal decompositions of $\DCoh(X)$. We expect that applying a suitable additive invariant of dg-categories valued in vector spaces should allow one to recover the decomposition of $\H^\bullet(X)$ obtained from the A-model F-bundle. For example, applying Blanc's topological K-theory \cite{BlancTopK} functor to the decompositions obtained in the present work should recover A-model F-bundle decomposition of $\H^\bullet(X)$.

While the categorical decompositions predicted by \cite{NMMP} would allow the construction of finer \emph{categorical} invariants of $X$, the price is that the theory seems to depend on difficult constructions of stability conditions in higher dimensions. We have made some first steps in this direction in the present work. However, even in the relatively simple cases considered here, proving canonicity of the decompositions obtained from the quasi-convergent paths, i.e. the global version of \Cref{Introconj:uniqueness}, necessitates a better global understanding of $\Stab(X)$. 

In the present work, we treat only the case of smooth Fano varieties so that small quantum cohomology is convergent. On the other hand, the other works treat more general varieties, and therefore need to address convergence issues. In \cite{NMMP}, a polynomial truncation of the quantum differential equation is proposed, which circumvents these convergence issues. On the other hand, in \cite{KKPY} the authors employ techniques of non-Archimedean analysis to obtain convergence. 

Finally, we also mention the recent work of Elagin--Schneider--Shinder \cite{SurfaceAtomic} which constructs canonical semiorthogonal decompositions of $\DCoh(X)$ in the case where $X$ is a smooth projective surface. These decompositions are furthermore compatible with standard operations such as blow-ups and formation of projective bundles. It should be investigated whether one can reproduce a version of the results of \cite{SurfaceAtomic} using the techniques of the NMMP outlined above.

\section*{Acknowledgements}

It is our pleasure to thank Arend Bayer, Hannah Dell, Daniel Halpern-Leistner, Wahei Hara, James Hotchkiss, Emanuele Macr\`{i}, Fumihiko Sanda, and Yukinobu Toda for many useful conver\-sations about the material in this paper. We are especially grateful to Emanuele Macr\`{i} for his comments on an early version of the paper.

T.K. was supported by JSPS KAKENHI Grant Number 24KJ0713. A.R. was supported by NSF grant DMS-2503404. V.Z. was supported by ERC Synergy Grant 854361 HyperK.

A.R. also thanks Maria Teresa for her unwavering support during the preparation of this work.

\newpage

\section*{Notation and conventions}
\label{S:notationandconventions}

We gather here some notation used throughout the paper:

\begin{center}
\begin{tabular}{@{}ll@{}}
$X$ & smooth complex projective variety, usually Fano \\
$\DCoh(X)$ & bounded derived category of coherent sheaves on $X$ \\
$\H^\bullet(X)$ & singular cohomology of $X$ with complex coefficients $\H^\bullet(X,\bf{C})$\\
$B$ & locus of convergence of the quantum product in $\H^\bullet(X)$ containing $0$\\
$\mathscr{B}$ & $B \setminus \{\tau \in B: \text{the Euler operator has repeated eigenvalues}\}$ \\
$\Stab(X)$ & space of stability conditions on $\DCoh(X)$ -- see \Cref{SS:stabilityconditionsreview} \\
$\cZ_w^\tau(-)$ & quantum cohomology central charge at $\tau \in B\subset \H^\bullet(X)$\\
$\cD$ & $k$-linear triangulated category \\
$\mathcal{H}^i_{\mathcal{A}}(-)$ & $i^{\rm{th}}$ cohomology object functor with respect to a heart $\cA$\\
$\mathcal{H}^i(-)$ & $\cH^i_{\Coh(X)}(-)$ on $\DCoh(X)$\\
$\bf{R}_E(-)$ & right mutation at exceptional object $E$, $\Cone( - \rightarrow \RHom(-,E)^{\vee}\otimes E)[-1]$\\
$\bf{L}_{E}(-)$ & left mutation at exceptional object $E$, $\Cone(\RHom(E,-)\otimes E \rightarrow -)$ \\
$\H^i(-)$ & $i^{\rm{th}}$ sheaf cohomology functor on $X$ \\
$\ch^\beta(E)$ & $e^{-\beta H}\ch(E)$ for $\beta \in \H^\bullet(X)$ and $H$ the hyperplane class \\
$\Ch(-)$ & $(2\pi \mathtt{i})^{\deg/2} \ch(-)$\\
$\cI_{X/Y}$ & ideal sheaf of a closed subvariety $X$ of another variety $Y$\\ 
$\cI_{x,Y}$& $\cI_{x/Y}$, where $x$ is a closed point in $Y$\\
$\mathfrak{S}_n$ & symmetric group on $n$ elements\\
$\mathfrak{B}_n$ & braid group on $n$ strands\\
$\mathscr{S}(\varphi,\epsilon)$ & $\{w\in \bf{C}^*: \arg(w) \in (\varphi-\epsilon,\varphi+\epsilon)\}$, for $\varphi \in \bf{R}$ and $\epsilon>0$\\
$\mathscr{S}$ & an angular sector in $\bf{C}^*$, i.e. $\mathscr{S}(\varphi,\varepsilon)$ for some $\varphi,\epsilon$\\
$f(t)\sim g(t)$ & $\lim_{t\to 0} \log f(t) - \log g(t) = 0$ (or $t\to\infty$, depending on context)\\
$f(t)\approx g(t)$ & $\lim_{t\to 0}f(t)-g(t) = 0$ (or $t\to\infty$, depending on context)\\
$\GL_2^+(\bf{R})^\sim$ & universal cover of $\GL_2^+(\bf{R})$ -- see \cite{Br07}*{Lem. 8.2} for its action on $\Stab(X)$ \\ 
$\mathscr{U}_N$ & configuration space of $N$ unlabelled points in $\bf{C}$ 
\end{tabular}
\end{center}
\medskip

A homomorphism $A\to B$ of Abelian groups is called a \emph{rational surjection} if the induced map $A \otimes_\bf{Z} \bf{Q} \to B\otimes_{\bf{Z}} \bf{Q}$ is a surjection.

We say that $Z\subset \bf{C}$ is in \emph{general position} if for all $x\ne y\in Z$, one has $\Re(x) \ne \Re(y)$ and $\Im(x) \ne \Im(y)$.

\section{Construction of some geometric stability conditions}
\label{S:constructionofstab}
In this section we construct geometric stability conditions using the gluing construction of Collins--Polishchuk \cite{Collins_Polischuk_2010}, summarized as \Cref{thm:glued}. When there is a special resolution of the diagonal $\cO_\Delta$ sheaf on $X\times X$ by sums of exceptional sheaves as in \cites{Kapranov85,Kapranov1988}, \Cref{T:geomexistence} implies existence of glued geometric stability conditions. When the derived category admits a Kuznetsov-type decomposition, we can generalize this technique to construct geometric stability conditions in some cases, including generic cubic fourfolds -- see \Cref{thm_geomstab_cubic4_noplane}.

\subsection{Bridgeland stability conditions}
\label{SS:stabilityconditionsreview}
We briefly recall the definition of Bridgeland stability conditions \cites{Br07} and the gluing technique of \cite{Collins_Polischuk_2010}, which plays an essential role in this paper. Throughout, $\cD$ is a $k$-linear triangulated category. 

\begin{defn}
\label{D:prestability}
    A \emph{slicing} $\cP$ on $\cD$ is a collection of full additive subcategories $\{\cP(\phi)\}_{\phi \in \bf{R}}$ of $\cD$ such that \vspace{-2mm}
    \begin{enumerate}
        \item $\phi_1>\phi_2\Rightarrow \Hom_{\cD}(\cP(\phi_1),\cP(\phi_2)) = 0$ \vspace{-2mm}
        \item $\cP(\phi)[1] = \cP(\phi+1)$ for all $\phi \in \bf{R}$, and \vspace{-2mm}
        \item for every non-zero object $E$ of $\cD$ there exists a sequence of real numbers $\phi_1>\cdots>\phi_n$ and a sequence of morphisms $0 = E_0 \to E_1\to \cdots \to E_n = E$ such that $\Cone(E_{i-1}\to E_i) \in \cP(\phi_i)$ for each $i=1,\ldots, n$.\vspace{-2mm}
    \end{enumerate}
    A \emph{pre-stability condition} on $\cD$ is a pair $(Z,\cP)$ where $\cP$ is a slicing and $Z\in \Hom_{\bf{Z}}(\rm{K}_0(\cD),\bf{C})$ is called the \emph{central charge} such that for all $\phi \in \bf{R}$ and all non-zero $E\in \cP(\phi)$ we have $Z(E) \in \bf{R}_{>0}\cdot \exp(i\pi \phi)$. Such an object $E$ is called \emph{semistable} of phase $\phi$, and $\lvert Z(E)\rvert =:m(E)$ is its \emph{mass}.
\end{defn}

The sequence of maps in \Cref{D:prestability}(3) is called a \emph{Harder-Narsimhan filtration}. It is a standard fact that a prestability condition $\sigma$ on $\cD$ is equivalent to specifying a heart $\cA$ of a bounded t-structure on $\cD$ and a ``stability function'' $Z:\rm{K}_0(\cA) \to \bf{C}$ satisfying the Harder-Narasimhan property \cite{Br07}*{Prop. 5.3}. Thus, sometime we denote a (pre-)stability condition by $\sigma = (Z,\cA)$, where $Z$ is the central charge and $\cA = \cP_\sigma(0,1]$, i.e. the extension closure in $\cD$ of $\bigcup_{\phi \in (0,1]}\cP(\phi)$, is the associated heart. 

In recent years, it has become common practice to consider the following strengthening of Bridgeland's original notion, as proposed by Kontsevich-Soibelman \cite{KS:08}. By a slight abuse of notation, we write $\cP = \bigcup_{\phi \in \bf{R}} \cP(\phi)$. 

\begin{defn}
    We fix once and for all a finitely generated Abelian group $\Lambda$ of positive rank and a homomorphism $v:\rm{K}_0(\cD) \to \Lambda$ which is a surjection after tensoring with $\bf{Q}$.\footnote{In the sequel, we call such a $v$ a rational surjection.} Choose any norm $\lVert \:\cdot\:\rVert$ on $\Lambda_{\bf{R}}$. We say that a pre-stability condition $\sigma = (Z,\cP)$ is a \emph{stability condition} if it satisfies the \emph{support property} with respect to $v$: 
    \[
        \inf_{0\ne E \in \cP} \frac{\lvert Z(E)\rvert}{\lVert v(E)\rVert} > 0.
    \]
    The set of stability conditions satisfying the support property with respect to $v$ is denoted $\Stab_\Lambda(\cD)$. We usually omit $\Lambda$ from the notation, but it is considered implicit.
\end{defn}

Consider a pre-stability condition $\sigma = (Z,\cP)$ on $\cD$ and a non-zero object $E$. By \Cref{D:prestability}, $E$ has a Harder-Narasimhan filtration $E_0\to E_1\to \cdots \to E_n = E$ with Harder-Narasimhan factors $A_1,\ldots, A_n$ defined by $A_i = \Cone(E_{i-1}\to E_i) \in \cP(\phi_i)$. It can be shown that the $A_i$ are unique up to isomorphism. Consequently, we can define the maximal phase of $E$ as $\phi_\sigma^+(E):=\phi_1$, the minimal phase of $E$ as $\phi_\sigma^-(E) := \phi_n$, and the \emph{mass} of $E$ as $m_\sigma(E) = \sum_i \lvert Z(A_i)\rvert$.

It is a non-trivial fact that the space of pre-stability conditions can be given a topology induced by a generalized metric \cite{Br07}*{\S 6}. The following result, sometimes called Bridgeland's deformation theorem, is the main result about stability conditions. It was proven originally in \cite{Br07} subject to slightly different hypotheses. The following version is proven in \cite{Bayershort}*{Thm. 1.2}. 

\begin{thm}
    The space $\Stab_{\Lambda}(\cD)$ has a unique complex structure such that the map
    \[
        \Stab_{\Lambda}(\cD) \to \Hom_{\bf{Z}}(\Lambda, \bf{C}), \quad \sigma=(Z_{\sigma},\cP_{\sigma}) \mapsto Z_{\sigma}
    \]
    is a local biholomorphism.
\end{thm}

In particular, deformations of a stability condition are controlled by the deformations of its central charge, which is an element of a finite dimensional complex vector space.

\begin{ex}
\label{E:latticeexample}
    We explain the choice of map $v:\rm{K}_0(\cD) \to \Lambda$ to be used in what follows in the case of $\cD = \DCoh(X)$. We regard $\DCoh(X)$ as a pre-triangulated dg-category with its canonical enhancement. For a $\bf{C}$-linear dg-category $\cD$, Blanc \cite{BlancTopK} constructs a topological K-theory spectrum $\bf{K}^{\rm{top}}(\cD)$ and a morphism $\bf{K}(\cD) \to \bf{K}^{\rm{top}}(\cD)$ of spectra, where $\bf{K}(\cD)$ is the algebraic K-theory spectrum introduced in \cite{Schlichting06}. When $\cD = \DCoh(X)$, taking $\pi_0$ of this map of spectra recovers the Chern character 
    \begin{equation}
    \label{E:chernchar}
        \ch:\rm{K}_0(\cD) \to \rm{K}_0^{\rm{top}}(X).
    \end{equation}
    Tensoring with $\bf{Q}$, we have an isomorphism $\rm{K}_0^{\rm{top}}(X)_{\bf{Q}} \cong \bigoplus_{i=0}^n \H^{2i}(X,\bf{Q})$, where $n = \dim X$, and we take our lattice $\Lambda$ to be the image of the map \eqref{E:chernchar}, which is identified with the algebraic cohomology $\H^{\bullet}_{\rm{alg}}(X)$ of $X$.

    In practice, our central charges will be defined to depend on a composite of $\ch$ with a $\bf{C}$-linear automorphism of $\H^{\bullet}(X,\bf{C})$. Examples of this are $\Ch(-) := \sum_j(2\pi \mathtt{i})^j\ch_j(-)$ or the Mukai vector $v(-):=\sqrt{\rm{td}(X)}\cdot \ch(-)$. It is important in our approach to use a lattice $\Lambda$ with a canonical embedding in $\H^{\bullet}(X,\bf{C})$ because this is where the quantum differential equation is defined.
\end{ex}

It is common practice in the literature on stability conditions to consider stability conditions on $\DCoh(X)$ that are \emph{numerical} in that their charges factor through the numerical Grothendieck group of $X$, denoted $\cN(X)$. There is a canonical map $\rm{K}_0(X) \twoheadrightarrow \cN(X)$; however, as remarked in \cite{BridgelandDb(intro)}*{p. 8} it is a difficult problem to construct a homomorphism $\cN(X) \to \H^{\bullet}(X,\bf{Q})$ compatible with $\ch:\rm{K}_0(X) \to \H^{\bullet}(X,\bf{Q})$. Since it is crucial for us that our stability conditions are \emph{topological}, in that they factor through topological K-theory of $\cD$, we explain the comparison in cases of interest to us.

\begin{ex}
    If $\DCoh(X)$ admits a full exceptional collection $\cE = \{E_1,\ldots, E_n\}$, then $\rm{K}_0(X) \cong \bigoplus_{i=1}^n \bf{Z}\cdot E_i$ and $\cN(X) = \rm{K}_0(X)$, since the kernel of the Euler pairing is trivial in this case. Further, $\ch$ induces an isomorphism $\rm{K}_0(X)_{\bf{Q}}\to \H^{\bullet}(X,\bf{Q})$, so that numerical stability conditions coincide with the topological ones. 
\end{ex}

\begin{ex}
    Next, we consider cubics.
    \begin{enumerate}
        \item Let $X$ be a cubic threefold. In this case, $\cN(X)$ is freely generated by the classes of $\cO_X,\cO_H,\cO_\ell, \cO_p$, where $H$ is a hyperplane section of $X$, $\ell$ is a line on $X$, and $p$ is a point -- see \cite{bernardara2012categorical}*{Prop. 2.7}. Consequently, there is an induced map $\ch:\cN(X)\to \H^{\bullet}(X,\bf{Q})$ which induces an isomorphism onto the lattice of algebraic classes. It follows that numerical and topological stability conditions in these cases are equivalent. Furthermore, \cite{bernardara2012categorical}*{Lem. 2.6} gives a decomposition $\cN(X) = \cN(\Ku(X)) \oplus \bf{Z}\cdot \cO_X \oplus \bf{Z}\cdot \cO_X(1)$, where $\cN(\Ku(X))$ is characterized as the left orthogonal complement to $\cO_X$ and $\cO_X(1)$ with respect to the Euler pairing; thus, numerical and topological stability conditions on $\Ku(X)$ are also equivalent.
        \item The case where $X$ is a cubic fourfold is more complicated, but has been treated in \cite{AddingtonThomas}. Indeed, on p. 1891 of \emph{loc. cit.} it is explained that $\cN(\Ku(X))$ can be identified with the image of the canonical map $\rm{K}_0^{\rm{top}}(\Ku(X)) \to \rm{K}_0(\Ku(X))$. It again follows that numerical and topological stability conditions coincide, both for $\Ku(X)$ and $\DCoh(X)$.
    \end{enumerate}
\end{ex}

\begin{rem}
    It does not seem easy to verify that $\cN(X)$ admits a direct sum decomposition compatible with any given semiorthogonal decomposition $\DCoh(X) = \langle \cD_1,\ldots, \cD_n\rangle$ when $X$ is more general. Indeed, it seems that in the relatively simple cases considered here, the reason that $\cN(X)$ splits is that it coincides with the image of the natural map $\rm{K}_0^{\rm{top}}(X) \to \rm{K}_0(X)$, and additivity of topological K-theory. In this sense, topological K-theory is a more suitable choice for studying the relationship between stability conditions and semiorthogonal decompositions.
\end{rem}

\subsubsection*{Gluing stability conditions}
Next, we recall the notion of gluing stability condition. Gluing for a semiorthogonal decomposition with two components was introduced in \cite{Collins_Polischuk_2010}. The more general case is discussed in \cite{HLJR}*{\S 3}.

\begin{lem}{\cite{Collins_Polischuk_2010}*{Lem. 2.1}}
Let $\cD = \langle \cD_1 ,\cD_2\rangle$ be a semiorthogonal decomposition
and let $\cA_i$ be the heart of a bounded t-structure on $\cD_i$ for $i=1,2$. If $\Hom^{\leq 0}(\cA_1,\cA_2) = 0$, then there exists a t-structure on $\cD$ with heart
\[
  \cA_1 \circ \cA_2 \coloneq \left\{E \in \cD \mid \pr_1(E) \in \cA_1,\pr_2(E) \in \cA_2 \right\},
\]
where $\pr_i$ is the projection functor $\cD \rightarrow \cD_i$.
\end{lem}

As mentioned in the previous section, we fix a rational surjection $v: \rm{K}_0(\cD) \to \Lambda$ to a finitely generated Abelian group of positive rank. In the presence of a semiorthogonal decomposition $\cD = \langle \cD_1,\ldots,\cD_n\rangle$, we further assume that there is a splitting $\Lambda = \bigoplus_{i=1}^n \Lambda_i$ such that $v$ restricts to a rational surjection $v_i:\rm{K}_0(\cD_i)\to \Lambda_i$ for each $i=1,\ldots, n$.

\begin{defn}{\cite{Collins_Polischuk_2010}}
    Consider a semiorthogonal decomposition $\cD = \langle \cD_1,\cD_2\rangle$. A stability condition $\sigma = (Z,\cA) \in \Stab(\cD)$ is \emph{glued} from $\sigma_1 = (Z_1,\cA_1) \in \Stab(\cD_1)$ and $\sigma_2 =(Z_2,\cA_2)  \in \Stab(\cD_2)$
    if: \vspace{-2mm}
    \begin{enumerate}
        \item $\Hom^{\leq 0}(\cA_1,\cA_2) = 0$, \vspace{-2mm}
        \item the heart $\cA = \cA_1\circ\cA_2$, and  \vspace{-2mm}
        \item $Z = Z_1\oplus Z_2$. \vspace{-2mm}
    \end{enumerate}
    We abbreviate this by $\sigma=\sigma_1*\sigma_2$.
\end{defn}

For $\theta \in [0,1]$, we let $\mathbf{H}_{\theta} = \left\{ r \cdot \exp(i \pi \phi) : r \in \bf{R}_{>0}, \phi \in [\theta, 1]\right\}$.  

\begin{thm}
\label{thm:glued}\label{prop:support}
    Suppose given a semiorthogonal decomposition $\cD = \langle \cD_1,\cD_2\rangle$ and $\sigma_i = (Z_i,\cP_i) \in \Stab(\cD_i)$ for $i=1,2$. 
    Assume $\Hom^{\leq 0}(\cP_1(0,1],\cP_2(0,1]) = 0$.
    If there exist $a\in (0,1)$ such that $\Hom^{\le 0}(\cP_1(a,a+1],\cP_2(a,a+1]) = 0$ and \vspace{-2mm}
    \begin{enumerate}
        \item $\theta \in (0,1]$ such that $Z_2(\cA_2)\subset \bf{H}_\theta$; or \vspace{-2mm}
        \item $\theta \in (0,1)$ such that $Z_1(\cA_1)\subset \bf{H}\setminus \bf{H}_\theta$\vspace{-2mm}
    \end{enumerate}
    then there exists $\sigma \in \Stab(\cD)$ glued from $\sigma_1$ and $\sigma_2$.
\end{thm}

\begin{proof}
    The result is a combination of \cite{Collins_Polischuk_2010}*{Thm. 3.6} with \cite{karube2024noncommutative}*{Props. 3.11, 3.12}. 
\end{proof}

\begin{cor}
\label{C:gluingconditions}
    In the notation of \Cref{thm:glued}, if $\Hom^{\le 0}(\cA_1,\cA_2) = 0$ and \vspace{-2mm}
    \begin{enumerate}
        \item $\cA_2$ is generated by finitely many simple objects; or\vspace{-2mm}
        \item $\cA_1$ is generated by finitely many simple objects of phase not equal to one,\vspace{-2mm}
    \end{enumerate}
    then there exists $\sigma \in \Stab(\cD)$ glued from $\sigma_1$ and $\sigma_2$.
\end{cor}

\begin{proof}
    If (1) holds, then so does \Cref{thm:glued}(1). Thus, it suffices to show that there exists $a\in (0,1)$ such that $\Hom^{\le 0}(\cP_1(a,a+1],\cP_2(a,a+1]) = 0$. We can choose $a \in (0,1)$ such that $\cP_2(a,a+1] = \cA_2$. Then, $\cP_1(a,a+1] \subset \langle \cA_1,\cA_1[1]\rangle_{\rm{ext}}$ and thus $\Hom^{\le 0}(\cP_1(a,a+1],\cP_2(a,a+1]) = 0$.

    Dually, if (2) holds, the argument is the same except that now we note that we can choose $a\in (0,1)$ sufficiently close to $1$ such that $\cP_1(a,a+1] = \cA_1[1]$. The result now follows, since $\cP_2(a,a+1] \subset \langle \cA_2,\cA_2[1]\rangle_{\rm{ext}}$.
\end{proof}

\begin{rem}
    The hypotheses (1) and (2) of \Cref{thm:glued} are necessary in the case where $\Hom^{\le 0}(\mathcal{A}_1,\mathcal{A}_2) = 0$ but $\Hom^1(\mathcal{A}_1,\mathcal{A}_2) \neq 0$. Indeed, if $\Hom^{\leq 1}(\mathcal{A}_1,\mathcal{A}_2) = 0$, then every semistable object $E$ in $\cA$ is a sum of semistable objects from $\cA_1$ and $\cA_2$, hence gluing holds.
\end{rem}

\subsection{Geometric stability from full exceptional collections}
\label{SS:geomFEC}
In this sec\-tion, we give a pro\-cedure for producing geometric stability conditions on $\DCoh(X)$ for a smooth projective variety $X$, when it admits a full exceptional collection of sheaves $\cE = \{E_1,\ldots, E_n\}$ satisfying certain special properties. 

\begin{defn}
\label{D:gradedFECnew}
    A \emph{grading} of an exceptional collection $\cE = \{E_1,\ldots, E_n\}$ is a total preorder $\preceq$ on $\cE$ such that $i < j$ implies that $E_i\preceq E_j$ and $E_i \sim E_j$\footnote{That is, $E_i \preceq E_j$ and $E_j\preceq E_i$.} implies that $E_i$ and $E_j$ are orthogonal or equal. 
\end{defn}

Recall that two objects $E$ and $F$ in a triangulated category $\cD$ are orthogonal if $\Hom_{\cD}(E,F[i]) = \Hom_{\cD}(F,E[i]) = 0$ for all $i\in \bf{Z}$. The relation $\sim$ is an equivalence relation and we call $B(\cE) = \cE/{\sim}$ the set of \emph{blocks} of $(\cE,\preceq)$. The total preorder $\preceq$ induces a total order on $B(\cE)$.

\begin{ex}
\label{Ex:basicexamples}
    We collect several basic examples. \vspace{-2mm}
    \begin{enumerate}
        \item Every exceptional collection has the trivial grading given by defining $E_i \preceq E_j$ if and only if $i\le j$. In the sequel, we will consider the full exceptional collection $\langle \Omega^n(n),\ldots, \Omega^1(1),\cO\rangle$ on $\bf{P}^n$ with the trivial grading. \vspace{-2mm}
        \item In the case of the Grassmannian $\Gr(k,V)$, Kapranov \cite{Kapranov85} constructs full exceptional coll\-ections of vector bundles. One of the two dual exceptional collections described in \emph{loc. cit.} uses the sheaves $\Sigma^\alpha(S)$, where $S$ is the tautological subbundle over $\Gr(k,V)$, and $\Sigma^\alpha$ denotes the Schur functor indexed by $\alpha$, where $\alpha$ is a Young diagram with $\le k$ rows and $\le n-k$ columns. The grading of $\{\Sigma^\alpha(S)\}$ is given by putting $\Sigma^\alpha(S) \preceq \Sigma^\beta(S)$ if and only if $\lvert \alpha \rvert \ge \lvert \beta\rvert$. 
        
        Note that this contains $\bf{P}^n$ as a special case. Indeed, in that case $V = \bf{C}^{n+1}$, $k=1$, and $\Omega^k(k)$ corresponds to the Young diagram which is a single column with $k$ rows. 
    \end{enumerate}
\end{ex}

\begin{lem}
\label{L:heartposet}
    Let $\cD$ be a $k$-linear triangulated category. Let \vspace{-2mm}
    \begin{enumerate}
        \item $(\cE,\preceq)$ be a graded full exceptional collection contained in the heart of a bounded t-structure on $\cD$; and \vspace{-2mm}
        \item $\nu:B(\cE) \to (\bf{Z},\le)$ an injective anti-homomorphism of posets.\vspace{-2mm}
    \end{enumerate} 
    Then, there is a heart of a bounded t-structure $\cA := \langle E[\nu(E)]:E\in \cE\rangle_{\rm{ext}}$ on $\cD$.
\end{lem}

\begin{proof}
    Write $\nu([E_i]) = \nu_i$. It suffices to show that $\{E_i[\nu_i]\}_{i=1}^n$ forms an Ext-exceptional collection by \cite{Macristabilityoncurves}*{Lem. 3.14}. That is, $\Ext^{\le 0}(E_i[\nu_i],E_j[\nu_j]) = 0$ for all $i<j$. By \Cref{D:gradedFECnew}, $i<j$ implies that $\nu_i\ge \nu_j$. 
    Now, $\Ext^{\le 0}(E_i[\nu_i],E_j[\nu_j]) = \Ext^{\le \nu_j-\nu_i}(E_i,E_j) = 0$ if $\nu_i>\nu_j$ by hypothesis (1). If $\nu_i = \nu_j$, then $\bigoplus_{\ell\in \bZ}\Ext^\ell(E_i,E_j)=0$, unless $i=j$.
\end{proof}

\begin{defn}
\label{D:normfunction}
    Given a graded exceptional collection $(\cE,\preceq)$, its \emph{norm} is the unique surjection $\nu: \cE\to \{0,\ldots,k\}$ such that \vspace{-2mm}
    \begin{enumerate}
        \item $\nu$ descends to a bijection $B(\cE) \to \{0,\ldots, k\}$; and \vspace{-2mm}
        \item enumerating $B(\cE) = \{b_k\prec\cdots \prec b_0\}$, we have $\nu(b_i) = i$.
    \end{enumerate}
\end{defn}

\begin{setup}
\label{Setup:complexheart}
    We assume that $\cD = \DCoh(X)$ and that $\cE$ is a graded full exceptional collection of sheaves with norm function $\nu$. Denote by $\mathfrak{b}_i$ the extension closure of the objects of $b_i$ placed in cohomological degree $-i$.
\end{setup}

\begin{rem}
    We make several remarks about \Cref{Setup:complexheart}. First, in this notation the heart $\cA$ from \Cref{L:heartposet} is simply $\langle \mathfrak{b}_i:i=0,\ldots, k\rangle_{\rm{ext}}$. Second, note that $\mathfrak{b}_i$ is simply the closure of $b_i[i]$ under direct sums. Finally, note that the norm function $\nu$ determines $\preceq$ and vice versa.
\end{rem}

\begin{ex}
    When we consider $(\Omega^n(n),\ldots, \Omega^1(1),\cO)$ on $\bf{P}^n$, $b_i = \{\Omega^i(i)\}$, and $\mathfrak{b}_i = \{\Omega^i(i)[i]^{\oplus l}:l\ge 0\}$. Thus, the heart $\cA$ is $\langle \Omega^n(n)[n],\ldots, \Omega^1(1)[1],\cO\rangle_{\rm{ext}}$.
\end{ex}

\begin{lem}
\label{L:standardformheart}
    In the context of \Cref{Setup:complexheart}, the heart $\cA$ from \Cref{L:heartposet} is the strict closure of the full subcategory of $\DCoh(X)$ consisting of complexes $Y_\bullet = (Y_k\to \cdots \to Y_0) $ where $Y_i \in \mathfrak{b}_i$ for all $i=0,\ldots, k$.
\end{lem}

\begin{proof}
    Let $\cF$ denote the full subcategory of $\DCoh(X)$ containing all complexes of the form $Y_\bullet$ as in the statement. Given $Y_\bullet \in \Ob(\cF)$, the stupid truncations $\sigma_{\ge j}$ as in \cite{stacks-project}*{\href{https://stacks.math.columbia.edu/tag/0118}{Tag 0118}} give morphisms
    \begin{equation*}
            \sigma_{\ge 0}(Y_\bullet)\to \sigma_{\ge 1}(Y_\bullet) \to  \cdots \to \sigma_{\ge k-1}(Y_\bullet) \to Y_\bullet
    \end{equation*}
    where $\Cone(\sigma_{\ge i-1}(Y_\bullet) \to \sigma_{\ge i}(Y_\bullet)) \in \mathfrak{b}_{i}$ for all $i=1,\ldots, k$. Thus, $\Ob(\cA) \supseteq \Ob(\cF)$. For the reverse inclusion, since $\Ob(\mathfrak{b}_i)\subseteq \Ob(\cF)$ for all $1\le i \le k$ it suffices to prove that $\cF$ is extension closed. For this, consider $A,B\in \Ob(\cF)$. By \cite{Kapranov1988}*{Lem. 1.6}, any morphism $f:A\to B[1]$ in $\DCoh(X)$ comes from a morphism of complexes. On the other hand, classes in $\Ext^1(A,B)$ correspond to morphisms $f:A\to B[1]$ by sending $f$ to the triangle $B\to \Cone(f)[-1]\to A$. However, the $n^{\rm{th}}$ entry of $\Cone(f)[-1]$ is $A_n\oplus B_n$ and so $\cF$ is extension closed.
\end{proof}

We remain in \Cref{Setup:complexheart}. Since $\cE = \{E_1,\ldots, E_n\}$ is a full exceptional collection, $\rm{K}_0(X) = \bigoplus_{i=1}^n \bf{Z}\cdot [E_i]$ and the cone of classes coming from $\cA$ is $\bigoplus_{i} \bf{N}\cdot (-1)^{\nu(i)}\cdot [E_i]$. 

\begin{defn}
\label{D:sharpgFEC}
    A graded exceptional collection $(\cE,\preceq)$ is called \emph{sharp} if it has a unique maximal element.
\end{defn}

In our setup, this means that $b_0 = \{E_n\}$. When $X = \Gr(k,V)$ as in \Cref{Ex:basicexamples}, $b_0 = \{\cO_X\}$. For the rest of the section, we suppose that $(\cE,\preceq)$ is a sharp graded full exceptional collection in $\DCoh(X)$ unless otherwise specified. 

Choose an object $F$ of $\cA$, which is isomorphic to a complex $Y_\bullet$ by \Cref{L:standardformheart}. The projection of $F$ onto the subgroup of $\rm{K}_0(\cA)$ generated by $[E_n]$ is $[\sigma_{\ge 0}(Y_\bullet)] = d_n \cdot [E_n]$. This value is independent of $Y_\bullet$ and we set $\pr_{E_n}(F) := d_n$.

For completeness, we record the following well-known lemma:

\begin{lem}
\label{L:stabconditionfinitelength}
    Let $\cD$ denote a $k$-linear triangulated category and let $\cA$ denote a heart of a bounded t-structure on $\cD$. Suppose that $\cA$ is finite length and has finitely many simple objects $S_1,\ldots, S_n$. Then, specifying a stability condition on $\cD$ with underlying heart $\cA$ is equivalent to specifying $Z(S_1),\ldots, Z(S_n) \in \bf{H}\cup \bf{R}_{<0}$. 
\end{lem}

\begin{proof}
    That this defines a pre-stability condition on $\cD$ is immediate from \cite{Br07}*{Lem. 2.4} combined with Prop. 5.3 \emph{ibid}, using the finite length property. To check the support property, for any non-zero $E$ in $\cA$, write $[E] = \sum_{i=1}^n m_i\cdot [S_i]$ for $m_i \in \bf{Z}_{\ge 0}$. Then, for any norm $\lVert\:\cdot\:\rVert$ on $\rm{K}_0(\cD)_{\bf{R}}$
    \[
        \frac{\lvert Z(E)\rvert}{\lVert E\rVert} \ge \min \left\{\frac{\lvert Z(S_i)\rvert}{\lVert S_i\rVert}\right\}_{i=1}^n>0
    \]
    from which the result follows.
\end{proof}

Returning to \Cref{Setup:complexheart}, by \Cref{L:standardformheart}, $\cA$ is a finite length heart since it is generated under extensions by the simple objects $\{E_i[\nu_i]\}_{i=1}^n$. Consequently, by \Cref{L:stabconditionfinitelength} we can specify a stability condition $\sigma$ on $\DCoh(X)$ with underlying heart $\cA$ uniquely by $z_i := Z(E_i[\nu_i]) \in \bf{H}\cup \bf{R}_{<0}$ for each $i=1,\ldots, n$. 

Recall that given a central charge homomorphism $Z:\rm{K}_0(\cA)\to \mathbf{C}$ and $E\in \Ob(\cA)$, we let $\phi(E) = \tfrac{1}{\pi}\arg Z(E)$ where $\arg$ is the branch of the argument function which on $\bf{H}\cup \bf{R}_{<0}$ is valued in $(0,\pi]$. We begin with a technical definition:

\begin{defn}
    In \Cref{Setup:complexheart}, an object $F$ of $\cD$ is \textit{efficient} with respect to $(\cE,\nu)$ if 
    \[
        \Hom_{\cD}(E_i[\nu_i],F) = 0
    \]
    for all indices $i$ such that $\nu_i \ne 0$. 
\end{defn}

When we consider efficient objects below, we omit $(\cE,\preceq)$ when it is obvious from the context. The following lemma shows that efficient objects arise in practice.

\begin{lem}
\label{L:efficientsheaf}
    In \Cref{Setup:complexheart}, every object of $\Coh(X)$ is efficient.
\end{lem}

\begin{proof}
    This is immediate since $\Hom(E_i[\nu_i],F) = \Ext^{-\nu_i}_{\Coh(X)}(E_i,F) = 0$ for all $\nu_i>0$.
\end{proof}

\begin{hyp}
\label{H:stability}
    In \Cref{Setup:complexheart}, suppose that $(\cE,\preceq)$ is sharp and let $\cA$ be the heart constructed by \Cref{L:heartposet} and let $0\neq F\in \Ob(\cA)$. 
    By \Cref{L:stabconditionfinitelength}, we can choose $Z \in \Hom_{\mathbf{Z}}(\rm{K}_0(\cA),\bf{C})$ taking $\cA$ to $\bf{H}\cup \bf{R}_{<0}$ such that 
    \[
        \phi(E_n) < \phi(F) < \phi^-(b_1) \le  \phi^+(b_1)\le \cdots\le  \phi^-(b_k)\le \phi^+(b_k).
    \]
    where $\phi^+(b_j) = \max \{\phi(E_i[\nu_i]):E_i\in b_j\}$ and $\phi^-(b_j)$ is defined analogously for all $1\le i \le k$.
\end{hyp}

\begin{prop}
\label{P:efficientstable}
    We use the notation and assumptions of \Cref{H:stability}. If $F$ is an efficient object of $\cA$ such that $\pr_{E_n}(F) = 1$, then $F$ is stable with respect to $\sigma = (Z,\cA)$. 
\end{prop}

\begin{proof}
    Consider a subobject $0 \to Y \to F$ in $\cA$. Since $\pr_{E_n}(-):\Ob(\cA) \to \bf{N}$ is additive on exact sequences, it follows that $\pr_{E_n}(Y) \in \{0,1\}$. If $\pr_{E_n}(Y) = 0$, then $Y$ is in $\langle E_i[\nu_i]:\nu_i \ne 0\rangle$ and $\Hom(Y,F) = 0$ since $F$ is efficient. 
    
    So, if $Y\ne 0$, it must be that $\pr_{E_n}(Y) = 1$ and $\pr_{E_n}(Q) = 0$, where $Q$ is the resulting quotient object. Thus, $\phi(F) < \phi(Q)$ since $Z(Q)$ lies in the cone in $\mathbf{H}\cup \bf{R}_{<0}$ generated by $\{Z(E_i):i \ne n\}$. Therefore, $\phi(Y)<\phi(F)$ and $F$ is stable.
\end{proof}

Recall that a stability condition on $\DCoh(X)$ is called \emph{geometric} if all skyscraper sheaves of points are stable of the same phase. 

\begin{cor}
\label{C:geomstability}
    Suppose that $\DCoh(X)$ admits a sharp graded full exceptional collection of sheaves $(\cE,\preceq)$, that $\cO_x \in \Ob(\cA)$ for all $x\in X$, where $\cA$ is as in \Cref{L:heartposet}, and that $\pr_{E_n}(\cO_x) = 1$. 
    Then, $\DCoh(X)$ admits a geometric stability condition with underlying heart $\cA$.
\end{cor}

\begin{proof}
    By \Cref{L:efficientsheaf}, $F = \cO_x$ is efficient for all $x\in X$ and lies in $\cA$ by hypothesis. Thus, by \Cref{P:efficientstable}, there is a central charge $Z$ such that all $\cO_x$ are stable of the same phase $\frac{1}{\pi}\arg Z(\cO_x) \in (0,1]$.
\end{proof}

\subsection{Geometric stability for some homogeneous varieties}
\label{SS:geomstabilityfromFEC}
We use the results of \Cref{SS:geomFEC} to produce new examples of geometric stability conditions in several cases. We also construct stability conditions which are ``close'' to being geometric.

\begin{defn}
\label{D:almostgeometric}
    Let $\cX$ denote a Deligne-Mumford stack with a nonempty maximal open dense substack $U$ isomorphic to a scheme. A stability condition $\sigma$ on $\DCoh(\cX)$ is \emph{almost geometric} if there is an open dense subset $U' \subseteq U$ and $\phi \in \mathbf{R}$ such that for all $x\in U'$, $\cO_x$ is $\sigma$-stable of phase $\phi$.
\end{defn}

The rest of the section is dedicated to proving the following result -- some of the terminology is introduced below.

\begin{thm}
\label{T:geomexistence}
    The following varieties admit geometric stability conditions: \vspace{-2mm}
    \begin{enumerate}
        \item finite products of Grassmannians; and \vspace{-2mm}
        \item smooth projective quadric hypersurfaces $Q\subset \bf{P}^n$. \vspace{-2mm}
    \end{enumerate}
    In addition, for all $n \ge 1$ \vspace{-2mm}
    \begin{enumerate}[resume]
        \item $\Hilb^{n}(\bf{P}^2)$, $\Hilb^n(\bf{P}^1\times \bf{P}^1)$; and \vspace{-2mm}
        \item the weighted projective stacks $\mathbf{P}(a_0,\ldots, a_n)$ for $\gcd(a_0,\ldots, a_n) = 1$ \vspace{-2mm}
    \end{enumerate}
    admit almost-geometric stability conditions. Furthermore, for any $X$ in the above list, any stability condition $\sigma \in \Stab(X)$ constructed here, and any $v\in \Lambda$, the stack $\cM_\sigma^{\rm{ss}}(v)$ of $\sigma$-semistable objects of class $v$ admits a proper good moduli space.
\end{thm}

\subsubsection*{Geometric stability conditions for products of Grassmannians}

\begin{setup} 
\label{Setup:Koszul}
    Consider a variety $X$ and a closed point $x\in X$ such that there is a vector bundle $E$ of rank $r$ on $X$ and a morphism of sheaves $f_x:E\to \cO_X$ with cokernel $\cO_x$. We obtain a Koszul resolution
    \[
        \cK_\bullet(f_x) := \left[0 \to \Lambda^{r} E \to \cdots \to \Lambda^2 E \to E \xrightarrow{f_x} \cO_X\right] \simeq  \cO_x
    \]
    of $\cO_x$. Suppose furthermore that\vspace{-2mm}
    \begin{enumerate}
        \item $\cO_X$ is exceptional; \vspace{-2mm} and 
        \item there is a graded exceptional collection $(\cE,\preceq)$ of vector bundles on $X$, with blocks $\{b_r\prec \cdots \prec b_0\}$ such that $\Lambda^i(E)[i] \in \mathfrak{b}_i$ in the notation of \Cref{Setup:complexheart}.\vspace{-2mm} 
    \end{enumerate}
    By definition, $\Lambda^0 E = \cO_X$ so that $b_0 = \{\cO_X\}$.
\end{setup}

\begin{lem} 
\label{L:FEC}
    In \Cref{Setup:Koszul}, if for each $x\in X$ there exists $f_{x} \in \Hom(E,\cO_X)$ such that $\coker f_{x} = \cO_{x}$, then $\cE$ is full in $\DCoh(X)$.
\end{lem}

\begin{proof}
    Being generated by an exceptional collection, $\cA = \langle \cE\rangle$ is admissible by \cite{Bondal1990}*{Thm. 3.2}. On the other hand, the quasi-isomorphism $\cK_\bullet(f_x) \simeq \cO_x$ for each $x\in X$ implies that $\cO_x \in \cA$.
    Consider the semiorthogonal decomposition $\DCoh(X) = \langle \cA,{}^\perp \cA\rangle$. Any $F\in {}^\perp \cA$ has $\RHom(F,\cO_x) = 0$ for all $x\in X$, but then since $\{\cO_x:x\in X\}$ is a spanning class for $\DCoh(X)$ and $\DCoh(X)$ has a Serre functor, ${}^\perp \cA = 0$ by \cite{bridg98_equivFM}*{Ex. 2.2}.
\end{proof}

\begin{ex}
\label{E:Kapranovcollection}
    The key example where \Cref{Setup:Koszul} holds is $\Gr(k,V)$, where $V$ is a finite dim\-ensional vector space and $1\le k \le \dim V - 1$. Consider the tautological subbundle $S \subset \Gr(k,V) \times V$ which has fiber over $P\in \Gr(k,V)$ the vector space $P$, and the perpendicular bundle $S^\perp$ which has as its fiber over $P$ the space of $\phi \in V^*$ such that $\phi|_P = 0$. 
    
    There is a natural evaluation map $S\boxtimes S^\perp \to \cO_{G\times G}$ given on the fiber over a closed point $(P,Q) \in G\times G$ by $(v,\phi)\mapsto \phi(v)$. The cokernel of this morphism is $\cO_{\Delta}$ and the resulting Koszul complex
    \[
        0 \to \Lambda^{k(n-k)}(S\boxtimes S^\perp) \to \cdots \to S\boxtimes S^\perp \to \cO_{G\times G}\to \cO_\Delta \to 0
    \]
    gives a resolution of $\cO_\Delta$. By \cite{Kapranov85}*{Lem. 0.4}, for each $1\le m \le k(n-k)$ we have
    \[
        \Lambda^m(S\boxtimes S^\perp) = \bigoplus_{\lvert \alpha \rvert = m} \Sigma^\alpha(S) \boxtimes \Sigma^{\alpha^*}(S^\perp)
    \]
    where the sum ranges over all Young diagrams of size $m$, with at most $k$ rows and $\dim V - k$ columns. Pulling back along $\Gr(k,V)\times \{Q\} \hookrightarrow \Gr(k,V)\times \Gr(k,V)$ for some closed point $Q\in \Gr(k,V)$ gives us a resolution 
    \[
        \cdots \to \bigoplus_{\lvert \alpha \rvert = m} \Sigma^\alpha(S) \otimes \Sigma^{\alpha^*}_Q(S^\perp) \to \cdots \to S \otimes Q^\perp \to \cO_{G} \to \cO_{\{Q\}} \to 0
    \]
    of $\cO_{\{Q\}}$. It is proven in \cite{Kapranov85} that the set of sheaves $\Sigma^\alpha(S)$ appearing in the entries of the above complex forms a strong full exceptional collection on $\DCoh(\Gr(k,V))$ which we denote by $\mathfrak{K}$ and refer to as the \emph{Kapranov collection}. Recall that there is a grading on $\mathfrak{K}$ with norm $\nu: \mathfrak{K} \to \{0,\ldots, k(n-k)\}$ given by $\nu(\Sigma^\alpha(S)) = \lvert \alpha \rvert$.
\end{ex}

\begin{rem}
    In \cite{Kapranov1988}*{Ex. 3.11}, similar resolutions of point sheaves on all type $A$ flag varieties $\rm{Fl}(i_1,\ldots, i_k;n)$ are constructed. However, there is no grading on the resulting strong full exceptional collections denoted $X(i_1,\ldots, i_k;n)$ compatible with the cohomological grading of the terms of the resolution as required in \Cref{Setup:Koszul} when $k>1$.
\end{rem}

\begin{hyp}
\label{H:koszulproduct}
    Suppose that $X$ (resp. $Y$) is a variety with a vector bundle $E$ (resp. $F$) of rank $r$ (resp. $s$) as in \Cref{Setup:Koszul} and \Cref{L:FEC}. In particular, $\DCoh(X)$ (resp. $\DCoh(Y)$) has a graded full exceptional collection of sheaves $(\cE,\preceq)$ (resp. $(\cF,\preceq')$) with norm function $\nu$ (resp. $\mu$).
\end{hyp}

In what follows, we abuse notation by writing $E$ and $F$ instead of $\pr_1^*(E)$ and $\pr_2^*(F)$ and regarding them as sheaves on $X\times Y$.

\begin{prop}
\label{P:productstabconditions}
    If \Cref{H:koszulproduct} hold, then there is a Koszul resolution 
    \begin{equation*}
        \cK_\bullet(f_x,f_y):=\left[0\to \Lambda^{r+s}(E\oplus F) \to \cdots \to \Lambda^2 (E\oplus F) \to E\oplus F \xrightarrow{(f_x,f_y)}\cO_{X\times Y}\right] \simeq \cO_{(x,y)}
    \end{equation*}
    for each closed point $(x,y)\in X\times Y$. Furthermore, \vspace{-2mm}
    \begin{enumerate}
        \item $\cE \boxtimes \cF := \{E'\boxtimes F':E'\in \cE,F'\in \cF\}$  is a full exceptional collection, with the grading defined by the norm $\nu\boxtimes \mu: \cE\boxtimes \cF \to \{0,\ldots, r+s\}$ given by $(\nu\boxtimes\mu)(E'\boxtimes F') = \nu(E') + \mu(F')$; and \vspace{-2mm}
        \item if $\cE$ and $\cF$ are strong, then so is $\cE\boxtimes \cF$.\vspace{-2mm}
    \end{enumerate}  
\end{prop}

\begin{proof}
    The claim that $\cK_\bullet(f_x,f_y)$ defines a resolution of $\cO_{(x,y)}$ is immediate from the fact that $\coker(f_x,f_y) = \cO_{(x,y)}$. To see that $\cE\boxtimes \cF$ can be ordered so as to form an exceptional collection, note that there are no morphisms in $\DCoh(X\times Y)$ between $b_p \boxtimes b_q'$ and $b_k\boxtimes b_l'$ whenever $p+q \ge k+l$ by the K\"unneth formula \cite{stacks-project}*{\href{https://stacks.math.columbia.edu/tag/0BEC}{Tag 0BEC}}. 
    
    Consequently, $\cB = \langle \cE\boxtimes \cF\rangle$ is an admissible subcategory of $\DCoh(X\times Y)$ and the identity $\Lambda^k(E\oplus F) = \bigoplus_{p+q = k} \Lambda^p(E) \boxtimes \Lambda^q(E)$ implies that $\Lambda^k(E\oplus F)$ is a sum of sheaves in $\bigcup_{p+q = k} (b_p\boxtimes b_q')$. Thus, $\cO_{(x,y)} \in \cB$ for all $(x,y)$ and by \Cref{L:FEC} $\cE\boxtimes \cF$ is full.
    
     Claim (2) about strength of the exceptional collection is proven similarly using the K\"unneth formula.
\end{proof}

\begin{cor}
    For any $(d_1,\ldots, d_n) \in \bf{N}^n$, the variety $\prod_{i=1}^n \bf{P}^{d_i}$ admits geometric stability conditions.
\end{cor}

\begin{proof}
    Apply \Cref{P:productstabconditions} inductively, using the fact that $\bf{P}^d$ satisfies \Cref{H:koszulproduct} and the hypothesis of \Cref{L:FEC} with $E = \Omega^1_{\bf{P}^d}(1)$.
\end{proof}

\subsubsection*{Almost geometric stability conditions for $\Hilb^n(\bf{P}^2)$, $\Hilb^n(\bf{P}^1\times \bf{P}^1)$, and $\mathbf{P}(a_0,\ldots, a_n)$}

We write $\bf{P}^{d\times n} = (\bf{P}^d)^n$. To obtain a Koszul resolution of the structure sheaf of a closed point $(x_1,\ldots, x_n) \in \bf{P}^{d\times n}$, one can use the bundle $\Omega_{\bf{P}^{d\times n}}(1):= \bigoplus_{i=1}^n \pr_i^*\Omega_{\bf{P}^d}(1)$ on $\bf{P}^{d\times n}$, and the morphism 
\[
    (f_1,\ldots,f_n) \in \Hom (\Omega_{\bf{P}^{d\times n}}(1),\cO_{\bf{P}^{d\times n}})\cong \prod_{i=1}^n \Hom(\Omega_{\bf{P}^d}(1),\cO_{\bf{P}^d})
\]
such that $\coker(f_i:\Omega_{\bf{P}^d}(1) \to \cO_{\bf{P}^d}) = \cO_{x_i}$. For any $p_\bullet = (p_1,\ldots, p_n) \in \bf{N}^n$, define 
\[
    \Omega(p_\bullet) = \Lambda^{p_1}\Omega_{\bf{P}^d}(1) \boxtimes \cdots \boxtimes \Lambda^{p_n}\Omega_{\bf{P}^d}(1).
\]
Repeated application of the identity $\Lambda^k(E\oplus F) = \bigoplus_{p+q = k} \Lambda^p(E) \otimes \Lambda^q(F)$ gives 
    \[
        \Lambda^k \Omega_{\bf{P}^{d\times n}}(1) = \bigoplus_{\Sigma_i \:p_i \:= \:k} \Omega(p_\bullet)
    \]
for all $1\le k \le nd$. Thus, applying \Cref{P:productstabconditions} we obtain a graded strong full exceptional collection $\{\Omega(p_\bullet): 0\le \sum p_i \le nd\}$ on $\bf{P}^{d\times n}$ with norm function $\nu(\Omega(p_\bullet)) =\sum p_i$. 

\begin{cor}
\label{C:geomstabproductpd}
    The heart $\cA = \langle\Omega(p_\bullet)[\sum p_i]: 0\le \sum p_i \le nd \rangle_{\rm{ext}}$ on $\DCoh(\bf{P}^{d\times n})$ supports a geometric stability condition $\sigma$ such that $\sum p_i < \sum q_i\Rightarrow \phi(\Omega(p_\bullet)[\sum p_i]) < \phi(\Omega(q_\bullet)[\sum q_i])$. Further\-more, if $\sum p_i = \sum q_i \Rightarrow Z_\sigma(\Omega(p_\bullet)[\sum p_i]) = Z_\sigma(\Omega(q_\bullet)[\sum q_i])$ then $\sigma$ is $\mathfrak{S}_n$-invariant. 
\end{cor}

\begin{proof}
    One can see that $\cA$ is the heart of a bounded t-structure by applying \Cref{L:heartposet} with $\nu(\Omega(p_1,\ldots, p_n)) = \sum p_i$. On the other hand, \Cref{P:productstabconditions} combined with \Cref{C:geomstability} gives existence of the claimed stability condition $\sigma$. 

    Next, observe that $\tau \in \mathfrak{S}_n$ acts on $\{\Omega(p_\bullet)[\sum p_i]\}$ by $\Omega(p_\bullet)[\sum p_i]\mapsto \Omega(p_{\tau(1)},\ldots, p_{\tau(n)})[\sum p_i]$ so that $\mathfrak{S}_n$ preserves $\cA$. Also, if $\sum p_i = \sum q_i \Rightarrow Z_\sigma(\Omega(p_\bullet)[\sum p_i]) = Z_\sigma(\Omega(q_\bullet)[\sum q_i])$ then $Z$ is $\mathfrak{S}_n$-invariant and the claimed result follows. 
\end{proof}

Let $S$ be a smooth projective surface. The works of Haiman \cite{haiman2001hilbert} and Bridgeland-King-Reid \cite{BKR2001} construct an exact equivalence
\[ 
    \Phi\colon \cD^b(\Hilb^n(S)) \xlongrightarrow{\sim} \cD^b([S^{n}/{\mathfrak{S}_n}])
\]
which we call the BKRH-equivalence. There is a canonical morphism of stacks $f:[S^{ n}/{\mathfrak{S}_n}] \to \rm{B}\mathfrak{S}_n$ and so because $\Coh(\rm{B}\mathfrak{S}_n) = \rep(\mathfrak{S}_n)$, $\rm{K}_0([S^{n}/{\mathfrak{S}_n}])$ has a module structure 
\[
    \rm{K}_0(\rep(\mathfrak{S}_n)) \otimes_{\bf{Z}} \rm{K}_0([S^{n}/{\mathfrak{S}_n}]) \to \rm{K}_0([S^{ n}/{\mathfrak{S}_n}])
\]
given by $[V]\otimes [E] \mapsto [f^*(V)\otimes E].$ 

Next, since a sheaf on $[S^{n}/\mathfrak{S}_n]$ is the same as an $\mathfrak{S}_n$-equivariant sheaf on $S^{ n}$, there is a canonical exact functor $\Coh([S^{ n}/\mathfrak{S}_n]) \to \Coh(S^{ n})$ which forgets the equivariant structure. We denote by $\cF:\DCoh([S^{ n}/\mathfrak{S}_n]) \to \DCoh(S^{ n})$ the induced functor and by $\cI:\DCoh(S^{ n}) \to \DCoh([S^{ n}/\mathfrak{S}_n])$ its left adjoint. The functors above induce a morphism $\cF_*\circ \Phi_*:\rm{K}_0(\Hilb^n(S)) \to \rm{K}_0(S^n)$. \cite{DellHengLicata}*{Lem. 4.5(iv)} implies that $\cF_*\circ \cI_* \in \End(\rm{K}_0(S^n))$ equals $n!\cdot \id$ and thus $\cF_*\circ \Phi_*$ is a rational surjection.

\begin{prop}
\label{P:Hilbnstab}
    For $S = \bf{P}^2$ or $\bf{P}^1\times \bf{P}^1$, $\DCoh(\Hilb^n(S))$ admits stability conditions satisfying the sup\-port prop\-erty with respect to $\rm{Im}(\cF_*\circ \Phi_*: \rm{K}_0(\Hilb^n(S)) \to \rm{K}_0(S^{ n}))$. These stability conditions are $\rep(\mathfrak{S}_n)$-equivariant in that their central charges $Z$ satisfy
    \[
        Z(f^*(V)\otimes(-)) = \dim V \cdot Z(-) \text{ for all $V\in \rep(\mathfrak{S}_n)$}.
    \]
\end{prop}

\begin{proof}
    The stability conditions on $\DCoh(\Hilb^n(S))$ are constructed by first inducing stability cond\-itions along the functor $\cF: \DCoh([S^{n}/\mathfrak{S}_n]) \to \DCoh(S^n)$ (cf. \cites{DellHengLicata,macri2009inducing}) and then applying $\Phi: \DCoh(\Hilb^n(S)) \xlongrightarrow{\sim} \DCoh([S^{n}/\mathfrak{S}_n]).$ To induce along $\cF$, one must find an $\mathfrak{S}_n$-invariant stability condition on $\DCoh(S^n)$. For $\bf{P}^{2\times n}$, this follows directly from \Cref{C:geomstabproductpd}. For $(\bf{P}^1\times \bf{P}^1)^n$, this follows from \Cref{C:geomstabproductpd} applied to $\bf{P}^{1\times 2n}$ and the diagonal subgroup $\mathfrak{S}_n\subset \mathfrak{S}_{2n}$. The claims about $\rep(\mathfrak{S}_n)$-equivariance and the support property follow from \cite{DellHengLicata}*{Thm. 4.8} and Appendix A \emph{ibid}.
\end{proof}

We consider the following setup in the subsequent proposition.

\begin{setup}
\label{Setup:quotientstability}
    Let $X$ be a smooth projective variety with an action by a finite group $G$ and let $\sigma \in \Stab(X)$ denote a $G$-invariant geometric stability condition such that all point sheaves have phase $\phi \in\bf{R}$. Denote by $\pi:X\to [X/G]$ the quotient stack.
\end{setup}

The morphism $\pi$ induces an exact functor $\pi^*:\Coh([X/G])\to \Coh(X)$ which induces an exact functor $\pi^*:\DCoh([X/G]) \to \DCoh(X)$. By \cite{macri2009inducing}*{Prop. 2.17} and the surrounding discussion, there is an \emph{induced} $\eta\in \Stab([X/G])$ with central charge factoring through the induced map $\rm{K}_0([X/G]) \to \rm{K}_0(X)$ and $\cP_\eta(\phi) = \{E\in \DCoh([X/G]):\pi^*(E) \in \cP_\sigma(\phi)\}$ for all $\phi \in \bf{R}$.

\begin{lem}
\label{L:quotientstackstability}
    Given \Cref{Setup:quotientstability}, denote by $\eta$ the induced stability condition on $\Stab([X/G])$. The sheaves in $\Coh([X/G])$ corresponding to $G$-orbits are $\eta$-semistable and they are $\eta$-stable if and only if they correspond to an orbit with cardinality $\lvert G\rvert$.
\end{lem}

\begin{proof}
    This is a minor modification of part of the proof of \cite{Dellfreeabelian}*{Thm. 3.3}. Consider $\cO_p$, where $p$ is a closed point of $[X/G]$ corresponding to a $G$-orbit in $X$. By definition, $\pi^*\cO_p = \bigoplus_{g\in G} \cO_{g x}$ and thus $\cO_p$ is semistable of phase $\phi$. Suppose there exists a  short exact sequence
    \[
        0 \to E \to \cO_p\to F\to 0
    \]
    in $\cP_\eta(\phi)$. Since $\pi^*(E)$ is a subobject of $\pi^*(\cO_p)$ in $\cP_\sigma(\phi)$, we have $\pi^*(E) = \bigoplus_{p\in A} \cO_{px}$ where $A\subseteq G$ is a subset. But, since $\pi^*E$ is $G$-equivariant, $\rm{supp}(\pi^*E)$ is $G$-invariant. When $p$ has trivial automorphism group, this means that $A = \varnothing$ or $A = G$ and thus $\cO_p$ is stable.

    If the automorphism group of $p$ is non-trivial, choose $y$ in the orbit corresponding to $p$ with non-trivial isotropy subgroup $G_y$. In this case, $\pi_*\cO_y = \cO_p$ and the (bi)adjunction $\pi_*\dashv \pi^*\dashv \pi_*$ gives 
    \[
        \Hom(\pi_*(\cO_y),\pi_*(\cO_y)) \cong \Hom(\cO_y,\pi^*\pi_*(\cO_y)) \cong \bf{C}^{\lvert G_y\rvert}
    \]
    whence it follows that $\pi_*(\cO_y) = \cO_p$ is not $\eta$-stable.
\end{proof}

It is an immediate consequence of \Cref{L:quotientstackstability} that if the $G$-action is generically free, then the resulting stability conditions on $[X/G]$ are almost-geometric.

\begin{cor}
\label{C:weightedprojectivestab}
    The weighted projective stacks $\bf{P}(a_0,\ldots, a_n)$ for $\gcd(a_0,\ldots, a_n) = 1$ admit almost geometric stability conditions with respect to 
    \[
        \pi^*:\rm{K}_0(\bf{P}(a_0,\ldots, a_n)) \to \rm{K}_0(\bf{P}^n)
    \]
    such that all but the orbifold points are stable of the same phase.
\end{cor}

\begin{proof}
    First, note that $\bf{P}(a_0,\ldots, a_n)$ can be obtained from $\bf{P}^n$ by quotienting by the action of the diagonally embedded $G:=\prod_{i=0}^n \bf{Z}/a_i\bf{Z} \hookrightarrow \GL_{n+1}(\bf{C})$. The action of $\GL_{n+1}(\bf{C})$ on $\bf{P}^n$ induces an action on $\Stab(\bf{P}^n)$, which by \cite{PolishchukConstant}*{Cor. 3.5.2} is trivial. The result now follows from \Cref{L:quotientstackstability}, noting that the fixed points of the $G$-action are exactly the coordinate axes $\bf{C}\cdot e_i \subset \bf{C}^{n+1}$ for $i=0,\ldots, n$. 
\end{proof}

\begin{cor}
    If $\sigma$ is one of the stability conditions in \Cref{P:Hilbnstab}, then for any $x\in \Hilb^n(S)$, $\cO_x$ is $\sigma$-stable if and only if the corresponding $\mathfrak{S}_n$-cluster contains $n!$ distinct points. In particular, $\sigma$ is almost-geometric.
\end{cor}

\begin{proof}
    This follows from \Cref{L:quotientstackstability}, since the BKRH equivalence $\Phi$ sends $\cO_x$ to $\cO_{Z_x}$ where $\cO_{Z_x}$ is the associated $\mathfrak{S}_n$-cluster, regarded as a $\mathfrak{S}_n$-equivariant sheaf on $S^n$. 
\end{proof}

\subsubsection*{Geometric stability conditions for quadrics} Consider an $n$-dimensional vector space $V$ eq\-uipped with a non-degenerate symmetric bi\-linear form $q$. We let $Q = Q(V)\subseteq \bf{P}(V)$ denote the quadric of $q$-isotropic lines of $V$. Kapranov \cite{Kapranov1988} constructs a resolution of the diagonal on $Q\times Q$:
    \begin{equation}
    \label{E:quadricres}
        \left[\cdots\to \Psi_2\boxtimes \cO(-2)\to \Psi_1\boxtimes \cO(-1) \to \cO_{Q\times Q}\right]  \simeq \cO_\Delta
    \end{equation}
where the $\Psi_i$ are certain canonically defined locally free sheaves \cite{Kapranov1988}*{p. 497}. This resolution gives rise to a strong full exceptional collection $\cE = \{\Sigma_{(\pm) }(-n+2),\cO(-n+3),\ldots, \cO\}$ where $\Sigma_{(\pm)}$ is the unique spinor bundle for $n$ odd and the positive and negative spinor bundles for $n$ even. $\cE$ admits a natural grading:
\[
    \Sigma_+(-n+2)\sim \Sigma_-(-n+2) \prec \cO(-n+3)\prec \cdots \prec \cO.
\]

\begin{prop}
    $\DCoh(Q)$ admits geometric stability conditions, glued from $\cE$.
\end{prop}

\begin{proof}
    Restricting \eqref{E:quadricres} along $\{x\}\times Q \hookrightarrow Q\times Q$ gives a resolution of the form
    \[
        \left[\Sigma_+(-n+2)^{\oplus m_+} \oplus \Sigma_-(-n+2)^{\oplus m_-}\to  \cO(-n+3)^{\oplus r_{n-3}} \to \cdots \to \cO(-1)^{\oplus r_{1}}\to \cO_Q\right] \simeq \cO_x
    \]
    when $n$ is even by \cite{Kapranov1988}*{Prop. 4.7}. In the case where $n$ is odd, the leftmost term is $\Sigma(-n+2)^{\oplus m}$. Here, $r_i$ denotes the rank of the bundle $\Psi_i$. In both cases, $(\cE,\preceq)$ is a sharp, graded, and strong full exceptional collection. The result now follows from \Cref{L:heartposet} and \Cref{C:geomstability}, taking $\nu$ to be the norm function of $\cE$.
\end{proof}

\subsubsection*{Existence of proper good moduli spaces} To conclude the proof of \Cref{T:geomexistence}, we only need to address the claim about moduli spaces of semistable objects. 

\begin{lem}
\label{L:goodmodulispaceFEC}
    Let $\cD$ be a $k$-linear pre-triangulated dg-category with a full exceptional col\-lection 
    \begin{equation}
    \label{E:FEC}
        \cD = \langle E_1,\ldots, E_n\rangle.
    \end{equation} 
    For any $v\in \rm{K}_0(\cD)$ and any $\sigma \in \Stab(\cD)$ on the same connected component of a stability condition glued from \eqref{E:FEC}, the moduli stack $\cM^{\rm{ss}}_\sigma(v)$ admits a proper good moduli space.
\end{lem}

\begin{proof}
    This follows from the results of \cite{augmented}. Indeed, any $\sigma\in \Stab(\cD)$ constructed from gluing using \eqref{E:FEC} has a mass-Hom bound in the sense of Def. 2.6 \emph{ibid}. Lem. 2.7 of \emph{loc. cit.} then implies that any $\tau$ on the same connected of $\Stab(\cD)$ has a mass-Hom bound. Finally, \cite{augmented}*{Thm. 2.31} implies that for any $\sigma \in \Stab(\cD)$ with a mass-Hom bound and any $v\in \rm{K}_0(\cD)$, $\cM_\sigma^{\rm{ss}}(v)$ admits a proper good moduli space.
\end{proof}

\Cref{L:goodmodulispaceFEC} implies that all of the stability conditions constructed on products of Grass\-mannians and quadrics above admit proper good moduli spaces. The only remaining claim concerns good moduli spaces for the stability conditions on $\Hilb^n(\bf{P}^2), \Hilb^n(\bf{P}^1\times \bf{P}^1)$ and $\bf{P}(a_0,\ldots, a_n)$. This can be deduced from \cite{augmented}*{Thm 2.31} combined with the fact that Lem. 2.11 \emph{ibid.} which says that induction in the sense of \cite{macri2009inducing} preserves mass-Hom bounds. This is sufficient because \cite{DellHengLicata}*{Thm. 4.8} uses induction to construct the $\rep(\mathfrak{S}_n)$-equivariant stability conditions on $\Hilb^n(\bf{P}^2)$ and $\Hilb^n(\bf{P}^1\times \bf{P}^1)$, and analogously for $\bf{P}(a_0,\ldots, a_n)$ with the group $\prod_{i=0}^n \bf{Z}/a_i\bf{Z}$.

\subsection{Gluing stability on a Kuznetsov-type decomposition}\label{section_gluingpaths_general}

In this section we consider a smooth projective variety $X$ whose derived category $\DCoh(X)$ has a semiorthogonal decomposition of the form $\DCoh(X) = \langle \cT,\cN\rangle$ where $\cN = \langle E_0,\ldots, E_n\rangle$, for $\cE = \{E_0,\ldots, E_n\}$ a graded exceptional collection of sheaves. 
The main result is a criterion to glue stability conditions on such decomposition and get a geometric stability condition, see \Cref{thm:gluedgeometricKuz}.
We also give  a criterion to glue paths of stability conditions, see \Cref{prop_gluing_paths_Ku}.

Fix a norm function $\nu \colon \cE \rightarrow \bf{Z}_{\geq 0}$ such that $\nu(E_n) = 0$ -- see \Cref{D:normfunction}.

\begin{defn}\label{def:spiked}
    Let $\cN$ be a triangulated category and consider $\eta \in \Stab(\cN)$ with heart $\cC$ generated by simple objects $S_0,\dots,S_n$. Fix $u \in \rm{K}_0(\cN)$ and consider the cone 
    \[
        C = \left(\bigoplus_{i = 0}^{n-1} \bf{R}_{\geq 0} \cdot [S_i]\right)\backslash  \{0\} \subset \rm{K}_0(\cN)_{\bf{R}}.
    \]
    A stability condition $\eta = (W,\cC)$ is \emph{spiked} at $S_n$ with respect to $u$ if \vspace{-2mm}
    \begin{enumerate}
    \item $\Re W(S_i) > 0$ and $\Im W(S_i)>0$ for $i = 0,\dots n$, and \vspace{-2mm} 
    \item the following inequality  holds:
    \[
    \Re W(S_n) - \frac{\Re W(w)}{\Im W(w)} \Im W(u) -1 > 0
    \]
    for any $w \in C$. \vspace{-2mm}
    \end{enumerate}
\end{defn}
For a graded full exceptional collection $(\cE,\preceq)$ on $\cN$ with norm $\nu$, we will set $S_i:=E_i[\nu(E_i)]$.
\begin{lem}
\label{L:spikedexist}
    Choose a norm $\nu$ on $\cE$ and let $\cC = \langle E[\nu(E)]:E\in \cE\rangle_{\rm{ext}}$ denote the resulting heart on $\cN$ from \Cref{L:heartposet}. Then, for any $0\ne E \in \cC$, there exists a stability condition on $\cN$ with underlying heart $\cC$ spiked at $[E_n]$ with respect to $[E]\in \rm{K}_0(\cN)$.
\end{lem}
\begin{proof}
    By \Cref{L:stabconditionfinitelength}, we can construct a stability condition with heart $\cC$ by freely choosing $Z(E[\nu(E)]) \in \bf{H}\cup \bf{R}_{<0}$ for each $E\in \cE$. In particular, we may choose $Z$ such that $Z(E[\nu(E)])$ lies in the first quadrant for all $E\in \cE$. Then, for all $w \in \bigoplus_{i=0}^{n-1} \bf{R}_{\ge 0}\cdot E_i[\nu_i] \setminus \{0\}$, we have 
    \[
        \cot(\theta_{\max}) \le \Re Z(w)/\Im Z(w) \le \cot(\theta_{\min})
    \]
    where $\theta_{\min} = \min_{i=0}^{n-1}\{\arg Z(E_i[\nu_i])\}$ and $\theta_{\max}$ is defined analogously. Thus, 
    \begin{align*}
        \Re Z(E_n) - \frac{\Re Z(w)}{\Im Z(w)}\Im Z(E) -1 > \Re Z(E_n) - \cot (\theta_{\max})\Im Z(E) -1
    \end{align*}
    If necessary, we can take $\Re Z(E_n)$ larger by \Cref{L:stabconditionfinitelength} such that the expression on the right is positive.
\end{proof}

\begin{thm}\label{thm:gluedgeometricKuz}
    Assume that the pair $(\cE,\preceq)$ on $\cN$ is sharp. For $F\in \cD^b(X)$, consider the distinguished triangle 
    \[
        N^F \rightarrow F \rightarrow T^F \rightarrow N^F[1]
    \] 
    where $N^F \in \cN$ and $T^F \in \cT$.
    Let $\tau = (V,\cB) \in \Stab(\cT)$ and $\eta = (W,\cC)\in \Stab(\cN)$, with heart defined as in \Cref{L:heartposet} with respect to the norm $\nu$.
    Assume:\vspace{-2mm}
    \begin{enumerate}
    \renewcommand{\labelenumi}{(\roman{enumi})}
        \item the stability condition $\eta$ is spiked at $E_n$ with respect to $[N^F] \in \rm{K}_0(\cN)$;\vspace{-2mm}
        \item $V(T^F)  = -1$; \vspace{-2mm}
        \item there exists a stability condition $\sigma = (Z,\cA) \in \Stab(X)$ glued from $\tau$ and $\eta$; and\vspace{-2mm}
        \item the object $F$ lies in $\cA$ and satisfies $\mathrm{pr}_{E_n}(F) = 1$.
    \end{enumerate}
    Consider $\phi^+(b_j)$ and $\phi^-(b_j)$ as in \Cref{H:stability}.
    Suppose $N^F$ is efficient in $\cC$, $T^F$ is stable of phase one in $\cB$, and the following holds: 
    \begin{align}
       \phi(E_n)  <\phi(N^F)<\phi^-(b_1) \leq\phi^+(b_1)  < \cdots<\phi^-(b_k)  \leq \phi^+(b_k) <1. \label{eq:hypoKuz}
    \end{align}
    Then $F$ is $\sigma$-stable or $ F = N^F \oplus T^F$.
\end{thm}
\begin{proof}

    Consider a short exact sequence $0\to Y\to F \to Q \to 0$ in $\cA$ and suppose that $F \ne N^F\oplus T^F$. Next, note that there is a fully faithful exact embedding $\cB\hookrightarrow \cA$ such that by \cite{karube2024noncommutative}*{Lem. 3.4}: if $E\in \cB$ and $E'\hookrightarrow E$ is any $\cA$-subobject then $E' \in \cB$. 
    Thus, $T^F$ is simple in $\cA$ since it is simple in $\cB$. Furthermore, the pullback $N^F \cap Y \in \cA$ is non-zero, since otherwise the composite $Y\to T^F$ would be an isomorphism and the induced map $T^F\to F$ would give a splitting $N^F\oplus T^F = F$.

    By \Cref{P:efficientstable}, $N^F$ is stable; also, $\pr_{E_n}(F) = \pr_{E_n}(N^F) = 1$ and so $\pr_{E_n}(Y)\le 1$. If $\pr_{E_n}(Y) = 0$, then $\pr_{E_n}(Y\cap N^F) = 0$ and the argument of \emph{loc. cit.} shows that $N^F\cap Y \hookrightarrow N^F$ is a destabilizing subobject.
    Also, simplicity of $T^F$ implies that $\coker(N^F\cap Y \to Y)$ is zero or $T^F$.

    Next, we will deduce stability of $F$ from conditions (i) and (ii). Note that $\arg Z(N^F\cap Y) < \arg Z(N^F) < \arg Z(F)$, where the first inequality is by stability of $N^F$ and the second inequality is from condition (ii). So, we may assume without loss of generality that $Y\cap N^F \ne Y$ and thus that $\coker(N^F\cap Y\to Y) = T^F$. Thus, we have relations $[Y] = [N^F\cap Y] + [T^F]$ and $[F] = [N^F] + [T^F]$ in $\rm{K}_0(X)$. Since $V(T^F) = -1$, stability of $F$ is equivalent to $\arg Z(Y) = \arg(-1+W(Y\cap N^F)) < \arg(-1+W(N^F)) = \arg Z(F)$ which in turn is equivalent to 
    \[
        \Delta:= \frac{1-\Re W(N^F)}{\Im W(N^F)} - \frac{1-\Re W(Y\cap N^F)}{\Im W(Y\cap N^F)} > 0.
    \]
    Let $v_c = x - y$ where $x = [N^F]$ and $y = [Y\cap N^F]$ and consider the basis $(E_0[w_0],\ldots, E_n[w_n])$ of $\rm{K}_0(\cN)$ where $\cC = \langle E_0[w_0],\ldots, E_n[w_n]\rangle_{\rm{ext}}$. In $\rm{K}_0(\cN)$ we have
    \begin{align*}
        x& = \sum N_i\cdot E_i[w_i]\\
        y& = \sum M_i \cdot E_i[w_i].
    \end{align*}
    Since $v_c$ is the class of $\coker(Y\cap N^F \rightarrow N^F) \in \cC$, we have $N_i\ge M_i\ge 0$ for all $i$ and it follows from efficiency and stability of $N^F$ that $N_n = M_n = 1$. One can now compute that 
    \begin{align*}
        \Delta & = \frac{\Im W(v_c)}{\Im W(x)\Im W(y)}\left(\Re W(y) - \frac{\Re W(v_c)}{\Im W(v_c)} \Im W(y) - 1\right)\\
        & > \frac{\Im W(v_c)}{\Im W(x)\Im W(y)}\left(\Re W(y) - \frac{\Re W(v_c)}{\Im W(v_c)} \Im W(x) - 1\right) \\
        & > \frac{\Im W(v_c)}{\Im W(x)\Im W(y)}\left(\Re W(E_n) - \frac{\Re W(v_c)}{\Im W(v_c)} \Im W(x) - 1\right) > 0. 
    \end{align*}
    The first inequality follows from \Cref{def:spiked}(1), which implies that $\Im W(x) > \Im W(y)$. The second inequality follows from $\Re W(y) > \Re W(E_n)$, which in turn is a consequence of $\Re W(E_i[w_i])> 0$ for any $i$. The final inequality, is a direct consequence of \Cref{def:spiked}(2).
\end{proof}

With an eye toward projective hypersurfaces, we study criteria that allow gluing of paths from semiorthogonal decompositions of the form
\begin{equation}
\label{E:mutatedKuzdecomp}
    \DCoh(X)=\langle \underbrace{E_1,\ldots,E_n}_{\cN^L},\cT,\underbrace{E_{n+1},\dots,E_{m}}_{\cN^R}\rangle 
\end{equation}
where $E_1,\ldots, E_m$ are exceptional objects, and $\cT$ is an admissible subcategory. Consider $\tau_0 = (Z_{\tau_0},\cP_{\tau_0}) \in \Stab(\cT)$ and a path
\begin{equation}\label{eq_pathGL}
    g_t=(M_t,f_t):[0,\infty)\to \GL_2^+(\bf{R})^\sim 
\end{equation}
where $g_0=\id$, $M_t\in \GL_2^+(\bf{R})$, and $f_t:\bf{R}\to \bf{R}$ is an increasing function satisfying $f_t(x+1)=f_t(x)+1$ such that $M_t\cdot \exp(i \pi \phi) = \exp(i \pi f_t(\phi))$ for all $\phi \in \bf{R}$.\footnote{Here, we are using the description of $\GL_2^+(\bf{R})^\sim$ given in \cite{Br07}*{Lem. 8.2}.} This gives a path $(Z_t,\cA_t) = \tau_t:= g_t\cdot \tau_0$ in $\Stab(\cT)$ where $\cA_t=\cP_{\tau_0}(f_t(0),f_t(1)]$. 

We next define a path $\eta_t: [0,\infty) \to \Stab(\cN^L)\times \Stab(\cN^R)$ valued in the product of the regions glued from $(E_1,\ldots, E_n)$ and $(E_{n+1},\ldots, E_m)$. For $d \in \bf{Z}_{\ge 1}$, a region $\cS^d \subset \bf{C}^d$ was defined in \Cref{L:extensionoflog}. We let $\cS^{n,m}:=\cS^n \times \cS^{m-n} \subset \bf{C}^m$ and consider a map $w_t:[0,\infty)\to \cS^{n,m}$. By \cite{Macristabilityoncurves}*{\S 3}, there is a unique path 
\begin{equation}\label{eq_def_stab_path_ex_coll}
    (\eta_t^L,\eta_t^R) : [0,\infty)\to \Stab(\cN^L) \times \Stab(\cN^R)
\end{equation}
such that writing $\eta_t^L = (W_t^L,\cB_t^L)$, one has  $W_t^L(E_i)=\exp(w_{t,i})$ and $E_i$ is $\eta_t^L$-stable of phase $\phi_t^i:=\Im(w_t^i)/\pi$ for each $i=1,\dots,n$. Also, by construction $\cB_t^L=\langle E_1[k_t^1], \dots, E_n[k_t^n]\rangle_{\rm{ext}}$ where $k_t^i:=1-\lceil\phi_t^i\rceil\in \bf{Z}$. An analogous description is available for $\eta_t^R$.

\begin{prop}\label{prop_gluing_paths_Ku}
    Consider \eqref{E:mutatedKuzdecomp} and paths $(\eta_t^L,\eta_t^R)$ as above, $\tau\in \Stab(\cT)$, and a path $g_t$ as in~\eqref{eq_pathGL}. 
    If \vspace{-2mm}
    \begin{enumerate}
        \item $\Hom^{\leq 0}(\cA_t,\cB^R_t)=0$ for any $t$ such that $\lceil \phi_t^i\rceil-\lceil\phi_0^i\rceil \leq0$ for some $i=n+1,\dots, m$,\vspace{-2mm}
        \item $-1<f_t(0)\leq 1$ for any $t\in \bf{R}_{\geq 0}$, and\vspace{-2mm}
        \item $\Hom^{\leq 1}(\cB_t^L,\cB^R_t)=\Hom^{\leq 1}(\cB_t^L,\cA_t)=0$ for any $t$,
    \end{enumerate}
    then $(\eta_t^L,\tau_t,\eta^R_t)$ glue for any $t\in \bf{R}_{\geq 0}$ to $\sigma_t :[0,\infty)\to \Stab(X).$ If, in addition, $\phi_t^n\to -\infty$ and $\phi^{n+1}_t \to \infty$ as $t\to \infty$, then $\sigma$ is quasi-convergent and induces
    \begin{equation}\label{eq_SOD_gluing}
        \DCoh(X) = \langle \cD_1,\dots,\cD_{d},\cT,\cD'_{1},\dots \cD'_{l}\rangle
    \end{equation}
    where $\langle \cD_1,\dots,\cD_{d}\rangle$ coarsens $\langle E_1,\ldots,E_n\rangle$ and $\langle \cD'_1,\dots,\cD'_{l}\rangle$ coarsens $\langle E_{n+1},\ldots,E_m\rangle$.
\end{prop}

\begin{proof}
    We first glue $\tau_t*\eta^R_t$ by applying \Cref{prop:support} and \Cref{C:gluingconditions}: it suffices to verify that $\Hom^{\leq 0}(\cA_t,\cB^R_t)=0$ for any $t\in \bf{R}_{\geq 0}$. By (1), we may assume $\lceil \phi_t^i\rceil-\lceil\phi_0^i\rceil\geq 1$ for all $i=n+1,\dots, m$. Consider $A\in \cA_t=\cP(f_t(0),f_t(1)]$ and the unique $n\in \bf{Z}$ such that $A[n]\in \cA_0$. By (1) applied to $t = 0$, if $n+k_t^i-k_0^i\leq 0$ then
    \begin{equation*}
    \Hom^{\leq 0}(A,E_i[k_t^i]) =\Hom^{\leq n+k_t^i-k_0^i}(A[n],E_i[k_0^i])=0 
    \end{equation*}
    since indeed $A[n]\in \cA_0$ and $E_i[k_0^i]\in \cB^R_0$. It is also clear that
\[
n=
\begin{cases}
    -\lceil f_t(0)\rceil &\text{ if }  \phi(A)>\lceil f_t(0)\rceil\\
    -\lceil f_t(0)\rceil+1 &\text{ if }  \phi(A)\leq \lceil f_t(0)\rceil\\
\end{cases}
\]
where $\phi(A)$ is the $\tau_0$-phase of $A$. Hence the condition $n+k_t^i-k_0^i\leq 0$ translates into 
\[
\lceil \phi_t^i\rceil-\lceil\phi_0^i\rceil\geq 
\begin{cases}
    -\lceil f_t(0)\rceil &\text{ if }  \phi(A)>\lceil f_t(0)\rceil\\
    -\lceil f_t(0)\rceil +1&\text{ if }  \phi(A)\leq \lceil f_t(0)\rceil.
\end{cases}
\]
The left hand side of the inequality is always $\ge 1$, and the condition (2) ensures that the right hand side is always $\le 1$.
By the third condition and \cref{prop:support} we can glue $\eta_t^L$ and $\tau_t*\eta^R_t$.
\end{proof}

\begin{rem}\label{rem_gluing_pathKu}
    If we replace conditions (1), (2), (3) of \Cref{prop_gluing_paths_Ku} with the hypothesis: for all $\tau$-stable object $E\in \cT$, $\liminf_{t\to\infty} \phi(E) - \phi(E_n) = \infty = \liminf_{t\to\infty} \phi(E_{n+1})-\phi(E)$, then by a similar argument we obtain a quasi-convergent path $\sigma_t=\eta_t*\tau_t*\eta'_t$ for all $t\gg 1$ with induced semiorthogonal decomposition as in \eqref{eq_SOD_gluing}.
\end{rem}

\subsection{Geometric stability conditions on cubic threefolds}\label{section_stab_cuibc3}
In this section, we construct geom\-etric stability conditions on cubic threefolds using the gluing technique of \cite{Collins_Polischuk_2010}. 

\subsubsection*{Stability conditions on the Kuznetsov components on cubic threefolds}
We begin by summ\-arizing without proof some relevant facts about Kuznetsov components of cubic threefolds.

Let $Y\hookrightarrow\bf{P}^4$ denote a cubic threefold and $H$ the restriction of the hyperplane class to $Y$. The pair $(\cO_Y,\cO_Y(H))$ is exceptional and we set $\cN_Y := \langle \cO_Y ,\cO_Y(H)\rangle$. We call $\Ku(Y) := \cN_Y^{\perp}$ the \textit{Kuznetsov component} of $Y$ so that $\DCoh(Y) = \langle \Ku(Y),\cN_Y\rangle$. First, we consider the projection of a skyscraper sheaf $\cO_x$ for a closed point $x \in Y$ to $\Ku(Y)$, which can be written as $K^x = \bf{L}_{\cO_Y}\bf{L}_{\cO_Y(H)}(\cO_x)[-2]$. By the definition of a semiorthogonal decomposition, there is a following exact triangle:
    \[
        N^x \rightarrow \cO_x \rightarrow K^x[2] \rightarrow N^x[1],
    \]
    where $N^x \in \cN_Y$.  Next, we study $K^x$ and $N^x$.

\begin{prop}
    In the above notation, $K^x$ is a $\mu_H$-stable sheaf on $X$ and fits into an exact triangle
    \[
        K^x \rightarrow \cO_Y^{\oplus 4} \rightarrow I_x(H) \rightarrow K^x[1].
    \] 
    The object $N^x$ satisfies $\cH^i(N^x) = 0$ for $i\ne -1,0$ and fits into an exact triangle
    \[
         \cO_Y(H) \rightarrow N^x \rightarrow \cO_Y^{\oplus 4}[1] \rightarrow \cO_Y(H)[1]
    \]
    in $\DCoh(Y)$. Furthermore, $\ch(K^x) = (3,-H,-H^2/2,H^3/6)$ and $\ch(N^x) = (-3,H,H^2/2,H^3/6)$. 
\end{prop}

\begin{proof}
    For the first triangle, by \cite{bayer2024desingularization}*{Cor. 5.2} we only need to show that $\Cone(\mathcal{O}_Y^{\oplus 4} \to I_x(H))[-1] \cong K^x$. This directly follows from the relation $\bf{L}_{\cO_H(H)}(\mathcal{O}_x) = I_x(H)[1]$. The second exact triangle and the following claims are a consequence of the first triangle and the octahedral axiom.
\end{proof}

\begin{prop}{\cite{kuznetsov2004derived}*{Lem. 4.7}} For any line $l \subset Y$, the ideal sheaf $I_l$ is in $\Ku(Y)$.
\end{prop}

The numerical Grothendieck group of $\rm{Ku}(Y)$ is the lattice generated by $[I_l] \text{ and } [\cS_{\Ku(Y)}(I_l)]$, where $\cS_{\Ku(Y)}$ is the Serre functor of $\Ku(Y)$. Consider the lattice  
\[
    \Lambda = \im (\rm{K}_0(\Ku(Y))\rightarrow \rm{K}_0^{\rm{top}}(Y)_{\bf{Q}}).
\]
As in \Cref{E:latticeexample}, $\Lambda$ is also generated by $\ch(I_l)$ and $\ch(\cS_{\Ku(Y)}(I_l))$ by \cite{bernardara2012categorical}. In what follows, stability conditions on $\Ku(Y)$ satisfy the support property with respect to $\Lambda$. Next, we review the const\-ruction of a family of stability conditions $(Z(\alpha,\beta),\cA(\alpha,\beta))\in \Stab(\Ku(Y))$ parametrized by $(\alpha,\beta)\in \bf{R}_{>0}\times\bf{R}$ following \cite{Bayer2017StabilityCO}. The central charge is 
\[
    Z(\alpha,\beta) := -H^2\ch^{\beta}_1 +\mathtt{i}\left(-\frac{1}{2}\alpha H^3 \ch_0^{\beta}+H\ch_2^{\beta}\right).
\]
while the heart $\mathcal{A}(\alpha,\beta)$ is defined by a sequence of two tilts. First, we recall the slope function 
\begin{equation}
\label{E:muH}
\mu_H(E) = 
\begin{cases}
    \frac{\ch_1(E)H^2}{\rk(E)} & \rk(E) \not =0,\\
    \infty &  \rk(E) =0,
\end{cases}
\end{equation}
for $E \in \Coh(Y)$; $E$ is called $\mu_H$-semistable if for any $0 \ne F \subsetneq E$, $\mu_H(F) \le \mu_H(E/F)$. It is $\mu_H$-stable if the inequality is strict. The slope function $\mu_H$ defines a torsion pair $(\cT^{\beta},\cF^{\beta})$ for any $\beta\in\bf{R}$ in $\Coh(Y)$ by:
\begin{align*}
    \cT^{\beta} &= \langle E \colon E \text{ is $\mu_H$-stable with }\mu_H(E) > \beta \rangle,\\
    \cF^{\beta} &= \langle F \colon F \text{ is $\mu_H$-stable with }\mu_H(F) \leq \beta \rangle
\end{align*}
It follows from the results of \cite{happel1996tilting} that there is a new heart $\Coh^\beta(Y) := \langle \cT^{\beta},\cF^{\beta}[1] \rangle$ called the \textit{tilting (or tilted) heart} with respect to $\mu_H = \beta$. Next, consider the slope function
\[
    \nu_{\alpha,\beta}(E) = \frac{-\frac{1}{2}\alpha^2 H^3 \mathrm{ch}_0^{\beta}(E)+H\mathrm{ch}_2^{\beta}(E)}{H^2\mathrm{ch}^{\beta}_1(E)}
\]
defined on $\mathrm{Coh}^{\beta}(Y)$. Using $\nu_{\alpha,\beta}$, we can define stability for objects of $\Coh^\beta(Y)$. An object $E\in \Coh^\beta(Y)$ is $\nu_{\alpha,\beta}$-semistable if for all $0\ne F \subsetneq E$ one has $\nu_{\alpha,\beta}(F)\le \nu_{\alpha,\beta}(E/F)$. It is $\nu_{\alpha,\beta}$-stable if the inequality is strict -- see {\cite{Bayer2011BridgelandSC}*{Defn. 3.2.3}. 

As before, there is a tilting heart $\cA_{\alpha,\beta}$ on $\DCoh(Y)$ defined with respect to $\nu_{\alpha,\beta} = 0$ and we let $\cA(\alpha,\beta) := \cA_{\alpha,\beta} \cap \Ku(Y)$.

\begin{thm}{\cite{Bayer2017StabilityCO}*{Thm. 6.8}}
For every pair $(\alpha,\beta)$ in 
\[
    V= \left\{ (\alpha,\beta) \in \bf{R}_{>0} \times \bf{R} \colon -\frac{1}{2} \leq \beta < 0,\alpha < -\beta \text{ or } -1 < \beta <  -\frac{1}{2}, \alpha \leq 1+ \beta \right\},
\]
the pair $\sigma_{\alpha,\beta} = (Z(\alpha,\beta),\cA(\alpha,\beta))$ defines a stability condition on $\Ku(Y)$.
\end{thm}

\subsubsection*{Glued geometric stability conditions on cubic threefolds}
First, we construct a glued stability condition whose heart contains as simple objects $K^x[2],\cO_Y[1]$ and $\cO_Y(H)$. Let $\cB\subset \cN_Y$ denote the heart generated by $\cO_Y[1]$ and $\cO_Y(H)$.
The key observation is that if $\cA_{\alpha,\beta}$ contains $K^x[1],\cO_Y[1]$ and $\cO_Y(H)$, then $\Hom^{\leq 0}(\cA(\alpha,\beta)[1],\cB) = 0$.

\begin{lem}\label{lem:line_bundle}
    For any $(\alpha,\beta) \in V$, $\cO_Y(H)$ and $\cO_Y[1]$ lie in $\cA_{\alpha,\beta}$.
\end{lem}

\begin{proof}
    All line bundles $L$ are $\mu_H$-stable, as for any proper subsheaf $F\subsetneq L$ one has $\mu_H(F) < \mu_H(L/F)$ since $L/F$ is a torsion sheaf. Since cubic threefolds have Picard number one, it follows from \cite{Bayer2011BridgelandSC}*{Prop. 7.4.1} and the subsequent discussion that $\cO_Y$ and $\cO_Y(H)$ are $\nu_{\alpha,\beta}$-stable. The result now follows from $\mu_H(\cO_Y) = 0$, $\mu_H(\cO_Y(H)) = H^3$, $\nu_{\alpha,\beta}(\cO_Y) = (\alpha^2-\beta^2)/2\beta <0$, and $\nu_{\alpha,\beta}(\cO_Y(H)) = (\alpha^2 - (1-\beta)^2)/2(\beta-1) > 0$.
\end{proof}

\begin{lem}
    For any $x \in Y$, the sheaf $I_x(H)$ is stable with respect to $\mu_H$ and $\nu_{\alpha,\beta}$.
\end{lem}

\begin{proof}
    There is an exact sequence $0 \rightarrow I_x(H) \rightarrow \cO_Y(H) \rightarrow \cO_x \rightarrow 0$ in $\Coh(Y)$ and thus $\mu_H(\cO_Y(H)) = \mu_H(I_x(H))$. 
    Suppose given $0\ne F \subsetneq I_x(H)$ with $\mu_H(F) > \mu_H(I_x(H)/F)$.
    The diagram
    \[
        \begin{tikzcd}
	       & F & F \\
	       0 & I_x(H)& \cO_Y(H) &\cO_x & 0 \\
	       0 & I_x(H)/F & \cO_Y(H)/F & \cO_x & 0.
	       \arrow[equal, from=1-2, to=1-3]
	       \arrow[hook, from=1-2, to=2-2]
	       \arrow[hook, from=1-3, to=2-3]
	       \arrow[from=2-1, to=2-2]
	       \arrow[hook, from=2-2, to=2-3]
	       \arrow[from=2-2, to=3-2]
	       \arrow[from=2-3, to=2-4]
	       \arrow[from=2-3, to=3-3]
	       \arrow[from=2-4, to=2-5]
	       \arrow[equal, from=2-4, to=3-4]
	       \arrow[from=3-1, to=3-2]
	       \arrow[from=3-2, to=3-3]
	       \arrow[from=3-3, to=3-4]
	       \arrow[from=3-4, to=3-5]
        \end{tikzcd}
    \]
    implies that $\mu_H(\cO_Y(H)/F) = \mu_H(I_x(H)/F)$, which contradicts stability of $\cO_Y(H)$. Thus, $I_x(H)$ is $\mu_H$-stable in $\Coh(Y)$. Since $\ch_2^{\beta}(\cO_x) = 0$, the $\nu_{\alpha,\beta}$-stability of $I_x(H)$ follows similarly.
\end{proof}

By \cite{bayer2024desingularization}*{Lem. 8.4 and Prop. 8.10}, $K^x$ is $\sigma_{\alpha,\beta}$-stable for any $(\alpha,\beta)\in V$.
Set 
\[
V_+ =\{(\alpha,\beta) \in V \colon \beta \geq -1/3\}, \quad V_-=\{(\alpha,\beta) \in V \colon \beta < -1/3\}
\]
Since $\nu_{\alpha,\beta}(K^x) = {\frac{3\alpha^2 - 3\beta^2 - 2\beta + 1}{2(1+3\beta)}}$, $K^x$ lies in $\cA_{\alpha,\beta}$ if $(\alpha,\beta) \in V_+$.
To construct a geometric stability condition on $\DCoh(Y)$, we consider $z\cdot \sigma(\alpha,\beta)$ for $z\in \bf{C}$.

\begin{thm}\label{thm_glued_geom_stab_cubic3}
    Let $(\alpha,\beta)\in V_+$ and $z\in \bf{C}$ be given such that $K^x[2]$ is $z\cdot \sigma(\alpha,\beta)$-stable of phase one and take $\tau = (W,\cB)\in \Stab(\cN_Y)$ with $\cB = \langle \cO_Y(H),\cO_Y[1]\rangle_{\rm{ext}}$.
    Then, there exists a stability condition $\varrho := (z\cdot \sigma(\alpha,\beta))*\tau\in \Stab(X)$. Furthermore, we can choose $\tau$ such that 
    \[
        \phi(K^x[2])>\phi(\cO_Y[1]) > \phi(\cO_Y(H))
    \]
    and such that $\tau$ is spiked at $\cO_Y(H)$ with respect to $[N^x]$; in this case, $\cO_x$ is $\varrho$-stable for any $x \in Y$.
\end{thm}

\begin{proof}
    Let $I = \Im(z)/\pi$ and let $\cA_z$ be the heart underlying $z \cdot \sigma(\alpha,\beta)$, $\cP$ the slicing of $\sigma(\alpha,\beta)$, and $\cP_z$ the slicing of $z\cdot\sigma(\alpha,\beta)$. Since $K^x \in \cA(\alpha,\beta)$ and $K^x[2] \in \cP_z(1) = \cP(1+I)$, we have that $2 < 1+I \leq 3$.
    Therefore, $\cA_z = \cP(I,1+I] \subset \cP(1,3] = \langle \cA(\alpha,\beta)[1],\cA(\alpha,\beta)[2]\rangle_{\rm{ext}}$. To glue $z\cdot \sigma(\alpha,\beta)$ and $\tau$, it suffices to show $\Hom^{\leq 0}(\cA_z,\cB) = 0$ by \Cref{C:gluingconditions}. Since $\cA_z \subseteq \langle \cA(\alpha,\beta)[1],\cA(\alpha,\beta)[2]\rangle_{\rm{ext}}$, 
    it is enough to show that $\Hom^{\leq 0}(\cA(\alpha,\beta)[1],\cB) = 0$. By \Cref{lem:line_bundle}, $\cB = \cA_{\alpha,\beta}\cap\cN_Y$ and since $\Hom^{\le 0}(\cA_{\alpha,\beta}[1],\cA_{\alpha,\beta}) = 0$ the claim follows. So, there exists $\varrho$ as in the statement.

    Next, we will apply \Cref{thm:gluedgeometricKuz} to show that $\varrho$ can be taken to be geometric: we verify the conditions now, taking $F = \cO_x$. Let $x\in Y$ be given and consider the triangle $N^x\to \cO_x\to K^x[2]\to N^x[1]$ coming from $\DCoh(Y) = \langle \Ku(Y),\cN_Y\rangle$. Note that here the norm function on $\cN_Y = \langle\cO_Y,\cO_Y(H)\rangle$ sends $\cO_Y\mapsto 1$ and $\cO_Y(H)\mapsto 0$.
    
    By \Cref{L:spikedexist}, we can choose $\tau$ to be spiked with respect to $\cO_Y(H)$ without affecting the gluing conditions above. By construction, the central charge of $z\cdot \sigma(\alpha,\beta)$ applied to $K^x[2]$ gives $-1$. Next, the triangle $N^x\to \cO_x\to K^x[2]\to N^x[1]$ gives a short exact sequence in the heart $\cU$ underlying $\varrho$ and in particular $N^x$ is a subobject of $\cO_x$. Thus, we have 
    \[
        0\to \Hom_{\cU}(\cO_Y[1],N^x) \to \Hom_{\cU}(\cO_Y[1],\cO_x) = \Hom_{\DCoh(Y)}(\cO_Y[1],\cO_x) = 0,
    \]
    and it follows that $N^x$ is efficient with respect to $\{\cO_Y,\cO_Y(H)\}$. Finally, we can deform $\tau$ in the locus of $\Stab(\cN_Y)$ where the underlying heart is $\mathcal{B}$ until the phases satisfy $\phi(\mathcal{O}_Y(H)) < \phi(N^x) < \phi(\mathcal{O}_Y[1]) < 1$. Applying \Cref{thm:gluedgeometricKuz} then yields the result.
\end{proof}

\begin{rem}\label{rem:glu_cubic3}
Let us observe that the gluing in \Cref{thm_glued_geom_stab_cubic3} still holds if $\tau\in \Stab(\cN_Y)$ is chosen with heart $\cB=\langle\ko_Y(H)[-n],\ko_Y[-m]\rangle_{\mathrm{ext}}$ with $m\geq n\geq0$.
\end{rem}

\subsection{Geometric stability conditions on cubic fourfolds}\label{section_stab_cuibc4}
In this section we construct geometric stability conditions on cubic fourfolds which does not contain a plane.
We start by studying the projection of a skyscraper sheaf of a point in the Kuznetsov component. 
We continue by recalling briefly stability conditions on the Kuznetsov component of a cubic fourfold and finally, by gluing, we obtain  some geometric stability conditions, see \Cref{thm_geomstab_cubic4_noplane}.
\begin{notn}\label{notation:cohomology}
    Let $\mathcal{A}$ be the heart of a bounded t-structure on $\mathcal{D}$. For $a,b\in \bf{Z}\cup \{\pm \infty\}$ with $a < b$, we define a full subcategory $\mathcal{A}[a, b]$ of $\mathcal{D}$ by 
    \[
        \mathcal{A}{[a,b]} := \langle \mathcal{A}[n] :n\in \bf{Z} \text{ and }  a \leq n \leq b \rangle_{\rm{ext}}.
    \]
When $\mathcal{A}$ is the category of coherent sheaves, we simply write $[a, b]$ for $\mathcal{A}{[a,b]}$.

Specifically, for an object $E \in \mathcal{D}(X)$, the condition $E \in [a, b]$ is equivalent to $\mathcal{H}^i(E) = 0$ for $i < -b$ and $i > -a$.
\end{notn}
\subsubsection*{Projection of skyscraper sheaves}
Let $X\subset \bf{P}^5$ be a cubic fourfold.
We denote by $H$ both the hyperplane class in $\bf{P}^5$ and its pullback to $X$. The line bundles $(\cO_X,\cO_X(H),\cO_X(2H))$ define an exceptional triple; we put $\cN_X := \langle \cO_X,\cO_X(H),\cO_X(2H) \rangle$ and $\Ku(X) = \cN_X^\perp$. Thus, we have $\DCoh(X) = \langle \Ku(X),\cN_X\rangle$ and for each $x\in X$ an exact triangle 
\[
    N^x \rightarrow \cO_x \rightarrow K^x[3] \rightarrow N^x[1]
\]
where $N^x\in \cN_X$ and $K^x \in \mathrm{Ku}(X)$. 
We will compute $N^x$ and $K^x$ explicitly.
We note that by definition $K^x = \bf{L}_{\cO_X}\bf{L}_{\cO_X(H)}\bf{L}_{\cO_X(2H)}(\cO_x)[-3]$.

\begin{lem}\label{lem:surjh^0}
    Let $X$ be a cubic fourfold and $x\in X$. The evaluation map 
    \[
        \ev^0_X : \H^0(\cO_X(pH)) \otimes \H^0(I_{x,X}(qH)) \rightarrow \H^0(I_{x,X}((p+q)H)) 
    \]
    is surjective for all $p,q\in \bf{N}$.
\end{lem}
\begin{proof}
    First, we show the surjectivity of $\ev_{\bf{P}^5}^0$, which is defined analogously. Up to changing coordinates, we can assume $x\in \bf{P}^5$ is $x  = [0:\dots :1]$. The vector space $\H^0(\cO_{\bf{P}^5}(pH))$ is generated by homogeneous polynomials of degree $p$, and $\H^0(I_{x,\bf{P}^5}(qH)) \subseteq \H^0(\cO_{\bf{P}^5}(qH))$ consists of the degree $q$ polynomials vanishing at $x$. Thus, the monomials of degree $p+q$ in $x_0,\ldots, x_5$ excluding $x_5^{p+q}$ form a basis for $\H^0(I_{x,\bf{P}^5}((p+q)H))$. Since $\ev_{\bf{P}^5}^0$ is given by $v\otimes w  \mapsto vw$, it follows that $\ev_{\bf{P}^5}^0$ is surjective. Next, we show the desired statement.
    For any $k \in \bf{Z}$, there is the  following commutative diagram:
    \[
        \begin{tikzcd}
	& \cO_{\bf{P}^5}((k-3)H) & \cO_{\bf{P}^5}((k-3)H) \\
	0 & I_{x,\bf{P}^5}(kH) & \cO_{\bf{P}^5}(kH) & \cO_x & 0 \\
	0 & I_{x,X}(kH) & \cO_X(kH) & \cO_x & 0
	\arrow[from=1-2, to=1-3]
	\arrow[from=1-2, to=2-2,"\times f", hook]
	\arrow[from=1-3, to=2-3,"\times f", hook]
	\arrow[from=2-1, to=2-2]
	\arrow[from=2-2, to=2-3]
	\arrow[from=2-2, to=3-2, two heads]
	\arrow[from=2-3, to=2-4]
	\arrow[from=2-3, to=3-3, two heads]
	\arrow[from=2-4, to=2-5]
	\arrow[equal, from=2-4, to=3-4]
	\arrow[from=3-1, to=3-2]
	\arrow[from=3-2, to=3-3]
	\arrow[from=3-3, to=3-4]
	\arrow[from=3-4, to=3-5]
        \end{tikzcd}
    \]
    with exact columns: here $f \in \H^0(\cO_{\bf{P}^5}(3H))$. Since $\H^1(\cO_{\bf{P}^5}((k-3)H)) = 0$ for any $k$, there are surjections $\H^0(I_{x,\bf{P}^5}(kH))\twoheadrightarrow \H^0(I_{x,X}(kH))$ and $\H^0(\cO_{\bf{P}^5}(kH))\twoheadrightarrow \H^0(\cO_X(kH))$.
    Thus, the com\-mutative diagram:
    \[
    \begin{tikzcd}
	\H^0(\cO_{\bf{P}^5}(pH)) \otimes \H^0(I_{x,\bf{P}^5}(qH)) & \H^0(I_{x,\bf{P}^5}((p+q)H))  \\
	\H^0(\cO_X(pH)) \otimes \H^0(I_{x,X}(qH)) & \H^0(I_{x,X}((p+q)H)).
	\arrow["\ev_{\bf{P}^5}^0",two heads, from=1-1, to=1-2]
	\arrow[two heads, from=1-1, to=2-1]
	\arrow[two heads, from=1-2, to=2-2]
	\arrow["\ev_X^0",from=2-1, to=2-2]
\end{tikzcd} 
    \]  
    implies that $\ev_X^0$ is surjective.
\end{proof}

\begin{prop}\label{prop:exactseq}
    There exist exact triangles:
    \begin{align*}
        I_x(2H) \rightarrow \cO_X(2H) \xrightarrow{\ev_1} \cO_x \to I_x(2H)[1] 
        \\
        F_x(H) \rightarrow \cO_X(H)^{\oplus 5} \xrightarrow{\ev_2} I_x(2H)  \to F_x(H)[1] 
        \\
        K^x\rightarrow \cO_X ^{\oplus 10} \xrightarrow{\ev_3} F_x(H) \to K^x[1] 
    \end{align*}
    where $\ev_2$ is obtained by applying $-\otimes \cO_X(H)$ to the evaluation map $\H^0(I_x(H))\otimes \cO_X \rightarrow I_x(H)$, and $\Cone(\ev_2) =: F_x(H)[1]$. The map $\ev_3$ is similarly obtained from $\H^0(F_x(H)) \otimes \cO_X \rightarrow F_x(H)$. As a consequence of these exact triangles, we obtain
    \[
        \bf{L}_{\cO_X(2H)}(\cO_x) = I_x(2H)[1], \;\; \bf{L}_{\cO_X(H)}(I_x(2H)) = F_x(H)[1], \;\;\bf{L}_{\cO_X(H)}(F_x(H)) = K^x[1].
    \]
\end{prop}
\begin{proof}
    The first triangle is obtained by tensoring the ideal sheaf exact sequence $0\to I_x \to \cO_X \to \cO_x\to 0$ by $\cO_X(2H)$. Since $\RHom(\cO_X(2H),\cO_x) = \bf{C}[0]$, it follows that $\bf{L}_{\cO_X(2H)}(\cO_x) = I_x(2H)[1]$. Next, by \cite{ouchi2017lagrangian}*{Lemmas 4.2 \& 4.4} there is an exact triangle 
    \[
        F_x \rightarrow \H^0(X,I_x(H))\otimes\cO_X \xrightarrow{\ev_2} I_x(H)  \to F_x[1]
    \] 
    with $F_x \in \langle \cO_X\rangle^{\perp}$. Tensoring by $\cO_X(H)$ gives the second exact triangle. Next, we compute $\RHom(\cO_X,F_x(H))$. Applying $\RHom(\cO_X,-)$ to second exact triangle yields an exact sequence:
    \[
        0 \rightarrow \Hom^0(\cO_X,F_x(H)) \rightarrow  \Hom^0(\cO_X,\cO_X(H))\otimes \H^0(X,I_x(H)) \xrightarrow{\ev_2^0}  \Hom^0(\cO_X,I_x(2H)).
    \]
    One can compute that $\H^0(X,\cO_X(H)) = \bf{C}^6$ and $\H^0(X,I_{x}(2H)) = \bf{C}^{20}$.
    By \Cref{lem:surjh^0} the morphism $\ev_2^0$ is surjective. 
    Moreover $\H^0(I_x(H)) = \bf{C}^5$, thus, by the exact sequence, we get \[
    \dim \H^0(F_x(H)) = 5 \dim \H^0(\cO_X(H))- \dim\H^0(I_{x}(2H)) = 10.
    \]
    By the definition of left mutation functors, we have $\bf{L}_{\cO_X(H)}(I_x(2H))  = F_x(H)[1].$
    Thus, by definition, $K^x = \bf{L}_{\cO_X(H)}(F_x(H))[-1]$, and we obtain the third exact triangle.

\end{proof}

\begin{cor}\label{cor_Nx}
    There exists an object $M$ of $\DCoh(X)$ and distinguished triangles:
    \[
        \begin{tikzcd}
            \cO_X(2H) \arrow[r]& M\arrow[d] \arrow[r]& N^x\arrow[d]\\
            &\cO_X(H)^{\oplus 5}[1]\arrow[ul,dashed,"+1"]&\cO_X^{\oplus 10}[2].\arrow[ul,dashed,"+1"]
        \end{tikzcd}
    \]
\end{cor}

\begin{proof}
Let $M'[1] = \Cone(N^x \to I_x(2H)[1])$, where $N^x\to I_x(2H)[1]$ is the composite $N^x\to \cO_x\to I_x(2H)[1]$. Next, we observe that there is a commutative diagram with exact rows and columns:
\[
\begin{tikzcd}
	   K^x[2] & K^x[2] \\
	   M' & N^x & I_x(2H)[1]  \\
	   \cO_X(2H) & \cO_x & I_x(2H)[1] 
	\arrow[equal,from=1-1, to=1-2]
	\arrow[from=1-1, to=2-1,dashed,"a",swap]
	\arrow[from=1-2, to=2-2]
	\arrow[from=2-1, to=2-2]
	\arrow[from=2-1, to=3-1,dashed,"c",swap]
	\arrow[from=2-2, to=2-3]
	\arrow[from=2-2, to=3-2]
	\arrow[equal, from=2-3, to=3-3]
	\arrow[from=3-1, to=3-2]
	\arrow[from=3-2, to=3-3]
\end{tikzcd}
\]
The bottom horizontal triangle comes from \Cref{prop:exactseq}, while the middle horizontal triangle is the definition of $M'$. The middle column is the definition of $K^x$ and $N^x$. The bottom right square commutes by definition of $N^x\to I_x(2H)[1]$ and this gives the map $c$ and the bottom left commutative square. Finally, the octahedral axiom applied to the morphisms $N^x\to \cO_x$, $\cO_x\to I_x(2H)[1]$ and their composite $N^x\to I_x(2H)[1]$ gives the left column.

By semiorthogonality, $\Hom(\cO_X(2H),K^x[3]) = 0$ so $M' = K^x[2] \oplus \cO_X(2H)$.
We define $f: N^x \to F_x(H)[2]$ as the composite $N^x \to I_x(2H)[1] \to F_x(H)[2].$ Since $\Hom(N^x,K^x[3]) = 0$, $f$ factors through $\mathcal{O}_X^{\oplus 10}[2]$. Let $M[1] := \Cone(N^x \to \cO_X^{\oplus 10}[2])$.
Then, \Cref{prop:exactseq} gives the following commutative diagram with exact rows and columns:
\[
\begin{tikzcd}
	\cO_X(2H) & M' & K^x[2]  \\
	   M & N^x & \cO_X^{10}[2] \\
	\cO_X^{\oplus 5}(H)[1] & I_x(2H)[1] & F_x(H)[2].
	\arrow[from=1-1, to=1-2]
	\arrow[from=1-1, to=2-1]
	\arrow[from=1-2, to=1-3]
	\arrow[from=1-2, to=2-2]
	\arrow[from=1-3, to=2-3]
	\arrow[from=2-1, to=2-2]
	\arrow[from=2-1, to=3-1]
	\arrow[from=2-2, to=2-3]
	\arrow[from=2-2, to=3-2]
	\arrow[from=2-3, to=3-3]
	\arrow[from=3-1, to=3-2]
	\arrow[from=3-2, to=3-3]
    \arrow[from=2-2, to=3-3,"f"]
\end{tikzcd}
\]
The triangles $M\to N^x\to \cO_X^{\oplus 10}[2]$ and $\cO_X(2H)\to M \to \cO_X(H)^{\oplus 5}[1]$ imply the result.
\end{proof}

\subsubsection*{Review: stability conditions on the Kuznetsov components of cubic fourfolds} 
We define $\H_{\rm{alg}}(\Ku(X),\bf{Z})$ to be the image of the natural map $v:\rm{K}_0(\Ku(X))\rightarrow \rm{K}^{\rm{top}}_0(X)$ -- see \cite{Bayer2017StabilityCO}*{Prop/Defn. 9.5}. 
We recall the construction of a heart of a bounded t-structure on $\Ku(X)$. Fix a line ${\ell} \subset X$ and let $\sigma:\widetilde{X}\rightarrow X$ be the blow up of $X$ along ${\ell}$ with exceptional divisor $D$. We let $\widetilde{\bf{P}}^5 = \Bl_\ell(\bf{P}^5)$.

\[
\begin{tikzcd}
	D & {\widetilde{X}} & {\widetilde{\bf{P}}^5} \\
	{\ell} & X && {\bf{P}^3}
	\arrow[from=1-1, to=1-2]
	\arrow[from=1-1, to=2-1]
	\arrow[from=1-2, to=1-3,"\alpha"]
	\arrow[from=1-2, to=2-2,"\sigma"]
	\arrow[from=1-3, to=2-4,"q"]
        \arrow[from=1-2, to= 2-4,"\pi",swap]
	\arrow[from=2-1, to=2-2]
\end{tikzcd}
\]
We denote by $H$ (resp. $h$) the class of the hyperplane in $\bf{P}^5$ (resp. $\bf{P}^3$).
The variety $\widetilde{\bf{P}}^5$ is isomorphic to the rank 2 projective bundle $\bf{P}(\cO_{\bf{P}^3}^2 \oplus \cO_{\bf{P}^3}(-h))\to \bf{P}^3$. Consider the even part $\cB_0$ and the odd part $\cB_1$ of the Clifford algebra of $\pi$.

We define $\cB_0$-bimodules:
\[
\cB_{2j} := \cB_0\otimes \cO(jh), \ \cB_{2j+1} := \cB_1\otimes \cO(jh)
\]
for $j \in \bf{Z}$. We consider functors $\Phi : \cD^b(\bf{P}^3,\cB_0) \rightarrow \cD^b(\widetilde{X})$ and $\Psi : \cD^b(\widetilde{X}) \rightarrow  \cD^b(\bf{P}^3,\cB_0)$ defined by
\begin{align*}
    \Phi(-)&=\pi^*(-) \otimes \cE^{\prime} \\
    \Psi(-)&=\pi_*(- \otimes \cO_{\widetilde{X}}(h)\otimes \cE[1]).
\end{align*}
where $\cE'$ and $\cE$ are rank two vector bundle on $\bf{P}^3$ characterized by the following exact triangles:
\begin{align*}
    0\rightarrow q^*\cB_0(-2H) \rightarrow q^*\cB_1(-H) \rightarrow \alpha_*\cE^{\prime} \rightarrow 0\\
     0\rightarrow q^*\cB_{-1}(-2H) \rightarrow q^*\cB_0(-H) \rightarrow \alpha_*\cE \rightarrow 0.  
\end{align*}
By \cites{auel2014fibrations,kuznetsov2008derived}, $\Phi\colon \cD^b(\bf{P}^3,\cB_0) \rightarrow \cD^b(\widetilde{X}) $ is fully faithful  and $\Psi \dashv \Phi$. We define the projection functor $\pr_L:=  \bf{L}_{\cO_X}\bf{L}_{\cO_X(H)}\bf{L}_{\cO_X(2H)} \colon \cD^b(X) \rightarrow \Ku(X)$ and consider the commutative diagram of functors:
\[  
\begin{tikzcd}
	\cD^b(X) & \cD^b(\widetilde{X}) & \cD^b(\bP^3,\cB_0) \\
	\Ku(X)
	\arrow["{\sigma^*}", from=1-1, to=1-2]
	\arrow[""{name=0, anchor=center, inner sep=0}, "{\mathrm{pr}_L}"', shift left, curve={height=12pt}, from=1-1, to=2-1]
        \arrow[""{name=R, anchor=center, inner sep=0}, "{\mathrm{pr}_R}", shift right, curve={height=-12pt}, from=1-1, to=2-1]
	\arrow[""{name=1, anchor=center, inner sep=0}, "\Psi", from=1-2, to=1-3]
        \arrow[""{name=2, anchor=center, inner sep=0}, "\Phi", shift right, curve={height=-10pt}, from=1-3, to=1-2]
	\arrow[""{name=3, anchor=center, inner sep=0},"\iota"', from=2-1, to=1-1]
	\arrow["\Xi"', curve={height=15pt}, from=2-1, to=1-3]
	\arrow["\dashv"{anchor=center}, draw=none, from=0, to=3]
        \arrow["\dashv"{anchor=center}, draw=none, from=3, to=R]
	\arrow["\dashv"{anchor=center, rotate=-90}, draw=none, from=2, to=1]
\end{tikzcd}
\]
where $\iota \dashv \pr_R$. 
\begin{thm}{\cite{Bayer2017StabilityCO}*{Prop. 7.7.}}\label{thm:ffofXi}
The functor $\Xi$ is fully faithful.
Moreover,
\[
\cD^b(\bP^3,\cB_0) = \langle \Xi(\Ku(X))  ,\Psi(\cO_{\tilde{X}}(h-H)),\Psi(\cO_{\tilde{X}}(H)),\Psi(\cO_{\tilde{X}}(2h-H))\rangle
\]
\end{thm}

The forgetful functor $\mathrm{Forg}\colon \cD^b(\bP^3,\cB_0) \rightarrow \cD^b(\bP^3)$ forgets the $\cB_0$-module structure; the $\beta$-twisted Chern characters of $E \in \cD^b(\bP^3,\cB_0)$ are defined by:
\[
    \ch_{\cB_0}^\beta(E)  = e^{-\beta H}\cup \ch(\mathrm{Forg}(E)) \cup \left(1 - \tfrac{11}{32}H^2 \right)
\]
where $\beta \in \bf{R}$. 

We write $\Coh^{\beta}(\bP^3,\cB_0)$ for the heart obtained by tilting $\Coh(\bP^3,\cB_0)$ with respect to $\mu_H = \beta$, defined analogously to \eqref{E:muH}. Since $\Forg$ is faithful, an object $E \in \Coh(\bP^3,\cB_0)$  is $\mu$-(semi)stable if $\mathrm{Forg}(E)$ is $\mu$-(semi)stable. The heart  $\Coh^{\beta}(\bP^3,\cB_0)$ underlies a weak stability condition on $\cD^b(\bP^3,\cB_0)$ with central charge 
\[
    Z_{\alpha,\beta}(E) =   \ch^{\beta}_{\cB_0,1}(E)\cdot \mathtt{i} + \frac{1}{2}\alpha^2 \ch^{\beta}_{\cB_0,0}(E) - \ch^{\beta}_{\cB_0,2}(E).
\]
Let 
\[
    \nu_{\alpha,\beta} = \frac{\ch^{\beta}_{\cB_0,2}(E) - \frac{1}{2}\alpha^2 \ch^{\beta}_{\cB_0,0}(E)}{\ch^{\beta}_{\cB_0,1}(E)}.
\]
We fix $0<\alpha < \frac{1}{4}$ and $\beta = -1$ and denote by $\cA_{\alpha,-1}$ the tilting heart for $\nu_{\alpha,\beta} = 0$. 
As discussed in \cite{Bayer2017StabilityCO}*{Thm. 1.2, \S9} and \cite{Li2018TwistedCO}*{Thm. 3.8}, we have a weak stability condition $\sigma_{\alpha,-1}=(-\mathtt{i}\cdot Z_{\alpha,-1},\cA_{\alpha,-1})$ on $\cD^b(\bP^3,\cB_0)$.

Before proceeding we start with some computations that we will need in the next section.
\begin{lem}\label{lem:computemutation}
    Let $X$ be a cubic fourfold. We denote by $v: \bf{P}^5\to \bf{P}^{20}$ the degree 2 Veronese embedding and by $w \colon\bf{P}^5 \rightarrow \bf{P}^{55}$ the degree 3 Veronese embedding. Then:  \vspace{-2mm}
   \begin{enumerate}
       \item $\iota\pr_L (\cO_X(3H)) \in \Coh(X)[2,3]$ and has cohomology objects $\cH^{-3}\iota\pr_L(\cO_X(3H))  \cong \Omega_{\bf{P}^5}^3(3)$ and $\cH^{-2}\iota\pr_L(\cO_X(3H)) \cong \cO_X.$ \vspace{-2mm}
       \item  $\bf{L}_{\cO_X(2H)}(\cO_X(4H)) =(v^*(\Omega^1_{\bf{P}^{20}})\otimes \cO_X(4H)) |_X[1]$.  \vspace{-2mm}
       \item  $\bf{L}_{\cO_X(2H)}(\cO_X(5H)) =( w^*(\Omega^1_{\bf{P}^{55}})\otimes \cO_X(5H)) |_X[1]$.  \vspace{-2mm}
   \end{enumerate} 
\end{lem}

\begin{proof}
Using the Euler exact sequence, one computes that $\bf{L}_{\cO_X(2H)}(\cO_X(3H)) = \Omega_{\bP^5}(3H)|_X[1]$. The vanishing $\H^{\bullet}(\bP^5,\Omega_{\bP^5}^{1}(-H)) = 0$ implies that $\H^{\bullet}(X,\Omega_{\bP^5}^{1}(2H)|_X)\cong \H^{\bullet}(\bP^5,\Omega_{\bP^5}^{1}(2H)) \cong \bC^{15}[0].$
Taking the $p^{\rm{th}}$ exterior power of the Euler sequence on $\bf{P}^5$, with $1\le p \le 6$, yields
\begin{align}\label{exact:pthEuler}
   0 \rightarrow \Omega_{\bP^5}^p(pH) \rightarrow \cO_{\bP^5}^{\binom{6}{p}} \rightarrow \Omega_{\bP^5}^{p-1}(pH) \rightarrow 0. 
\end{align}
Taking $p = 2$ and applying $-\otimes \cO_X(H)$ yields
\[
    \cO^{15}_X(H) \rightarrow  \Omega_{\bP^5}(3H)|_X \rightarrow \Omega_{\bP^5}^2(3H)|_X[1].
\]
Thus, $\bf{L}_{\cO_X(H)}(\bf{L}_{\cO_X(2H)}(\cO_X(3H))) = \Omega_{\bP^5}^2(3H)|_X[2]$. 
Tensoring the exact sequence $0 \to \mathcal{O}_{\bP^5}(-3H) \to \mathcal{O}_{\bP^5} \to \mathcal{O}_X \to 0$ with $\Omega_{\bP^5}^2(3H)$ yields the following exact sequence:
\[
0 \to \Omega_{\bP^5}^2 \to \Omega_{\bP^5}^2(3H) \to \Omega_{\bP^5}^2(3H)|_X \to 0.
\]
 
Next, since it follows from the above exact sequence that 
\[
\H^{\bullet}(X, \Omega_{\bP^5}^2(3H)|_X) = \H^0(\Omega_{\bP^5}^2(3H))\oplus \H^2(\Omega_{\bP^5}^2)[-1],
\]
we have an exact triangle:
\[
    \cO_X^{\oplus 20}\oplus \cO_X[-1] \rightarrow \Omega_{\bP^5}^2(3H)|_X \rightarrow \bf{L}_{\cO_X}(\Omega_{\bP^5}^2(3H)|_X ) \to \cO_X^{\oplus 20}[1] \oplus \cO_X.
\]
Taking \eqref{exact:pthEuler} with $p =3$ gives an exact triangle:
\[
\cO^{\oplus 20}_X \rightarrow  \Omega_{\bP^5}^2(3H)|_X \rightarrow \Omega_{\bP^5}^3(3H)|_X[1] \to \cO_X^{\oplus 20} [1].
\]
Consider the following commutative diagram, with exact rows and columns:
\[
    \begin{tikzcd}
	\cO_X^{\oplus 20} & \cO_X^{\oplus 20} \\
	\cO_X^{\oplus 20}\oplus \cO_X[-1] &\Omega_{\bP^5}^2(3H)|_X   & \bf{L}_{\cO_X}(\Omega_{\bP^5}^2(3H)|_X ) & \cO_X^{\oplus 20}[1]\oplus \cO_X \\
	\cO_X[-1]  & \Omega_{\bP^5}^3(3H)|_X [1] &\bf{L}_{\cO_X}(\Omega_{\bP^5}^2(3H)|_X )&\cO_X
	\arrow[from=1-1, to=1-2]
	\arrow[from=1-1, to=2-1]
	\arrow[from=1-2, to=2-2]
	\arrow[from=2-1, to=2-2]
	\arrow[from=2-1, to=3-1]
	\arrow[from=2-2, to=2-3]
	\arrow[from=2-2, to=3-2]
	\arrow[from=2-3, to=2-4]
	\arrow[from=2-3, to=3-3]
	\arrow[from=2-4, to=3-4]
	\arrow[from=3-1, to=3-2]
	\arrow[from=3-2, to=3-3]
	\arrow[from=3-3, to=3-4]
    \end{tikzcd}
\]
Taking cohomology objects $\cH^i$ to the bottom row shifted by $[2]$ gives (1). Now, let $v \colon \bP^5 \rightarrow \bP^{20}$ be the Veronese map of degree two. Tensoring the Euler sequence of $\bf{P}^{20}$ by $\cO_{\bf{P}^{20}}(2)$ and pulling back along $v$.
We note that $v^*(\cO_{\bP^{20}}(1)) = \cO_{\bP^{5}}(2H)$ gives

\[
    0 \rightarrow v^*\Omega_{\bP^{20}}(4H) \rightarrow \cO_{\bP^5}(2H)^{\oplus 21} \rightarrow \cO_{\bP^5}(4H) \rightarrow 0.
\]
Applying $\H^i$ gives (2). Part (3) follows from a similar argument.
\end{proof}

The properties of the objects that appeared in the previous lemma are summarized below. They are required for the mutation calculation in \Cref{prop:homvanishing_1-2}. 

\begin{lem}\label{lem:numerical_computation}
    Let $X$ be a  cubic fourfold, and let $v$ and $w$ be as in \Cref{lem:computemutation}. \vspace{-2mm}
    \begin{enumerate}
        \item For all $i\ne 0$, $\H^i(v^*(\Omega_{\bP^{20}}^1)\otimes \cO_X(3H))= 0$ and $h^0(v^*(\Omega_{\bP^{20}}^1)\otimes \cO_X(3H)) = 71$. \vspace{-2mm}
        \item $\H^i(v^*(\Omega_{\bP^{20}}^1)\otimes \cO_X(H)) = \H^{i-1}(\cO_X(H))$ \vspace{-2mm}
        \item For all $i\ne 0$,  $\H^i(w^*(\Omega_{\bP^{55}}^1)\otimes \cO_X(4H)) = 0$ and $h^0(w^*(\Omega_{\bP^{55}}^1)\otimes \cO_X(4H)) = 126$. \vspace{-2mm}
        \item  $\H^i(w^*(\Omega_{\bP^{55}}^1)\otimes \cO_X(2H)) = \H^{i-1}(\cO_X(H))$. \vspace{-2mm}
    \end{enumerate}
\end{lem}

Some previous works \cites{Li2018TwistedCO,ouchi2017lagrangian} consider other mutation-equivalent decompositions of $\DCoh(X)$ such as $\DCoh(X) = \langle \cO_X(-H), \cT_X, \cO_X,\cO_X(H) \rangle$, where 
\[
\cT_X = {}^{\perp}\langle \cO_X(-H) \rangle \cap \langle \cO_X,\cO_X(H) \rangle^{\perp}.
\]
We define the projection functor $\Pi:= \bf{R}_{\cO_{X}(-H)}\bf{L}_{\cO_X}\bf{L}_{\cO_X(H)}\colon \cD^b(X) \rightarrow \cT_X.$
\begin{lem}\label{lem_T_Ku_noplane}
    The functor $\phi(-) := \bf{R}_{\cO_X}(-)\otimes \cO_X(-H) \colon \cD^b(X) \rightarrow \cD^b(X)$
    restricts to an equiv\-alence $\Ku(X) \rightarrow \cT_X$. Moreover, there is an isomorphism of functors $\phi \circ \pr_L \simeq \Pi \circ  (-\otimes \cO_X(-H))$.
\end{lem}

\begin{proof}
    Since $\bf{R}_{\cO_X}(\cO_X) = 0$,
    we obtain $\bf{R}_{\cO_X} \circ \bf{L}_{\cO_X} \simeq \bf{R}_{\cO_X}$.
    We note that $\bf{L}_E(-) \otimes \cL \simeq  \bf{L}_{E\otimes \cL}(-\otimes \cL )$ and $\bf{R}_E(-) \otimes \cL \simeq \bf{R}_{E\otimes \cL}(-\otimes \cL )$ for any $E$ and line bundle $\cL$. Thus:
    \begin{align*}
        \bf{R}_{\cO_X} \bf{L}_{\cO_X}\bf{L}_{\cO_X(H)}\bf{L}_{\cO_X(2H)}(-)\otimes \cO_X(-H) & \cong 
        \bf{R}_{\cO_X}\bf{L}_{\cO_X(H)}\bf{L}_{\cO_X(2H)}(-)\otimes \cO_X(-H)\\
        & \cong 
        \bf{R}_{\cO_X(-H)}\bf{L}_{\cO_X}\bf{L}_{\cO_X(H)}(-\otimes \cO_X(-H)).
    \end{align*}
\end{proof}

\subsubsection*{Geometric stability conditions on cubic fourfolds not containing a plane}
In this section, we will denote by $X$ a smooth cubic fourfold that does not contain any plane. 
Following \cite{lehn2017twisted}, twisted cubics on $X$ can be divided into two types: arithmetically Cohen-Macaulay (aCM) or non-Cohen-Macaulay (non-CM).
The non-CM curves are plane curves with an embedded point at a singular point.
We will denote by $F^{\prime}_C =\Pi(\cI_{C / X}(2H))$, where $C$ is a twisted cubic on $X$.
By \cite{Li2018TwistedCO}*{Prop. 5.7},  we have an isomorphism $\Pi(\cI_{C / Y}(2H)) \cong \Pi(\cI_{C / S}(2H))$ where $S\subset X$ is a cubic surface containing $C$.
Similarly, we define $F_C = \bf{L}_{\cO_X}(\cI_{C / S}(2H)) \in \Coh(X)$,
if $C$ is an aCM curve then $F_C$ lies in $\Ku(X)$ and $F_C = F^{\prime}_C$. 
On the other hand if $C$ is non-CM then there exists an exact triangle:
\[
F_C^{\prime} \rightarrow F_C \rightarrow \cO_X(-H)[1] \oplus \cO_X(-H)[2] \to F_C'[1].
\]
We set $E_C = \Xi(F_C)$ and $E_C^{\prime} = \Xi(F_C^{\prime})$.
Since $\Psi(\cO_X(-H)) = \cB_{-1}$, we have a distinguished triangle
\begin{align}\label{exact:E'_C}
    E_C^{\prime} \rightarrow E_C \rightarrow \cB_{-1}[1] \oplus \cB_{-1}[2].
\end{align}
Moreover $\ch_{\cB_0\leq 2}^{-1}(E_C) = \ch_{\cB_0\leq 2}^{-1}(E_C^{\prime}) = (0,6H,0).$

\begin{lem}\cite{Li2018TwistedCO}*{Prop. 3.5}\label{lem:phase of E'_C}
    For any twisted curve $C$ the objects $E_C^{\prime}[1]$ and $\cB_{-1}[2]$ lie in $\cA_{\alpha,-1}$ for $0<\alpha<1/4$.
    Moreover, $E_C^{\prime}[1]$ is $\sigma_{\alpha,-1}$-stable of phase 1.
\end{lem}
\begin{proof}
    By, \cite{Li2018TwistedCO}*{Prop 3.3} and \cite{Bayer2017StabilityCO}*{Proof of Theorem 1.2}, the objects $E_C$ and $ \cB_{-1}[1]$ lies in 
    $\Coh^{-1}(\bP^3,\cB_0)$, and $ \cB_{-1}[2]$ lies in $\cA_{\alpha,-1}$.
    It follows from the exact sequence (\ref{exact:E'_C}) that  $E_C^{\prime}$ is an object in $\Coh^{-1}(\bP^3,\cB_0)$.

    The stability of $E_C^{\prime}$ follows from \cite{Li2018TwistedCO}*{Prop. 3.5}.
    Since $Z_{\alpha,-1}(E_C^{\prime}) = 6 \mathtt{i}$, $E_C^{\prime}$ is of phase 
    $\frac{1}{2}$ 
    with respect to  $(Z_{\alpha,-1},\Coh^{-1}(\bP^3,\cB_0))$.
    As discussed in \cite{Li2018TwistedCO}*{Thm. 3.8} we have a weak stability condition $\sigma_{\alpha,-1}=(-\mathtt{i}\cdot Z_{\alpha,-1},\cA_{\alpha,-1})$ on $\cD^b(\bP^3,\cB_0)$.
    Therefore, $E_C^{\prime}$ is of phase $0$ with respect to $\sigma_{\alpha,-1}$.
\end{proof}

Consider now the functor
$\Xi^{\prime} := \Xi\circ\phi \colon \Ku(X) \rightarrow \cD^b(\bP^3,\cB_0)$.

\begin{prop}\label{prop:stableK^x}
Let $X$ be a cubic fourfold which does not contain a plane.
For any $x\in X$, 
there is a non-CM twisted cubic curve $C$ with an embedded point $x\in C$
such that $F_C^{\prime}[1]$ and $\Pi(\cO_x)$  are  identified. 
In particular $\Xi^{\prime}(K^x)\in \cD^b(\bP^3,\cB_0)$ is $\sigma_{\alpha,-1}$-stable of phase 1 in $\cA_{\alpha,-1}$.
\end{prop}
\begin{proof}
    The first claim is proved in \cite{Addington2014OnTS}*{\S 1}, \cite{Li2018TwistedCO}*{Prop. 5.7}.
    Since $E^{\prime}_C = \Xi^{\prime}(K^x)[-1]$ for any $x\in X$ and a suitable curve $C$, we deduce from \Cref{prop:stableK^x} and \Cref{lem:phase of E'_C} that $\Xi^{\prime}(K^x)\in \cD^b(\bP^3,\cB_0)$ is $\sigma_{\alpha,-1}$-stable of phase 1 in $\cA_{\alpha,-1}$.
\end{proof}

We will compute $\Xi^{\prime}(\pr_R(\cO_X(kH)))$ for $k = 0,1,2$.
Recall that by Serre duality in the form of \cite{B-KSerre}, $\pr_R(-) \simeq \pr_L(- \otimes \cO_X(3H))[-2]$ hence $\Xi^{\prime}\circ\pr_R(\cO_X(kH))  = \Psi \circ\sigma^*\circ\phi\circ\pr_L(\cO((k+3)H))[-2].$ 
By \Cref{lem_T_Ku_noplane} we have $\Xi^{\prime}\circ\pr_R(\cO_X(kH))= \Psi \circ\sigma^*\circ\Pi(\cO((k+2)H))[-2]$.

\begin{prop}\label{prop:homvanishing_1-2}
    Let $X$ be a  cubic fourfold, and let $\cL$ be either $\cO_X(4H)$ or $\cO_X(5H)$. Then $\cH^i(\phi \circ \pr_L(\cL)) =0$ for any $i\ne -3,-2,-1$. 
\end{prop}

\begin{proof}
   We consider only the case $\cL = \cO_X(4H)$, since the other case is similar.
   Recall that by \Cref{lem_T_Ku_noplane} $\phi\circ \pr_L(\ko_X(4))=\Pi(\ko_X(3))=\bf{R}_{\ko_X(-H)}\bf{L}_{\ko_X}\bf{L}_{\ko_X(H)}(\ko_X(3))$.
   It follows from \Cref{lem:computemutation} that 
   \[
   \bf{L}_{\cO_X(2H)}(\cO_X(4H)) \cong ( v^*(\Omega_{\bP^{20}})(4H))|_X[1],
   \]
   moreover 
   \[
   \bf{L}_{\cO_X(2H)}(\ko_X(4))=\bf{L}_{\ko_X(H)}(\ko_X(3))\otimes \ko_X(H)
   \]
   hence $\bf{L}_{\ko_X(H)}(\ko_X(3))=v^*(\Omega_{\bP^{20}})|_X\otimes \cO_X(3H)[1]$. 
   There is an exact triangle
   \[
   v^*(\Omega_{\bP^{20}})|_X\otimes \cO_X(3H)\rightarrow 
   \bf{L}_{\cO_X}( (v^*(\Omega_{\bP^{20}})|_X\otimes \cO_X(3H))) \rightarrow \H^{\bullet}(v^*(\Omega_{\bP^{20}})|_X\otimes \cO_X(3H)) \otimes \cO_X[1].
   \]
   By \Cref{lem:numerical_computation}, we obtain $\H^{\bullet}( (v^*(\Omega_{\bP^{20}})|_X\otimes \cO_X(3H))) \otimes \cO_X \cong \cO_X^{\oplus 71}$.
   Thus, we have 
   \[
    \bf{L}_{\cO_X}(\Omega_{\bP^{20}} |_X\otimes \cO_X(3H)) \in [0,1].
   \]
Since $\RHom(\cO_X,\cO_X(-H)) = 0$, it follows from the above exact triangle and \Cref{lem:numerical_computation}(2) that 
   \begin{align*}
       \Hom^{\bullet}(\bf{L}_{\cO_X}( (v^*(\Omega_{\bP^{20}})|_X\otimes \cO_X(3H)) ),\cO_X(-H))
       & \cong \Hom^\bullet(v^*(\Omega_{\bP^{20}})|_X\otimes \cO_X(3H),\cO_X(-H))\\
       & \cong \Hom^{4-\bullet}(\cO_X,v^*(\Omega_{\bP^{20}})|_X\otimes \cO_X(H))^{\vee}\\
       & \cong \H^{4-\bullet}(v^*(\Omega_{\bP^{20}})|_X\otimes \cO_X(H))^{\vee}\\
       & \cong \H^{3-\bullet}(\cO_X(H))^{\vee}.
   \end{align*}
   By definition, there exists an exact triangle:
   \begin{align}\label{exact:O(4)}
          \cO_X(-H)^{6}[2] \rightarrow \bf{R}_{\cO_X(-H)}\bf{L}_{\cO_X}((v^*(\Omega_{\bP^{20}})|_X\otimes \cO_X(3H))) \rightarrow \bf{L}_{\cO_X}(v^*(\Omega_{\bP^{20}})|_X\otimes \cO_X(3H)).
   \end{align}
   Therefore, since $\bf{R}_{\cO_X(-H)}\bf{L}_{\cO_X}(\Omega_{\bP^{20}} |_X\otimes \cO_X(3H)) \in [0,2]$, the statement holds.
\end{proof}

\begin{prop}\label{prop:123}
    Let $\cL$ be either $\cO_X(4H)$ or $\cO_X(5H)$. Then, $\sigma^* \circ \phi\circ \pr_L(\cL) \in [1,3]$. 
    In particular $\Xi^{\prime}\circ\pr_R(\ko_X(H)),\Xi^{\prime}\circ\pr_R(\ko_X(2H)) \in [-\infty,2]$.
\end{prop}
\begin{proof}
    Since $\sigma^*$ is fully faithful, we have 
    \[
    \sigma^*\bf{R}_{\cO_X(-H)}\bf{L}_{\cO_X}(\bf{L}_{\cO_X(2H)}(\cL)\otimes\cO_X(-H)) \cong \bf{R}_{\cO_{\widetilde{X}}(-H)}\bf{L}_{\cO_{\widetilde{X}}} (\bf{L}_{\cO_{\widetilde{X}}(2H)}(\sigma^*(\cL))\otimes\cO_{\widetilde{X}}(-H)).
    \]
    For each case, the statement follows from the same argument in \Cref{prop:homvanishing_1-2}.
    By definition the functor $\Psi(-)$ satisfies $\Psi([a,b]) \subset [-\infty,b+1],$ it follows that $\Xi^{\prime}\circ\pr_R(\cO_X(kH))=\Psi\circ \sigma^*\circ \phi\circ \pr_L(\ko((k+3)H))[-2] \in [-\infty,2]$ for $k=1,2$.
\end{proof}
We are now ready to prove the vanishing that we need to glue stability conditions.
Let us recall that in \cite{Bayer2017StabilityCO}*{Thm. 1.2, \S9} the authors construct   on $\Ku(X)$ the  stability condition  $\tau^{\alpha} = (Z^\alpha,\cA^{\alpha})\in \Stab(\Ku(X))$ where
\begin{equation}\label{eq_def_satb_cubic4}
    \begin{split}
        \cA^{\alpha} :={\Xi'}^{-1}( \cA_{\alpha, -1})\cap\Ku(X)\\
        Z^\alpha:=-\mathtt{i}Z_{\alpha,-1}.
    \end{split}
\end{equation}
Let us observe that by \Cref{prop:stableK^x} the object $K^x[3]$ lies in $\cA^{\alpha}[3]$. Hence we will glue a stability condition with heart $\cA(\alpha)[3]$ on $\Ku(X)$ with a stability condition with heart $\langle\ko_X[2], \ko_X(1)[1],\ko_X(2)\rangle_{\mathrm{ext}}$.

\begin{cor}\label{cor_vanishing2_cubic4fold_noplane}
    We have that $\Xi^{\prime}\circ\pr_R(\cO_X(kH))\in\cA_{\alpha,-1}[-\infty,2]$ for $k=1,2$ in particular  the following vanishing holds
\[
\Hom^{\leq 0} (\cA_{\alpha,-1}[3],\Xi^{\prime}\circ\pr_R(\cO_X(2H)))= 0.
\]
\end{cor}
\begin{proof}
    Since $\cA_{\alpha,-1}$ is obtained by tilting twice, $\Coh(\bP^3,\cB_0)$ is contained in 
\[
    \cA_{\alpha,-1}[-2,0] = \langle\cA_{\alpha,-1},\cA_{\alpha,-1}[-1],\cA_{\alpha,-1}[-2] \rangle_{\rm{ext}}.
\]
By \Cref{prop:123} and the fact that $[-\infty,2] \subset \cA_{\alpha,-1}[-\infty,2]$ we get that  $\Xi^{\prime}\circ\pr_R(\cO_X(kH))$ lies in $\cA_{\alpha,-1}[-\infty,2]$ for $k = 1,2$.
\end{proof}

\begin{lem}\label{lem:vanisshing-1}
    The following vanishing holds 
    \[
\Hom^{\leq 0} (\cA_{\alpha,-1}\cap \Xi'(\Ku(X))[3],\Xi^{\prime}\circ\pr_R(\cO_X(H))[1])= 0.
\]
\end{lem}
\begin{proof}
    First, we compute $\phi\circ \pr_R(\cO_X(H))$.
    By \Cref{lem_T_Ku_noplane} and \Cref{lem:computemutation}(2),
    \begin{align*}
        \phi\circ \pr_R(\cO_X(H)) &\cong \phi \circ \pr_L(\cO_X(4H))[-2]\\
        &\cong  \bf{R}_{\cO_X(-H)}\bf{L}_{\cO_X}(\bf{L}_{\cO_X(2H)}(\cO_X(4H))\otimes \cO_X(-H))[-2]\\
        &\cong  \bf{R}_{\cO_X(-H)}\bf{L}_{\cO_X}((v^*(\Omega_{\bP^{20}})|_X\otimes \cO_X(3H)))[-1].
    \end{align*}
     By the exact sequence \eqref{exact:O(4)},
     there is an exact triangle
     \[
          \cO_X(-H)^{6}[1] \rightarrow \phi\pr_R(\cO_X(H)) \rightarrow \bf{L}_{\cO_X}(v^*(\Omega_{\bP^{20}})|_X\otimes \cO_X(3H))[-1].
    \]
    Since $\Xi \colon T_X \rightarrow \cD^b(\bP^3,\cB_0)$ and $\phi$ are fully faithful,
    it is enough to show that
    \[
    \Hom^{\leq 0 }(\phi(\cA^{\alpha})[3],\cO_X(-H)^{6}[2]) = \Hom^{\leq 0 }(\phi(\cA^{\alpha})[3],\bf{L}_{\cO_X}(v^*(\Omega_{\bP^{20}})|_X\otimes \cO_X(3H))[2]) = 0,
    \]
    where $\cA^{\alpha} = {\Xi'}^{-1}( \cA_{\alpha, -1})\cap\Ku(X)$ is a heart on $\Ku(X)$.
    Since $\phi(\cA^{\alpha})$ is contained in $\cT_X$, 
    the hom group $\Hom^{\leq 0 }(\phi(\cA^{\alpha}),\cO_X(-H)^{6}[1])$ vanishes.
    We note that $\phi(\cA^{\alpha}) = \Xi^{-1}( \cA_{\alpha, -1}) \cap \cT_X$.
    Next, we have 
    \begin{align*}
         &\Hom^{\leq 0 }(\phi(\cA^{\alpha}),\bf{L}_{\cO_X}(v^*(\Omega_{\bP^{20}})|_X\otimes \cO_X(3H))[-1])\cong\\
         \cong  &\Hom^{\leq 0 }(\Xi^{-1}( \cA_{\alpha, -1}),\bf{L}_{\cO_X}(v^*(\Omega_{\bP^{20}})|_X\otimes \cO_X(3H))[-1])\\
          \cong &\Hom^{\leq 0 }(\cA_{\alpha, -1}, \Xi (\bf{L}_{\cO_X}(v^*(\Omega_{\bP^{20}})|_X\otimes \cO_X(3H))[-1]))
    \end{align*}
    Since $\Xi = \Psi\circ \sigma^*$ and $\Psi([a,b]) \subset [-\infty,b+1]$ for any integers $a,b$, and by \Cref{lem:numerical_computation} and \eqref{exact:O(4)}, the object $\Xi (\bf{L}_{\cO_X}(v^*(\Omega_{\bP^{20}})|_X\otimes \cO_X(3H))[-1]) $ lies in $[-\infty,0]$.
    We also observe that $[-\infty,0]$ is contained in $\cA_{\alpha,-1}[-\infty,0]$ hence
    $\Hom^{\leq 0 }(\cA_{\alpha, -1}, \Xi (\bf{L}_{\cO_X}(v^*(\Omega_{\bP^{20}})|_X\otimes \cO_X(3H))[-1]))=0$.
\end{proof}

\begin{prop}\label{prop:homvanishing_0}
     Let $X$ be a  cubic fourfold. Then, $\phi \circ \pr_R(\cO_X[2]) \in [2,3]$ and 
      furthermore $\Hom^{\leq 0}(\cA_{\alpha,-1}[3],\Xi^{\prime}\circ\pr_R(\cO_X[2])) = 0$.
\end{prop}
\begin{proof}
Recall that by Serre duality $\pr_R(\ko_X[2])=\pr_L(\ko_X(3))$.
By \Cref{lem_T_Ku_noplane} we have that 
\[
\phi\circ \pr_L(\ko_X(3H))=\bf{R}_{\ko_X(-H)}\bf{L}_{\ko_X}\bf{L}_{\ko_X(2H)}(\ko_X(2H)).
\]
    We saw in the proof of \Cref{lem:computemutation} that 
    $\bf{L}_{\cO_X(H)}\bf{L}_{\cO_X(2H)}(\cO_X(3H)) = \Omega_{\bP^5}^2(3H)|_X[2]$
    moreover 
    \[
    \bf{L}_{\cO_X(H)}\bf{L}_{\cO_X(2H)}(\cO_X(3H))=(\bf{L}_{\ko_X}\bf{L}_{\ko_X(H)}\ko_X(2H))\otimes\ko_X(H)
    \]
    hence 
    \[
    \bf{L}_{\ko_X}\bf{L}_{\ko_X(2H)}(\ko_X(2H))=\Omega_{\bP^5}^2(2H)|_X[2].
    \]
    To compute $\bf{R}_{\cO_X(-H)}(\Omega_{\bP^5}^2(2H)|_X)$ we observe that
    \begin{align*}
        \Hom^\bullet(\Omega_{\bP^5}^2(2H)|_X,\cO_X(-H)) &\cong \Hom^{4-\bullet}(\cO_X,\Omega_{\bP^5}^2|_X)^{\vee}\\ 
        &\cong \H^{4-\bullet}(\Omega_{\bP^5}^2|_X)^{\vee}\\
        & \cong
        \begin{cases}
            \bC &\text{ if } \bullet = 2\\
            0&\text{ otherwise.}
        \end{cases}
    \end{align*}
    Thus, the object $\bf{R}_{\cO_X(-H)}(\Omega_{\bP^5}^2(2H))$ sits in the following distinguished triangle
    \[
    \cO(-H)[3] \rightarrow \bf{R}_{\cO_X(-H)}(\Omega_{\bP^5}^2(2H))[2] \rightarrow \Omega_{\bP^5}^2(2H)[2].
    \]
    This shows that $\phi \circ \pr_L(\cO_X(3H))=\bf{R}_{\cO_X(-H)}(\Omega_{\bP^5}^2(2H))[2]\in [2,3]$.

    For the second statement we recall that $\Xi$ is a fully faithful functor,
    so it is enough to show that 
    \[
    \Hom^{\leq 0}(\phi(\cA^{\alpha})[3],\cO(-H)[3]) = \Hom^{\leq 0}(\phi(\cA^{\alpha})[3],\Omega_{\bP^5}^2(2H)[2]) =0.
    \]
    Since $\phi(\cA^{\alpha}) \subset \cT_X$, we have 
    \[
     \Hom^{\leq 0}(\phi(\cA^{\alpha}),\cO(-H)[3]) =0.
    \]
    A straight computation shows that 
    \[
    \Hom^{\leq 0}(\phi(\cA^{\alpha}),\Omega_{\bP^5}^2(2H)[2]) \cong \Hom^{\leq 0}(\cA_{\alpha,-1},\Xi(\Omega_{\bP^5}^2(2H)[2])) 
    \]
    Next, by the $p^{\textrm{th}}$ exterior power of the  Euler sequence,
    there is an exact triangle
    \[
    \Omega_{\bP^5}^2(2H)|_X \rightarrow \cO_{X}^{15} \rightarrow \Omega_{\bP^5}(2H)|_X. 
    \]
    Since $\Xi(\cO_{X})= 0$, see \cite{Bayer2017StabilityCO}*{Prop 7.7}, we get
    $\Xi(\Omega_{\bP^5}^2(2H)|_X ) \cong \Xi(\Omega_{\bP^5}(2H)|_X[-1] )$.
    Moreover $\Xi([a,b]) \subset [-\infty,b+1]$, hence
     $\Xi(\Omega_{\bP^5}^2(2H)|_X )$ lies in $[-\infty,0]$.
    Therefore, the vanishing $\Hom^{\leq 0}(\cA^{\alpha}[3],\Xi^{\prime}\circ\pr_R(\cO_X)[2]) = 0$ holds.
\end{proof}

\begin{thm}\label{thm_geomstab_cubic4_noplane}
    Let $X$ be a cubic fourfold not containing a plane.
    Let $\tau_\alpha = (-\frac{1}{6}Z^\alpha,\cA^{\alpha}[3])$ be a stability condition on $\Ku(X)$, see \eqref{eq_def_satb_cubic4}, with $0<\alpha <\frac{1}{4}$.
    Consider a stability condition $\nu$ on $\cN_X = \langle \cO_X,\cO_X(H),\cO_X(2H)\rangle$ with heart $\langle \cO_X[2],\cO_X(H)[1],\cO_X(2H)\rangle_{\mathrm{ext}}$.
    Then for any $w\in \CC$ with $\Im(w)\geq 0$ the stability conditions $\tau_{\alpha,w}:=w\tau_\alpha$ and $\nu$ glue to a stability condition $\sigma_{\alpha,w}=(Z_{\alpha,w},\cP_{\alpha,w})\in \Stab(X)$.
Moreover 
\begin{enumerate}
    \item the object $K^x[3]$ is stable of phase $1-(\Im w)/\pi$ for any $x\in X$,
    \item if $\phi_{\sigma}(\cO_X(2H))<\phi_{\sigma}(N^x)< \phi_{\sigma}(\cO_X(H)[1]) < \phi_{\sigma}(\cO_X[2]) < 1$ and $\nu$ is spiked at $\ko_X(2)$ with respect to $N^x$ then $\sigma_{\alpha,0}$ is geometric.
\end{enumerate}
\end{thm}
Let us observe that if $\Re(Z_{\alpha,w}(\ko(i)[2-i]))>0$ for $i=0,1,2$ and $\Re(Z_{\alpha,w}(\ko_X(2)))\gg 1$ then $\nu$ is spiked at $\ko_X(2)$ with respect to $N^x$.
\begin{proof}
To prove gluing we apply \Cref{C:gluingconditions}, hence we only need to show that 
\[
\Hom_X^{\leq 0}(\iota(w\cP_{\alpha}(0,1]),\ko_X(i)[2-i])=0 \text{ for } i=0,1,2.
\]
where $w\cP_{\alpha}$ is the slicing of $\tau_{\alpha,w}$.
Let us observe that for any $E \in \cD^b(X)$
\begin{align*}
    \Hom_X(\iota(\cA^{\alpha}[3]),E) 
    &\cong \Hom_{\Ku(X)}(\cA^{\alpha}[3],\pr_R(E))\\
    & \cong \Hom_{\cD^b(\bP^3,\cB_0)}(\cA_{\alpha,-1}[3],\Xi^{\prime}\circ\pr_R(E)).\\
\end{align*}
Recall that by \Cref{prop:homvanishing_1-2}, \Cref{lem:vanisshing-1} and \Cref{prop:homvanishing_0},
\begin{align*}
    \Hom^{\leq 0} (\cA_{\alpha,-1}[3],\Xi^{\prime}\circ\pr_R(\cO_X(2H)))&=  \Hom^{\leq 0}(\cA_{\alpha,-1}[3],\Xi^{\prime}\circ\pr_R(\cO_X(H)[1]))\\
    &=\Hom^{\leq 0}(\cA_{\alpha,-1}[3],\Xi^{\prime}\circ\pr_R(\cO_X[2])) = 0.
\end{align*}
We conclude that $\Hom_X^{\leq 0}(\cP_{\alpha}(0,1][-\infty,0],E)$ for $E=\ko_X(i)[2-i]),i=0,1,2$.
Recall that by definition $w\cP_{\alpha}(0,1]=\cP_{\alpha}(\Im w/\pi,1+\Im w/\pi]$ hence we get 
\[
\Hom^{\leq -(\lceil 1+\Im w\rceil-1)}(\cP^\alpha(1-\lceil1+\Im w\rceil,1],E)=0
\]
as a consequence we get the vanishing of $\Hom^{\leq 0}(\cP^\alpha(0,\lceil 1+\Im w\rceil],E)=0$ and finally
\[
\Hom^{\leq 0}(\cP^\alpha(\Im w, 1+\Im w],E)=0 \text{ for }E=\ko_X(i)[2-i]),i=0,1,2.
\]

Recall that by \Cref{lem:phase of E'_C,prop:stableK^x} the object $K^x\in \cA^{\alpha}=\Xi'^{-1}(\cA_{\alpha,-1}))$ is stable with $Z^{\alpha}(K^x)=-6$, hence $K^x[3]$ is $\tau_{\alpha}$-stable of phase $1$ and the first claim follows.

To prove the second claim we use \Cref{thm:gluedgeometricKuz}.
Moreover the object $N^x \in \langle \cO_X,\cO_X(H),\cO_X(2H)\rangle$ is efficient indeed 
a non zero morphism $\cO(H)[1] \to N^x$ would induce a non zero morphism $\cO(H)[1] \to \cO_x$, due to the injectivity $N^x \hookrightarrow \cO_x$.
Moreover the assumptions in the second item and the fact that $\ko_x\neq K^x[3]\oplus N^x$ finishes the proof.
\end{proof}

\begin{rem}\label{rem:glu_cubic4}
Let us observe that the gluing in \Cref{thm_geomstab_cubic4_noplane} still holds if $\nu\in \Stab(\cN_X)$ is chosen with heart $\cB=\langle\ko_X[-n_0+2],\ko_X(H)[-n_1+1],\ko(2H)[-n_2]\rangle_{\mathrm{ext}}$ with $n_2\geq n_1\geq n_0\geq0$.
\end{rem}

\begin{rem}
    In \cite{ouchi2017lagrangian}, the construction of stability conditions on $\Ku(X)$ for a generic cubic fourfold containing a plane $X$ is established, along with the stability of the projection of skyscraper sheaves $K^x$ for some stability condition $\tau \in \Stab(\Ku(X))$.
    The gluability of $\tau$ and $\eta \in \Stab(\cN_X)$ follows from an argument similar to the case of a cubic fourfold containing no plane.
    However, to construct geometric stability conditions, one needs to rotate using the $\bf{C}$-action to ensure that $K^x$ becomes simple.
    Although the rotated condition $w \cdot (\tau * \eta)$ is not necessarily obtained directly by gluing, we can deduce the result using the properties of the glued stability condition $\tau * \eta$ via an argument parallel to \Cref{thm:gluedgeometricKuz}. 
    We omit the details here to avoid significant overlap with the proof of \Cref{thm:gluedgeometricKuz}.
\end{rem}

\begin{rem}
    Quite recently, stability conditions on the Kuznetsov component of cubic fivefolds were constructed in \cite{Liustabilityfivefolds}. The general theory developed, for instance, in \cite{HLJR}*{Thm. 3.9} allows the construction of glued stability conditions on cubic fivefolds. However, it seems more difficult to construct geometric stability conditions by gluing in this case. Indeed, consider the Kuznetsov-type decomposition 
    \[
        \DCoh(X) = \langle \Ku(X), \mathcal{O}_X, \mathcal{O}_X(H), \mathcal{O}_X(2H), \mathcal{O}_X(3H) \rangle
    \] 
    of a cubic fivefold $X$. To apply \Cref{thm:gluedgeometricKuz}, for each $x\in X$ the projection $\pr_{L}(\mathcal{O}_x)$ should be pure with respect to the standard t-structure for each $L \in \{ \mathcal{O}_X, \mathcal{O}_X(H), \mathcal{O}_X(2H), \mathcal{O}_X(3H) \}.$ However, a direct computation shows that $\pr_{\mathcal{O}_X}(\mathcal{O}_x)$ is a two-term complex.
\end{rem}

\section{Noncommutative minimal model program for Fano Varieties}
\label{S:nMMPFano}
We begin this section by explaining how quasi-convergent paths in the space of stability conditions $\Stab(\cD)$ of a $k$-linear triangulated category $\cD$ give rise to semiorthogonal decompositions of $\cD$. Sub\-sequently, we review quantum cohomology of Fano varieties and prove results concerning isomonodromic deformations of the quantum connection and the asymptotics of associated central charges. Finally, we explain the NMMP for Fano varieties and we formulate some precise conjectures in the cases of both semisimple and non-semisimple quantum cohomology.

\subsection{Quasi-convergent paths}
We recall the notion of quasi-convergent path proposed in \cite{NMMP} and developed systematically in \cite{HLJR}. Our convention differs slightly here as we take limits as $t\to 0$, rather than to $\infty$ as in \emph{loc. cit.} Throughout, $\cD$ is a $k$-linear triangulated category.

Consider a path $\sigma_t:(0,a]\to \Stab(\cD)$ for $a\in \bf{R}_{>0}$. For any non-zero object $E$ of $\cD$ we let $\phi^+_t(E) := \phi_{\sigma_t}^+(E)$ and $\phi_t^-(E) := \phi_{\sigma_t}^-(E)$, regarded as functions of $t$. A non-zero object $E$ is called \emph{limit semistable} with respect to the path $\sigma_t$ if 
\[
    \lim_{t\to 0} \phi_t^+(E) - \phi_t^-(E) = 0.
\]
For technical purposes, it is useful to introduce the following average phase function
\[
    \phi_\sigma(E) := \frac{1}{m_\sigma(E)}\sum_i \phi_\sigma(F_i)\cdot m_\sigma(F_i)
\]
where $F_1,\ldots, F_n$ are the $\sigma$-HN factors of $E$. When $E$ is $\sigma$-semistable, the average phase recovers the usual phase so there is no risk of confusion. Finally, we let $\ell_\sigma(E) := \log m_\sigma(E) + i\pi \phi_\sigma(E)$ and write $\ell_t = \ell_{\sigma_t}$. The function $\ell_t(E)$ is an approximation of $\logZ_t(E)$ for $E$ limit semistable, with the benefit of being defined for all $t\in (0,a]$.

\begin{defn}
\label{D:quasiconvergent}
    A path $\sigma_t:(0,a]\to \Stab(\cD)$ is called \emph{quasi-convergent} if \vspace{-2mm}
    \begin{enumerate} 
        \item for all non-zero objects $E$ of $\cD$, there is a filtration $0=E_n\to \cdots \to E_1 \to E_0 = E$ such that $G_i:=\Cone(E_i \to E_{i-1})$ is limit semistable and 
        \[
            \liminf_{t\to 0} \phi_t(G_i) - \phi_t(G_{i-1}) >0
        \]
        for all $1\le i \le n$; and \vspace{-2mm}
        \item for any pair of limit semistable objects $E$ and $F$ the following limit exists:
        \[
            \lim_{t\to 0} \frac{\ell_t(F) - \ell_t(E)}{1+\lvert \ell_t(F) - \ell_t(E)\rvert}. 
        \]
    \end{enumerate}
\end{defn}
The purpose of quasi-convergent paths is to define filtrations of $\cD$. In certain cases, these filtrations are admissible and give semiorthogonal decompositions of $\cD$, as we now explain. We fix henceforth a quasi-convergent path $\sigma_t:(0,a]\to \Stab(\cD)$. We can define a relation on the class $\cP_{\sigma_t}$ of $\sigma_t$-limit-semistable objects, which is well-defined by \cite{HLJR}*{Lem. 2.15}.

\begin{defn}\cite{HLJR}*{Def. 2.16}
    Given $E,F\in \cP_{\sigma_t}$, we write $F \prec^i E$ if 
    \[\lim_{t\to0} \phi_t(E) - \phi_t(F) = \infty
    \]
    and $E\preceq^i F$ otherwise. We write $E \sim^i F$ if both $E\preceq^i F$ and $F\preceq^i E$. 
\end{defn}

It is proven in \emph{loc. cit.} that $\preceq^i$ is a total preorder on $\cP_{\sigma_t} = \cP$ and consequently that $\preceq^i$ induces a total order on $\cP/{\sim^i}$. Given $E \in \cP$, let $\cD^E$ denote the full subcategory of $\cD$ consisting of objects $A$ whose limit HN factors $H$ all satisfy $H\sim^i E$. 

In \cite{HLJR}*{Prop. 2.20}, it is proven that for any $E\in \cP$ the category $\cD^E$ is triangulated. Furthermore, $\cD^E$ depends only on the class of $E$ in $\cP/{\sim^i}$.

The definitions so far do not involve our choice of $v:\rm{K}_0(\cD)\to \Lambda$. In \cite{HLJR}, there is a further notion of \emph{numerical} quasi-convergent path, which involves $v$. We won't recall the definition here; however, it is important that when $\cD = \DCoh(X)$ and we take $v:\rm{K}_0(X)\to \H_{\rm{alg}}^\bullet(X)$ as in \Cref{E:latticeexample}, quasi-convergent paths are automatically numerical by \cite{HLJR}*{Ex. 2.42}. In \emph{loc. cit.} there is also a notion of \emph{support property} for numerical quasi-convergent paths, generalizing the one for stability conditions regarded as constant paths. 

For the present work, the key fact about quasi-convergent paths is the following special case of the results of \cite{HLJR}.

\begin{thm}
\label{T:qconvmainthm}
    Let $\cD$ denote a $k$-linear triangulated category and let $\sigma_t$ be a numerical quasi-convergent path satisfying the support property. There is a finite collection $\{E_1\prec^i\cdots \prec^i E_n\}\subseteq \cP_{\sigma_t}$ such that 
    \[
        \cD = \langle \cD^{E_1},\ldots, \cD^{E_n}\rangle.
    \]
    Furthermore, each $\cD^{E_j}$ admits a stability condition satisfying the support property with respect to $\rm{K}_0(\cD^{E_j})\twoheadrightarrow v(\cD^{E_j}) \subset \Lambda$.
\end{thm}

\begin{rem}
    The categories $\cD^{E_j}$ in \Cref{T:qconvmainthm} depend only on the class of $E_j$ in $\cP/{\sim^i}$. Thus, the more intrinsic indexing set is given by $\cP/{\sim^i}$. If one relaxes the hypothesis that $\sigma_t$ satisfy the support property for paths, one can only guarantee that each $\cD^{E_i}$ admits a pre-stability condition. Finally, we note that there is an explicit construction of the (pre-)stability conditions on the categories $\cD^{E_j}$ using $\sigma_t$, which is given in \cite{HLJR}*{Thm. 2.30}.
\end{rem}

\subsection{Quantum cohomology}
\label{SS:quantumcohomology}

In this section, we review some salient points from the theory of quantum cohomology \cites{GGI16, KontsevichManin, manin_frob_man}. We consider a smooth Fano variety $X$ and its co\-homology with complex coefficients $\H^\bullet(X) = \H^{\bullet}(X,\CC)$ with a homogeneous basis 
$\{\phi_i\}$. We write an element $\tau\in \H^\bullet(X)$ as $\tau=\sum \tau^{i}\phi_i$ where $\tau^{i_2}, \dots, \tau^{i_2+b_2-1}$ are the coordinates of $\H^2(X)$ and $b_i$, for $i=0,\dots, 2\dim (X)$, are the Betti numbers of $X$. The \emph{genus zero Gromov--Witten potential} is a formal power series 
\begin{equation*}
    \mathcal{F}^X(\tau)\in \CC\llbracket \tau^0, \dots,\tau^{i_1+b_1-1}, e^{\tau^{i_2}}, \dots, e^{\tau^{i_2+b_2-1}}, \tau^{i_3}\dots, \tau^{i_{2\dim(X)}+b_{2\dim(X)}-1}\rrbracket
\end{equation*}
defined by the genus $0$ Gromov--Witten invariants of $X$. Since $X$ is Fano, $\mathcal{F}^X_0(\tau)$ is a finite sum for $\tau\in \H^2(X)$ and thus convergent on $\H^2(X)$. We denote by $B \subset \H^\bullet(X)$ a non-empty submanifold containing $0$, not necessarily open in $\H^\bullet(X)$, where $\cF_0^X(\tau)$ converges.

The \textit{quantum product} at $\tau\in B$ is defined as follows: for basis vectors $\phi_i,\phi_j,\phi_k$ we set
\begin{equation*}
    (\phi_i\star_{\tau}\phi_j ,\phi_k)_X=\partial_i\partial_j\partial_k \mathcal{F}^X(\tau)
\end{equation*}
where $(\cdot,\cdot)_X$ is the Poincaré pairing on $\H^\bullet(X)$, defined by $\alpha \otimes \beta \mapsto \int_X \alpha \wedge \beta$. The product $\star_{\tau}$ defines a family of Frobenius algebra structures on $\H^\bullet(X)$ parametrized by $\tau \in B$, with pairing the Poincar\'{e} pairing. When we restrict to $\tau \in \H^2(X)$, we call $\star_{\tau}$ the \emph{small quantum product}. If $\tau$ is not confined to $\H^2(X)$, we call the resulting structure the \emph{big} quantum product. 

We consider the trivial vector bundle $\cH_X\to B\times\bf{P}^1$ with fiber $\H^\bullet(X)$ and fix an affine coordinate $w$ on $\bf{C}\subset \bf{P}^1$. We define the \emph{quantum connection} $\nabla$ on $\cH_X$, regarded as a $\bf{C}$-linear map $\cT_{B\times \bf{P}^1}\to \End(\cH_X)$, by the formulas:
\begin{equation}\label{eq_big_quant_conn}
    \begin{split}
        \nabla_{\alpha}&=\partial_{\alpha}+\frac{1}{w}\alpha\star_{\tau}(-)\\
        \nabla_{w\partial_w}&=w\partial_w-\frac{1}{w}\cE_\tau \star_\tau(-) + \mu.
    \end{split}
\end{equation}
Here, $\alpha$ refers to the section of $\cT_{B\times \bP^1}$ corresponding to $\alpha \in \cT_{B,\tau} = \H^\bullet(X)$, $\mu$ is the diagonal grading operator defined by 
\begin{equation*}
    \mu|_{\H^{p}(X)}=\frac{p-\dim(X)}{2}\cdot\mathrm{id}_{\H^{p}(X)},
\end{equation*}
and $\cE$ is the \emph{Euler vector field}
\begin{equation}
\label{E:eulerfield}
    \mathcal{E}:=\mathrm{c}_1(X)+\sum\left(1-\frac{\deg(\phi_i)}{2}\right)\tau^i\phi_i.
\end{equation}
Note that when $\tau \in \H^2(X)$, we have $\cE_\tau = c_1(X)$. The quantum connection $\nabla$ is meromorphic and flat, with a regular singularity at $w = \infty$ and an irregular singularity at $w=0$. We call the associated differential equation
\begin{equation}\label{eq_general_qucoho}
    0 = w \frac{\partial}{\partial w} - \frac{1}{w}\cE_{\tau}\star_\tau(-) + \mu
\end{equation}
on sections of $\cH_X|_{\{\tau\}\times \bf{P}^1} \to \bf{P}^1$ the \emph{quantum differential equation} at $\tau \in B$.

In general, convergence of big quantum cohomology is a difficult question. In fact, little seems to be known even for Fano hypersurfaces in $\bf{P}^n$. Nevertheless, there are some cases where one can construct an open neighborhood of $\H^\bullet(X)$ for which the quantum product converges. For instance, for many homogeneous spaces of the form $G/P$ for $G$ a semisimple algebraic group and $P$ a parabolic subgroup, the subspace $B\subset \H^\bullet(X)$ where the potential converges can be taken to be nonempty and open. 

Many examples of this type, including the smooth quadrics $Q\subset \bf{P}^n$ and Grassmannian varieties $\Gr(k,V)$ studied in later sections, are generically semisimple. This means that a dense subset $B^\rm{ss}$ of $B$ consists of \emph{semisimple points}, i.e. points for which $(\H^\bullet(X),\star_\tau)$ is semisimple as a finite dimensional $\bf{C}$-algebra. In this case, the Euler element $\cE_\tau$ can be decomposed as 
\[
    \cE_\tau = \sum_{i=1}^N u_i\cdot e_i,
\]
where $e_1,\ldots, e_N$ are mutually orthogonal idempotents. In particular, $(u_1,\ldots, u_N) \in \bf{C}^N$ are uniquely determined up to $\mathfrak{S}_N$-action, being the eigenvalues of the operator $\cE_\tau \in \End(\H^\bullet(X))$. In the more general framework of Frobenius manifolds, Dubrovin \cite{Dubrovin1998} proves that the eigenvalues of the Euler field give coordinates in an open neighborhood of a semisimple point, called \emph{canonical coordinates} and written usually as $(u_1,\ldots, u_N)$. The canonical coordinates give an identification between an open neighborhood $\mathscr{B}$ of $\tau \in B^{\rm{ss}}$ and a subset of $\mathscr{U}_N$, the unordered configuration space of $N$ points in $\bf{C}$.

\subsection{Isomonodromic deformation of the quantum connection}
\label{SS:isomonodromicdeformation}

We discuss isomonodromic deformations of the quan\-tum connection in the semisimple case. The analysis is along the lines of \cite{GGI16}. The reader is also encouraged to consult \cite{BridgelandTL} for related results in the context of stability conditions. 

Consider a flat meromorphic connection $\nabla$ on a trivial bundle $\cV \to B\times \bf{P}^1$ with fiber $V$, with a logarithmic singularity at $B\times \{\infty\}$ and an order $2$ pole along $B\times \{0\}$. 
\begin{defn} 
\label{D:isomonodromic_deformation}
    An \emph{isomonodromic deformation} of $\nabla$ is given by a locally closed embedding $B\hookrightarrow M$ into a complex manifold and a flat meromorphic connection $\widetilde{\nabla}$ on the trivial $V$-bundle $\widetilde{\cV}\to M\times \bf{P}^1$ with a logarithmic singularity along $M\times \{\infty\}$, an order $2$ pole along $M\times \{0\}$, and such that $\widetilde{\nabla}^b = \nabla^b$ for all $b\in B$, regarded as connections on $\widetilde{\cV}|_{\{b\}\times \bf{P}^1}\to \bf{P}^1$.
\end{defn} 

The reader is encouraged to consult \cite{BridgelandStokesMulti}*{\S3} and \cite{SabbahIsomonodromic} for more on isomonodromic deformations.

We assume that the quantum product $\star_\tau$ is convergent in a connected open neighbor\-hood $B$ of $0 \in \H^{\bullet}(X)$. This is true, for example, when $X$ is a smooth quadric $Q\subset \bf{P}^n$ or a Grassmannian $\Gr(k,V)$. Suppose that $\tau \in B$ is a semisimple point. Then, the eigenvalues $u = (u_1,\ldots, u_N) \in \bf{C}^N$ of the Euler operator $\cE_\tau \in \End(\H^{\bullet}(X))$ form local canonical coordinates around $\tau$ and we may deform $\tau$ so that $\cE_\tau$ has no repeated eigenvalues. Locally, the set of such $u$ defines an open subset $\mathscr{B} \subset \mathscr{U}_N$, the configuration space of $N$ distinct unordered points in $\bf{C}$. Denote by $u = u(\tau) \in \mathscr{U}_N$ the element corresponding to $\tau$. Henceforth, we choose a lift of $u$ along the universal covering map $\widetilde{\mathscr{U}}_N\to \mathscr{U}_N$, which we also denote $u$ by abuse of notation. Up to shrinking, we may assume that $\mathscr{B}$ is evenly-covered and lift it uniquely to an open neighborhood of $u \in \widetilde{\mathscr{U}}_N$, also denoted $\mathscr{B}$. The following result is attributed to Dubrovin \cite{dubrovin1998painleve}.

We denote by $\widetilde{\cH}_X$ the trivial bundle with fiber $\H^\bullet(X)$ over $\widetilde{\mathscr{U}}_N\times \bf{P}^1$.

\begin{prop}
\label{P:Dubrovin1998}
    \cite{GGI16}*{Prop. 2.7.2} There is a unique flat meromorphic connection $\widetilde{\nabla}$ on $\widetilde{\cH}_X$ defined uniquely by 
    \begin{equation}
    \label{E:isomonodromicsystem}
    \begin{split}
        \widetilde{\nabla}_{\partial_{u_i}} & = \frac{\partial}{\partial u_i} + \frac{1}{w}C_i  \\
        \widetilde{\nabla}_{w\partial_w} & =w\frac{\partial}{\partial w} -\frac{1}{w}U + V
    \end{split}
    \end{equation}
    where $C_i$, $U$, and $V$ are $\End(\H^{\bullet}(X))$-valued holomorphic functions on $\widetilde{\mathscr{U}}_N$. The connection $\widetilde{\nabla}$ restricts to the big quantum connection over $\mathscr{B} \subset \widetilde{\mathscr{U}}_N$ and the eigenvalues of $U$ are the coordinates on the base.
\end{prop}

We will later see that the braid group action on semiorthogonal decompositions is mirrored by the braid group action by deck transformations of $\widetilde{\mathscr{U}}_N \to \mathscr{U}_N$: changing sheets corresponds to braiding of the eigenvalues of $U$. We next explain how $\widetilde{\nabla}$ determines a family of Frobenius algebra structures $(\H^{\bullet}(X),\eta_u,\star_u)$ parametrized by $u\in \widetilde{\mathscr{U}}_N$.

The form $\eta_u$ is the Poincar\'{e} pairing $\eta:\H^{\bullet}(X)^{\otimes 2}\to \bf{C}$ for all $u$. As mentioned above, $C_1,\ldots, C_N$, $U$, and $V$ are holomorphic functions $\widetilde{\mathscr{U}}_N \to \End(\H^\bullet(X))$, regarded as multi-valued functions of $(u_1,\ldots, u_N) \in \mathscr{U}_N$. The holomorphicity of these operators combined with the fact that \eqref{E:isomonodromicsystem} extends the quantum connection from $\mathscr{B}$ imposes conditions on the operators $C_1,\ldots, C_N, U,V$. 

\begin{lem}
\label{L:selfadjointU}
    The operator $U$ is self-adjoint with respect to $\eta$ for all $u$. That is, $U:\widetilde{\mathscr{U}}_N\to \End(\H^\bullet(X))$ is valued in the subspace of $\eta$-self-adjoint operators $\Sym(\H^\bullet(X),\eta)$.
\end{lem}

\begin{proof}
    $U$ restricts to the Euler operator $\cE$ over $\mathscr{B}$. There, as noted in \cite{GGI16}*{\S 2.4} we can take a holomorphic frame $\psi_1,\ldots, \psi_N$ of $\H^\bullet(X)$ which is idempotent pointwise for the quantum product and which diagonalizes the Euler field $\cE$. Since the eigenvalues of $\cE$ are distinct over this locus, we have $\eta(\psi_i,\psi_j) = \delta_{ij}$. In particular, since for any $\alpha,\beta \in \H^\bullet(X)$ we have 
    \begin{align*}
        \eta\left(\cE \sum \alpha_i\psi_i,\sum\beta_j\psi_j\right) & = \sum_{i,j}\alpha_iu_i\beta_j \cdot \eta(\psi_i,\psi_j)\\
        & = \sum_{i} \alpha_iu_i\beta_i = \eta\left(\sum \alpha_i\psi_i,\cE\sum \beta_j\psi_j\right)
    \end{align*}
    we see that $\cE$ is self-adjoint. Thus, $U:\mathscr{B}\to \End(\H^\bullet(X))$ is valued in $\Sym(\H^\bullet(X),\eta)$, a proper linear subspace, so the identity principle applied to $U:\widetilde{\mathscr{U}}_N \to \End(\H^\bullet(X))/\Sym(\H^\bullet(X),\eta)$ implies that $U$ maps all of $\widetilde{\mathscr{U}}_N$ to $\Sym(\H^\bullet(X),\eta)$.
\end{proof}

We now extend the frame $\psi_1,\ldots, \psi_N$ over $\mathscr{B}$ to all of $\widetilde{\mathscr{U}}_N$. First, let $v_1,\ldots, v_N:\widetilde{\mathscr{U}}_N \times \bf{C}^N \to \H^\bullet(X)$ denote a holomorphic eigenframe of $U$, where $v_i$ is pointwise the eigenvector with eigenvalue $u_i$. Since $U$ is valued in $\Sym(\H^\bullet(X),\eta)$, we have $\eta(v_i,v_j) = 0$ for $i\ne j$, and by non-degeneracy $\eta(v_i,v_i)\ne 0$ for all $i$. Thus, there is a unique holomorphic extension of $\psi_1,\ldots, \psi_N$ from $B$ to $\widetilde{\mathscr{U}}_N$ such that $\psi_i = v_i/\eta(v_i,v_i)^{1/2}$. 

\begin{prop}
    There is a unique $\cO_{\widetilde{\mathscr{U}}_N}$-linear product $\star: \H^\bullet(X) \otimes \H^\bullet(X) \to \H^\bullet(X)$, which agrees with the quantum product when restricted to $\mathscr{B}$ and such that $\psi_i \star_u \psi_j = \delta_{ij}$. Furthermore, \vspace{-2mm}
    \begin{enumerate}
        \item $\star$ is compatible with $\eta$ such that $(\H^\bullet(X),\star_u,\eta)$ is a semisimple Frobenius algebra for all $u\in \widetilde{\mathscr{U}}_N$ with unity $1\in \H^\bullet(X)$; \vspace{-2mm}
        \item and the operator $C_i$ is the projector onto $\psi_i$ for all $i=1,\ldots, N$.\footnote{That is, $C_i(\sum a_j\psi_j) = a_i\psi_i$.} \vspace{-2mm}
    \end{enumerate}
\end{prop}

\begin{proof}
    The product is defined by the relation $\psi_i \star_u \psi_j = \delta_{ij}$ for all $i,j$. Since $\psi_1,\ldots, \psi_N$ is a global holomorphic frame for $\H^\bullet(X)$, it follows that this is a holomorphic product extending the one over $\mathscr{B}$. By the identity principle, the relation $\psi_i \star_u \psi_j = \delta_{ij}$ over $\mathscr{B}$ implies that this is the unique such extension possible. It is clear that this is a semisimple and associative product for each fixed $u$. For compatibility, note that 
    \begin{align*}
        \eta(\alpha \star \beta,\gamma) & = \eta\left(\sum \alpha_i\psi_i\star \sum\beta_j\psi_j,\sum \gamma_k\psi_k\right) \\ 
        & = \sum_i (\alpha_i\beta_i)\gamma_i = \sum_i \alpha_i(\beta_i\gamma_i) \\
        & = \eta(\alpha,\beta\star\gamma)
    \end{align*}
    for any $\alpha,\beta,\gamma \in \H^\bullet(X)$. As is well known, $1$ is the identity element for the quantum product; i.e. $1\mapsto 1\star_\tau(-) \in \End(\H^\bullet(X))$ is constantly equal to the identity operator. Thus, the identity principle implies this holds for all $u \in \widetilde{\mathscr{U}}_N$. Finally, per the definition of the quantum connection \eqref{eq_big_quant_conn} and the fact that $\partial_{u_i}$ corresponds to $\psi_i$ over $\mathscr{B}$, we see that $C_i = \psi_i\star_u(-)$ as endomorphisms of $\H^\bullet(X)$.
\end{proof}

We consider the global frame of $\widetilde{\cH}_X|_{\widetilde{\mathscr{U}}_N}\to \widetilde{\mathscr{U}}_N$ defined by $u\mapsto \Psi_u:=(\psi_1(u),\ldots, \psi_N(u))$. The following proposition is a version of \cite{GGI16}*{Prop. 2.5.1}. We recall that a phase $\varphi \in \bf{R}$ is \emph{admissible} for $\sigma(\cE_u)$ if $e^{-\mathtt{i}\varphi}\lvert \sigma(\cE_u)\rvert$ is in general position -- see the notation and conventions section.

\begin{prop}
\label{P:isomonodromicasymptoticsolution}
    Consider a point $u\in \widetilde{\mathscr{U}}_N$ and let $\varphi \in \bf{R}$ be an admissible phase for $\sigma(\cE_u)$. There exists an open $U \ni u$ and an analytic fundamental solution $Y_u(w) = (y_1(u,w),\ldots, y_N(u,w))$  of the system \eqref{E:isomonodromicsystem} for $(u,w) \in U\times \bf{C}^*$ such that 
    \[
        Y_u(w) \cdot e^{\cE/w} \to \Psi_u \quad \text{as $w\to 0$ in the sector } \lvert \arg(w) - \varphi \rvert < \frac{\pi}{2}+ \epsilon.
    \]
    The solution $Y_u(w)$ is unique satisfying this asymptotic condition.
\end{prop}

\begin{proof}
    The proof is identical to \cite{GGI16}*{Prop. 2.5.1}.  
\end{proof}

\subsubsection*{Quantum cohomology central charges}

In \cite{iritani_integral}*{p. 14}, Iritani describes a family of linear maps $\cZ_w^\tau :\H^\bullet(X) \to \bf{C}$, depending on a parameter $w\in \bf{C}^*$ and a class $\tau \in \H^\bullet(X)$, possibly restricted to lie in $\H^2(X)$. Iritani suggests that $\cZ_w$ should lift to a family of stability conditions $\sigma_w$ along the canonical map $\Stab(X) \to \Hom(\H^\bullet(X),\bf{C})$ and refers to $\cZ_w$ as the \emph{quantum cohomology central charge}. We explain here how this can be extended along the isomonodromic deformation of quantum cohomology described in \Cref{SS:isomonodromicdeformation}.

First of all, by \cite{GGI16}*{Prop. 2.3.1} there is a canonical holomorphic function $S:\bf{P}^1\setminus \{0\} \to \End(\H^\bullet(X))$ satisfying several important properties, the most crucial of which for us is that for all $\alpha \in \H^\bullet(X)$:
    \[
        \nabla|_{\tau = 0} (S(w)w^{-\mu}w^\rho \alpha) = 0,
    \]
    where $\mu$ is the grading operator and $\rho = c_1(X)\cup(-) \in \End(\H^\bullet(X))$. Consequently, the quantum differential equation $\nabla_{w\partial_w} = 0$ at $\tau = 0 $ admits a \emph{canonical fund\-amental solution}
    \[
        \Phi_w(-) := S(w)w^{-\mu}w^\rho(-),
    \] 
    defined \emph{a priori} on a sector $\mathscr{S}$ containing $\bf{R}_{>0}$. In particular, $\Phi_w$ allows us to canonically identify cohomology classes $\alpha \in \H^\bullet(X)$ with flat sections of $\nabla|_{\tau = 0}$ over $\mathscr{S}$ by $\alpha \mapsto \Phi_w(\alpha)$ for $w\in \mathscr{S}$.
    
    Next, suppose that there is an open neighborhood $B$ of $0$ in $\H^\bullet(X)$ on which the quantum product converges. By \cite{GGI16}*{Rem. 2.3.2}, there is a holomorphic map 
    \[
        S(\tau,w): B \times (\bf{P}^1\setminus \{0\}) \longrightarrow \End(\H^\bullet(X))
    \]
    such that for all $\tau \in B$ and $\alpha \in \H^\bullet(X)$, $S(\tau,w)w^{-\mu}w^\rho\alpha$ is a $\nabla$-flat section. The definitions of $w^{-\mu} = \exp(-\mu \log w)$ and $w^\rho = \exp(\mu \log w)$ involve $\log w =: t$, so that we obtain a holomorphic map 
    \[
        \Phi_{t}^\tau:B \times \widetilde{\bf{C}^*} \longrightarrow \End(\H^\bullet(X))
    \]
    which we call the \emph{canonical fundamental solution} over $B$. 
    
    In the semisimple case, by \Cref{P:Dubrovin1998} and the discussion in \Cref{SS:isomonodromicdeformation}, the big quantum connection admits a unique extension to a flat connection $\widetilde{\nabla}$ on $\widetilde{\cH}_X \to \widetilde{\mathscr{U}}_N\times \bf{P}^1$, restricting to the big quantum connection on $\mathscr{B}$. Since $\widetilde{\mathscr{U}}_N \times \widetilde{\bf{C}^*}$ is simply connected, $\Phi_t^\tau$ can be uniquely extended to a holomorphic map $\Phi_t^\tau:\widetilde{\mathscr{U}}_N \times \widetilde{\mathbf{C}^*} \to \End(\H^\bullet(X))$ with the property that $\nabla \Phi^\tau_{e^t}(\alpha) = 0$ for all $\alpha \in \H^\bullet(X)$.

    \begin{defn}
    \label{defn_perturbed_can_sol}
        Given $u \in \widetilde{\mathscr{U}}_N$, we define the \emph{(deformed) quantum cohomology central charge} at $u$ to be 
        \[
            \cZ_{e^t}^u(-) = (2\pi e^t)^{\dim X/2} \int_X \Phi_t^u(-),
        \]
        regarded as a holomorphic map $\widetilde{\bf{C}^*}\to \Hom(\H^\bullet(X),\bf{C})$. 
    \end{defn}

    We will usually write $\cZ_w^u$ instead, by abuse of notation, however in most use-cases later on a phase $\varphi \in \bf{R}$ will be specified, making the expression $\cZ_w^u$ unambiguous.

    \subsubsection*{Gamma II Conjecture} Here, we explain part of the statement of the Gamma II conjecture of \cite{GGI16}*{Conj. 4.6.1} which will be used in the sequel. See also the restatement in \cite{Galkin_Iritani_Gamma}*{Conj. 4.9}, which is closer to the version we use. Let $X$ be a Fano variety, and let $\tau \in \H^\bullet(X)$ be given such that $\star_\tau$ is convergent. When $\star_\tau$ converges on a neighborhood $B$ of $0$ which contains semisimple points, we further take $u\in \widetilde{\mathscr{U}}_N$ where $N = \dim \H^\bullet(X)$.

    Choose a phase $\varphi \in \bf{R}$ that is admissible for $u$ and write $u = (u_1,\ldots, u_N)$, ordered such that $\Im(-e^{-\mathtt{i}\varphi}u_1)\ge \cdots \ge \Im(-e^{\mathtt{i}\varphi}u_N)$. \Cref{P:isomonodromicasymptoticsolution} identifies unique sections $y_1(u,w),\ldots, y_N(u,w)$ of the trivial $\H^\bullet(X)$-bundle such that 
    \[
        y_i(u,w)e^{u_i/w} \sim \Psi_u(e_i)
    \]
    as $w\to 0$ in a sector $\mathscr{S}$ containing $\bf{R}_{>0}\cdot e^{\mathtt{i}\varphi}$. More precisely, 
    \begin{equation}
    \label{E:asymptoticallyexponential}
        \lVert e^{u_i/w}\cdot y_i(\tau,w)\rVert = O(w^{-m})\text{ as $w\to 0$ in $\mathscr{S}$},
    \end{equation}
    for some $m \ge 0$ and $\lVert \:\cdot\:\rVert$ a fixed norm on $\H^{\bullet}(X)$. There is a unique cohomology class $A_i(u,\varphi) \in \H^\bullet(X)$ such that $y_i(u,w)$ parallel translated to $\tau = 0$ and $\arg = 0$ can be written as $\Phi_w^0(A_i(u,\varphi))$. 
    \begin{defn}
        We call the collection of cohomology classes $(A_1(u,\varphi),\ldots, A_N(u,\varphi))$ described above the \emph{asymp\-totically exponential basis} of $\H^\bullet(X)$ of the pair $(u,\varphi)$.
    \end{defn}

    Recall that if $\delta_i$ are the Chern roots of the tangent bundle $\cT_X$ of $X$, the \emph{Gamma class} of $X$ is
    \begin{equation*}
        \begin{split}
          \widehat{\Gamma}_X := \prod_{i} (1+\delta_i)
            &=\exp\bigg( -C_{\rm{eu}}\cdot c_1(X)+\sum_{k\geq 2}(-1)^k\cdot (k-1)!\cdot \zeta(k)\cdot \ch_k(\cT_X) \bigg).
        \end{split}
    \end{equation*}
    where $C_{\rm{eu}}$ is the Euler-Mascheroni constant and $\zeta$ is the Riemann zeta function. 
    
    If $X$ has semisimple quantum cohomology near $0$, so that the isomonodromic deformation is defined as in \Cref{SS:isomonodromicdeformation}, we say that $X$ satisfies the Gamma II conjecture at $(u,\varphi)$ if there exists a full exceptional collection $(E_1,\ldots, E_N)$ of $\DCoh(X)$ such that 
    \begin{equation}
    \label{E:asymptoticexponentialcollection}
        A_i(u,\varphi) = \widehat{\Gamma}_X\Ch(E_i) \text{ for all }i=1,\ldots, N,
    \end{equation}
    where $\Ch(-):=\sum_j (2\pi \mathtt{i})^j\ch_j(-)$. We call $(E_1,\ldots, E_N)$ an asymptotically exponential full ex\-ceptional collection at $(u,\varphi)$ and we call $u_i$ the \emph{exponent} of $E_i$.
    
    \begin{rem}
        Note that since $y_i(u,w)$ and $\Phi_w^0(A_i(u,\varphi))$ are flat sections, defined everywhere on $\widetilde{\mathscr{U}_N} \times \widetilde{\mathbf{C}^*}$ and agreeing along the locus $\tau = 0$ and $\arg = 0$, it follows that they agree everywhere. This allows us to estimate $\cZ_w^\tau(E_i)$ when $w\to 0$ along a ray in $\mathscr{S}$, and $E_i$ is asymptotically exponential.
    \end{rem}

    It is explained in \cite{GGI16}*{Rem. 4.6.3} that the veracity of Gamma II at $(u,\varphi)$ implies that it holds for all other choices of $(u,\varphi)$, up to mutating the exceptional collection. For later use, we will explain the behavior when $u$ varies in $\widetilde{\mathscr{U}}_N$, starting from a semisimple point $\tau \in \mathscr{B}\subset \widetilde{\mathscr{U}}_N$ at which Gamma II holds.\footnote{Recall that $\mathscr{B}$ is a locus in $B$ where the Euler operator has no multiple eigenvalues.} We suppose throughout that the isomonodromic deformation of quantum cohomology to $\widetilde{\mathscr{U}}_N$ is defined as in \Cref{SS:isomonodromicdeformation}.

\begin{setup}
\label{S:mutationsetup}
    Let $X$ be a Fano variety for which Gamma II conjecture holds at a semisimple point $\tau \in B$. Using canonical coordinates, we may assume that $\tau \in \mathscr{B}$, with eigenvalues $(u_1,\ldots, u_N) \in \mathscr{U}_N$. Furthermore, we choose an admissible phase $\varphi \in \bf{R}$ and order $(u_1,\ldots, u_N)$ such that $\Im(-e^{-\mathtt{i}\varphi}u_1)<\cdots< \Im(-e^{-\mathtt{i}\varphi}u_N)$. Let $(A_1,\ldots, A_N)$ denote the associated asymptotic basis. By \cite{GGI16}*{Rem. 4.6.3}, Gamma II holds at this $\tau\in \mathscr{B}$ and we let $\mathfrak{E} = (E_1,\ldots, E_N)$ denote a corresponding asymptotically exponential full exceptional collection. 
\end{setup}

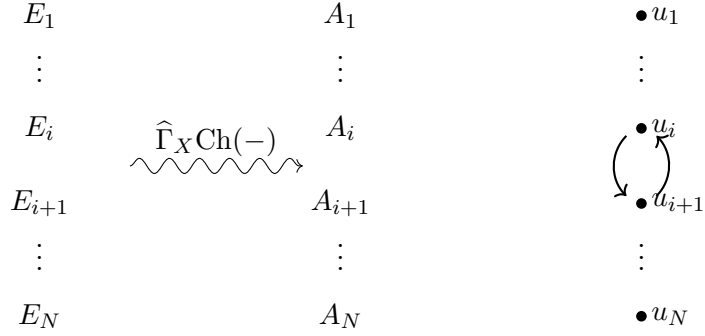
\begin{figure}
    \centering
    \begin{tikzpicture}
  \node at (-4,-2)  {$E_N$};
  \node at (-4,-0.5) {$E_{i+1}$};
  \node at (-4,0.5) {$E_{i}$};
  \node at (-4,2) {$E_1$};

  \node at (-4,-1.1) {$\vdots$};
  \node at (-4,1.4) {$\vdots$};

    \node at (0,-2) {$A_N$};
    \node at (0,-.5) {$A_{i+1}$};
    \node at (0,.5) {$A_{i}$};
    \node at (0,2) {$A_1$};

  \node at (0,-1.1) {$\vdots$};
  \node at (0,1.4) {$\vdots$};
  \draw[->,decorate,decoration={snake,amplitude=1mm,segment length=4mm}]
    (-2.8,0) -- node[above]{$\widehat{\Gamma}_X\Ch(-)$} (-.5,0);

    \fill (4,-2)   circle (2pt) node[right] {$u_N$};
  \fill (4,-0.5) circle (2pt) node[right] {$u_{i+1}$};
  \fill (4,0.5)  circle (2pt) node[right] {$u_{i}$};
  \fill (4,2)    circle (2pt) node[right] {$u_1$};

  \draw[->, bend right=45,thick] (4.2,-0.4) to (4.2,0.4);
  \draw[->, bend right=45,thick]  (3.8,0.4)  to (3.8,-0.4);

  \node at (4,-1.1) {$\vdots$};
  \node at (4,1.4) {$\vdots$};
\end{tikzpicture}
    \caption{The points $u_1,\ldots, u_N$ represent a configuration of points in $\bf{C}$. $(A_1,\ldots, A_N)$ is an asymptotically exponential basis of $\H^{\bullet}(X)$ such that $\cZ_w^u(A_i)\sim (2\pi w)^{\dim X/2} e^{-u_i/w}\int_X \Psi_u(e_i)$ and $(E_1,\ldots, E_N)$ is its lift to a full exceptional collection of $\DCoh(X)$. For any $i=1,\ldots, N$, we have $A_i = \widehat{\Gamma}_X\Ch(E_i)$.}
    \label{fig:mutation}
\end{figure}

\begin{prop}
\label{P:mutationforestimate}
    In \Cref{S:mutationsetup}, for any full exceptional collection $\mathfrak{E}' = (E_1',\ldots, E_N')$ mutation equivalent to $\mathfrak{E}$, there exist $u'\in \widetilde{\mathscr{U}}_N$ lying over $(u_1,\ldots, u_N) \in \mathscr{U}_N$ and $\varphi \in \bf{R}$ such that $\mathfrak{E}'$ is asymptotically exponential at $(u',\varphi)$ and the exponent of $E_i'$ is $u_i$ for all $i=1,\ldots, N$.
\end{prop}

\begin{proof}[Proof Sketch]
    For simplicity, assume that $\varphi = 0$ is admissible. Then, $(u_1,\ldots, u_N)$ is ordered such that $\Im(-u_1)<\cdots<\Im(-u_N)$. Up to deforming in $\widetilde{\mathscr{U}}_N$, we may assume that $u_k = (N-k)\cdot  \mathtt{i}$ for $k=1,\ldots, N$. Consider $u'\in \widetilde{\mathscr{U}}_N$ obtained by exchanging the positions of $u_i$ and $u_{i+1}$ by counterclockwise rotation -- see \Cref{fig:mutation}. The asymptotic basis $(A_1,\ldots, A_N)$ undergoes a left mutation giving a new basis of $\H^{\bullet}(X)$:
    \[
      (A_1',\ldots, A_N') := (A_1,\ldots, A_{i-1},A_{i+1},\bf{L}_{A_{i+1}}(A_i),A_{i+2},\ldots, A_N)
    \]
    such that $A_i'$ has exponent $u_i$ for each $i$. The mutation is defined using the data of the vector space $\H^\bullet(X)$ equipped with the non-symmetric pairing 
    \[
        [\alpha,\beta) := \frac{1}{(2\pi)^{\dim X}} \int_X(e^{\pi \mathtt{i}\rho}e^{\pi \mathtt{i}\mu}\alpha)\cup \beta
    \]
    so that in particular, $\bf{L}_{A_{i+1}}(A_i) :=  A_{i+1} - [A_i,A_{i+1})A_i$. See \cite{GGI16}*{\S 4.3} for thorough discussion. One can check that $\widehat{\Gamma}_X\Ch(\bf{L}_{E_{i+1}}(E_i)) = \bf{L}_{A_{i+1}}(A_i)$, so the corresponding full exceptional collection of $\DCoh(X)$ is obtained from $(E_1,\ldots, E_N)$ by the corresponding left mutation. In the case where we deform by clockwise rotation, the asymptotic basis and full exceptional collection undergo right mutation instead.
\end{proof}

\subsection{Asymptotics of the quantum cohomology central charge}
\label{SS:asymptotics}
We next derive asymptotic estimates of the quantum cohomology central charge (\Cref{defn_perturbed_can_sol}) of a smooth Fano variety $X$ for those $\tau \in B$ for which it is defined. Of particular interest to us are semisimple examples on the one hand and cubic threefolds and fourfolds on the other.

Choose a norm $\lVert \:\cdot\:\rVert$ on $\H^\bullet(X)$. Given a $\tau$-admissible phase $\varphi \in \bf{R}$, the $A$-model mutation system of \cite{SandaShamoto}*{\S 3} gives rise to a decomposition $\H^{\bullet}(X) = \bigoplus_{\lambda \in \lvert \sigma(\cE_\tau)\rvert} \cA_\lambda$, where
\begin{equation}
\label{E:Amodelmutationsystem}
    \cA_\lambda := \left\{v \in \H^{\bullet}(X): \exists m\in \bf{Z}_{\ge 0} \text{ such that }\lVert e^{\lambda/w}\Phi_w^0(v)\rVert 
    \leq O(|w|^{-m})
    \right\}
\end{equation}
for $w\to 0$ in the sector $\mathscr{S}(\varphi,\tfrac{\pi}{2} + \epsilon)$ with $0<\epsilon \ll 1$, see \cite{SandaShamoto}*{Lem. 3.6}. 

If $X$ satisfies property $\cO$ of \cite{GGI16}*{Def. 3.1.1}, then Prop. 3.2.1 of \emph{loc. cit.} identifies a fundamental matrix solution of the quantum differential equation on $\mathscr{S}_{\le 1}(0,\varrho) := \mathscr{S}(0,\varrho)\cap \{w\in \bf{C}^*:\lvert w\rvert \le 1\}$. We recall the salient points here. We fix $\tau = 0$ so that $\cE_0 = c_1(X)$. We have:
\begin{equation}
\label{E:coarseJordandecomp}
    \H^\bullet(X) = \bigoplus_{\lambda \in \lvert \sigma(\cE_0)\rvert} E(\lambda),
\end{equation}
where $E(\lambda)$ is the sum of the generalized eigenspaces of $\cE_0\star(-)$ corresponding to the eigenvalue $\lambda$. We enumerate $\lvert \sigma(\cE_0)\rvert$ as $\{\lambda_1,\ldots,\lambda_k\}$ such that $\lambda_1$ is the largest real eigenvalue, denoted $T$ in \cite{SandaShamoto}. Let $N_i = \dim E(\lambda_i)$ for $i=1,\ldots, k$. By property $\cO$, $N_1 = 1$. Next, we define a block-scalar matrix $U:= \diag(\lambda_1,\lambda_2 I_{N_2},\ldots, \lambda_k I_{N_k})$ and choose a linear isomorphism $\Psi:\bf{C}^N \to \H^{\bullet}(X)$ such that 
\[
    \Psi^{-1}(\cE_0 \star_0(-))\Psi = 
    \begin{bmatrix}
        B_1&&&\\
        &B_2&&\\
        &&\ddots& \\
        &&&B_k
    \end{bmatrix}
\]
where $B_i$ is an $N_i\times N_i$ matrix that is a sum of Jordan blocks with eigenvalue $\lambda_i$ for each $1\le i \le k$, corresponding to the Jordan decomposition of $\cE_0\star(-)|_{E(\lambda_i)}$. By \cite{GGI16}*{Prop. 3.2.1} there is a fundamental matrix solution of $\nabla_{w\partial_w} = 0$ of the form 
\[
    P(w)e^{-U/w}
    \begin{bmatrix}
        F_1(w)&&&\\
        &F_2(w)&&\\
        &&\ddots&\\
        &&&F_k(w)
    \end{bmatrix}
\]
over $\mathscr{S}_{\le 1}(0,\varrho)$ for some small $\varrho>0$. Furthermore, there is an asymptotic expansion: 
\[
    P(w) \sim \Psi + P_1w + P_2w^2+\cdots \text{ as }w\to 0 \text{ in } \mathscr{S}_{\le 1}(0,\varrho),
\] 
for some constant $N\times N$ matrices $P_i$. Also, $F_i(w)$ is a holomorphic $\GL(\bf{C}^{N_i})$-valued function such that $\max \{\lVert F_i(w)\rVert, \lVert F_i(w)^{-1}\rVert)\}\le C\exp(\delta \lvert w\rvert^{-p})$ on $\mathscr{S}$ for some $C,\delta>0$ and $0<p<1$. Finally, $F_1(w) = 1$.

\begin{rem}
\label{R:smallersector}
    We may choose an admissible phase $\varphi \in \bf{R}$ for $\sigma(\cE_0)$ such that $0<\varphi \ll 1$. Furthermore, we can choose $\epsilon>0$ small enough that $\mathscr{S}(\varphi,\epsilon)$ is in the first quadrant of $\bf{C}$, $\varphi +\epsilon < \varrho$, and such that for all $w\in \mathscr{S}(\varphi,\epsilon)$ one can enumerate $\lvert\sigma(\cE_0)\rvert = \{\lambda_1,\ldots, \lambda_k\}$ such that 
    \[
        \Re(\lambda_1/w)>\cdots>\Re(\lambda_k/w)
    \]
    where $\lambda_1 = T$ is the largest real eigenvalue of $\cE_0\star_0(-)$.
\end{rem}

\begin{prop}
\label{P:asymptoticsofsection}
    Consider $\mathscr{S}_{\le 1}(\varphi,\epsilon)$ where $\varphi$ and $\epsilon$ are as in \Cref{R:smallersector} and suppose that $X$ satisfies property $\cO$. For each $\lambda \in \lvert \sigma(\cE_\tau)\rvert$ and any non-zero $\alpha \in \cA_\lambda$, there is a non-zero $\Psi(\alpha) \in E(\lambda)$ and $m\in \bf{Z}_{\ge 0}$ such that 
    \[
        e^{\lambda/w}\cdot w^m\cdot \Phi_w^0(\alpha) \sim \Psi(\alpha)
    \]
    as $w\to 0$ in $\mathscr{S}_{\le 1}(\varphi,\epsilon)$.
\end{prop}

    The following argument is similar to that of \cite{GGI16}*{Prop. 3.3.1}. 
        
\begin{proof}
    For any $\alpha \in \H^{\bullet}(X)$, $\Phi_w^0(\alpha)$ is a flat section of $\cH_X$ and so by the above discussion we have 
    \[
        \Phi_w(\alpha) = P(w)e^{-U/w}(F_1(w)v_1+ \cdots + F_k(w)v_k),
    \]
    for $w\in \mathscr{S}_{\le 1}(\varphi,\epsilon)$ and where $v \in \H^{\bullet}(X)$ is decomposed as $v = \sum v_i$ according to \eqref{E:coarseJordandecomp}. Suppose that $\alpha \in \cA_{\lambda_p}$ -- we first show that $v_{p+1} = \cdots = v_k = 0$. Since $\alpha \in \cA_{\lambda_p}$, $\lVert e^{\lambda_p/w}\cdot \Phi_w(\alpha)\rVert$ is of moderate growth as $w\to 0$. Thus, $\lVert e^{(\lambda_p - \lambda_i)/w} \cdot F_i(w)v_i\rVert$ is of moderate growth for each $1\le i \le k$. However, if $i>p$ then
    \[
        \lVert v_i\rVert \le \left\lVert e^{(\lambda_i-\lambda_p)/w}F_i(w)^{-1}\right\rVert \cdot \left\lVert e^{(\lambda_p-\lambda_i)/w} F_i(w) v_i\right \rVert \le e^{-\varepsilon/w}\left \lVert e^{(\lambda_p-\lambda_i)/w} F_i(w) v_i\right\rVert
    \]
    for some $\varepsilon>0$. So, $\lVert v_i\rVert = 0$ for all $i>p$. Thus, we can write 
    \begin{align*}
        e^{\lambda_p/w}\cdot \Phi_w(\alpha) & = P(w)\sum_{i=1}^p e^{(\lambda_p-\lambda_i)/w}F_i(w)v_i \\
        & = P(w)F_p(w)v_p + \sum_{i=1}^{p-1}P(w)e^{(\lambda_p - \lambda_i)/w}F_i(w)v_i.
    \end{align*}
    By \Cref{R:smallersector}, $\Re(\tfrac{1}{w}(\lambda_p - \lambda_i)) < 0$ for all $w \in \mathscr{S}_{\le 1}(\varphi,\epsilon)$ and $i<p$ and thus $e^{\lambda_p/w}\cdot \Phi_w(\alpha) \sim \Psi \cdot  F_p(w) \cdot v_p$ as $w\to 0$ in $w \in \mathscr{S}_{\le 1}(\varphi,\epsilon)$. The coefficient functions of $F_p(w) v_p$ may have singularities as $w\to 0$, but the assumption that $e^{\lambda_p/w}\cdot \Phi_w(\alpha)$ is moderate growth implies they are at worst poles. Let $m\ge0$ be the highest order of such a pole. Then 
    \[
        \Psi(\alpha):=\lim_{w\to 0}\Psi \cdot w^m\cdot F_p(w) v_p \text{ for }w\in \mathscr{S}_{\le 1}(\varphi,\epsilon)
    \]
    exists, is non-zero, and lies in $E(\lambda_p)$.
\end{proof}

\Cref{P:asymptoticsofsection} gives estimates on the quantum cohomology central charge $\cZ_w^\tau(\alpha)$ for $\alpha \in \cA_\lambda$, when $\lambda \in \sigma(\cE_\tau)$. First, we need an elementary lemma. 

\begin{lem}
\label{L:Frobeniusdecomp}
    Let $A$ be a Frobenius $\bf{C}$-algebra with pairing $\langle\:\cdot,\cdot\:\rangle:A^{\otimes 2}\to \bf{C}$.\footnote{Here, $\bf{C}$ can be replaced by any algebraically closed field of characteristic zero.} Suppose $E\in A$ is central and denote the distinct eigenvalues of $E\cdot(-)\in \GL(A)$ by $\lambda_1,\dots, \lambda_n$. For each $1\le i \le n$, let $A_i$ denote the sum of the generalized eigenspaces with eigenvalue $\lambda_i$. Then
    \begin{equation*}
        A=\prod_{i=0}^n A_i
    \end{equation*}
    is an internal product decomposition of $\bf{C}$-algebras such that for each $1\le i \le n$, \vspace{-2mm}
    \begin{enumerate} 
        \item $A_i$ is a Frobenius algebra with pairing inherited from $A$; and \vspace{-2mm}
        \item $\ker(\langle 1,-\rangle:A\to \bf{C})$ does not contain $A_i$. \vspace{-2mm}
    \end{enumerate}
\end{lem}

\begin{proof}
    By definition, we have a decomposition $A = \bigoplus_{i=1}^n A_i$ of vector spaces. Next, let $a_i \in A_i$ and $a_j\in A_j$ be given. We can write $a_i = \sum v$, where the $v$ are generalized eigenvectors of eigenvalue $\lambda_i$. Similarly, we can write $a_j = \sum w$. Then, since $E$ is central, $vw$ is a generalized eigenvector for both $\lambda_i$ and $\lambda_j$; thus, $a_ia_j = \sum vw=0$. Thus, it follows that $\langle a_i,a_j\rangle = \langle 1,a_ia_j\rangle = 0$ for $i\ne j$ and that $A = \bigoplus_{i=1}^n A_i$ is orthogonal with respect to the pairing on $A$. Thus, $\langle\:\cdot,\cdot\:\rangle$ restricts to a non-degenerate pairing on each of the $A_i$. Using the direct sum decomposition, we can write $1 = \sum_i e_i$ where each $e_i$ is an idempotent. This gives an internal product decomposition $A = \prod_{i=1}^n A_i$ into Frobenius subalgebras, whence (1) follows.

    For (2), first note that $\langle 1,-\rangle :A_i \to \bf{C}$ equals the map $\langle e_i,-\rangle$. Next, suppose $\langle e_i,-\rangle = 0$ when restricted to $A_i$. Choose $x_i\in A_i$. For any $b = \sum b_ie_i\in A$, we have $\langle b,x_i\rangle = \langle e_i,b_ix_i\rangle = 0$. Thus, $x_i \in \ker\langle\:\cdot,\cdot\:\rangle$, contradicting non-degeneracy.
\end{proof}

\begin{ex}
\label{Ex:smallquantumcohomology}
    Let $X$ be a Fano variety and consider $A = \H^{\bullet}(X,\bf{C})$, equipped with the small quantum product $\star_\tau$ for $\tau \in \H^2(X)$. The pair $(\H^{\bullet}(X),\star_\tau)$ admits the structure of a Frobenius algebra with pairing 
    \[
        \langle \alpha,\beta\rangle := \int_X \alpha\star_\tau \beta.
    \]
    In addition, $(\H^{\bullet}(X),\star_\tau)$ admits the structure of a super-commutative algebra with $\H^{\bullet}(X) = \H^{\rm{even}}(X)\oplus \H^{\rm{odd}}(X)$, using the usual cohomological grading. Consequently, all even classes are central. By the definition of the Euler field $\cE_\tau$ in \eqref{E:eulerfield}, it follows that $\cE_\tau \in \H^{\rm{even}}(X)$ when $\tau \in \H^2(X)$, and thus we can apply \Cref{L:Frobeniusdecomp} to $\cE_\tau \in \GL(\H^{\bullet}(X))$ to obtain a decomposition $\H^{\bullet}(X) = \prod_{\lambda \in \lvert \sigma(\cE_\tau)\rvert} E(\lambda)$ into Frobenius subalgebras.
\end{ex}

We will use the notation of \Cref{Ex:smallquantumcohomology} in what follows.

\begin{cor}
\label{C:Zasymptoticestimate}
Suppose $X$ is a Fano variety satisfying property $\cO$. Let $\tau \in \H^2(X)$ be given and consider $0\neq \alpha \in \cA_\lambda$. There exist $m \in \bf{Z}_{\ge 0}$ and $0\ne \Psi(\alpha) \in E(\lambda)$ such that
    \[
        \cZ_w^\tau(\alpha) \sim \frac{(2\pi w)^{\dim X/2}}{w^m}\cdot e^{-\lambda/w} \cdot \int_X \Psi(\alpha).
    \]
    as $w\to 0$ in $\mathscr{S}_{\le 1}(\varphi,\epsilon)$. Furthermore, if $\dim E(\lambda) = 1$ then $\int_X \Psi(\alpha) \ne 0$.
\end{cor}

\begin{proof}
    The estimate for $\cZ_w^\tau(\alpha)$ is immediate from \Cref{P:asymptoticsofsection} and the definition of the quantum cohomology central charge in \Cref{defn_perturbed_can_sol}. Next, if $\dim E(\lambda) = 1$, then by \Cref{L:Frobeniusdecomp} one has $\int_X \Psi(\alpha) = \langle 1,\Psi(\alpha)\rangle \ne 0$.
\end{proof}

By \Cref{C:Zasymptoticestimate}, for any $\alpha \in \cA_\lambda$ such that $\int_X \Psi(\alpha)\ne 0$, we have
\begin{equation}
\label{E:asymptoticestimatelogZ}
    \log\cZ_w^\tau(\alpha) = C_\alpha - \frac{\lambda}{w} + \left(\frac{\dim X}{2} - m\right)\log(w) + \epsilon_\alpha(w),
\end{equation}
where $\epsilon_\alpha(w) \to 0$ as $w\to 0$ in $\mathscr{S}(\varphi,\epsilon)$. Here, $C_\alpha = \log (\int_X\Psi(\alpha)) + \tfrac{\dim X}{2}\log(2\pi).$

\subsubsection*{Estimates for deformed quantum cohomology charges}

Next, we apply \Cref{L:Frobeniusdecomp} to obtain est\-imates for the paths of central charges obtained using solutions of isomonodromic deform\-ations of the quantum connection as in \Cref{P:isomonodromicasymptoticsolution}. 

\begin{cor}
\label{C:isomonodromic_estimate}
    In \Cref{S:mutationsetup}, for any full exceptional collection $\mathfrak{E}' = (E_1',\ldots, E_N')$ mutation equivalent to $\mathfrak{E}$, there exists $v \in \widetilde{\mathscr{U}}_N$ lying over $(u_1,\ldots, u_N) \in \mathscr{U}_N$ such that 
    \[
        \cZ_w^v(E_i') \sim (2\pi w)^{\dim X/2} \cdot  e^{-u_i/w} \cdot \int_X \Psi_v(e_i)
    \]
    as $w\to 0$ in a sector $\mathscr{S}$ containing $\bf{R}_{>0}e^{\mathtt{i}\varphi}$. Furthermore, $\int_X \Psi_v(e_i) \ne 0$.
\end{cor}

The reader should compare with \cite{GGI16}*{\S 4.7}.

\begin{proof}
    By \Cref{P:mutationforestimate}, there exists $v\in \widetilde{\mathscr{U}}_N$ such that $\{\widehat{\Gamma}_X\Ch(E_i')\}_{i=1}^N$ is an asymptotically exponential basis for $\H^{\bullet}(X)$ with respect to $\widetilde{\nabla}^v_{w\partial_w} = 0$ and such that $E_i'$ has exponent $u_i$. Taking $i^{\rm{th}}$ components of \Cref{P:isomonodromicasymptoticsolution}, we have $y_i(v,w) = \Phi_w^v(\widehat{\Gamma}_X\Ch(E_i'))\sim e^{-u_i/w}\Psi_v(e_i)$  as $w\to 0$ in $\mathscr{S}$ and the claimed asymptotic estimate follows from the definition of $\cZ_w^v(-)$. 

    Finally, by definition $\Psi_v(e_i)$ is an eigenvector for $\cE_v\star_v(-)$, and thus by \Cref{L:Frobeniusdecomp} we have a decomposition $\H^\bullet(X) = \prod_{i=1}^N \Psi_v(e_i)$ of Frobenius algebras. Thus, $\int_X \Psi_v(e_i) = \langle 1,\Psi_v(e_i)\rangle \ne 0$. 
\end{proof}

In \Cref{S:mutationsetup}, \Cref{C:isomonodromic_estimate} gives us estimates:
\begin{equation}
\label{E:logZestimatesemisimple}
    \log \cZ_w^v(E_i') = C_i - \frac{u_i}{w} + \frac{\dim X}{2} \log(2\pi w) + \epsilon_i(w)
\end{equation}
where $\epsilon_i(w) \to 0$ as $w\to 0$ in $\mathscr{S}$. Here, $C_i = \log(\int_X\Psi_v(e_i)).$ These estimates will be crucial in what follows to construct quasi-convergent paths in $\Stab(X)$.

\subsubsection*{Estimates for cubics}
\label{section_QH_cubics}

Next, we recall some key results from the work of Sanda--Shamoto \cite{SandaShamoto}, which we use to compute the asymp\-totics of the quantum cohomology central charges $\cZ_w^0$ for cubic hypersurfaces in dim\-ensions 3 and 4. First, let $Y\subset \bf{P}^4$ denote a smooth cubic threefold. Recall the standard Kuznetsov decomposition
\begin{equation}
\label{E:KuznetsovCubicthreefold}
    \DCoh(Y)=\langle \Ku(Y), \cO, \cO(1)\rangle.
\end{equation}

By \cite{SandaShamoto}, we can compute the spectrum of $c_1(Y)\star_0(-) \in \End(\H^\bullet(Y))$ as follows: there is an orthogonal decomposition $\H^\bullet(Y) = \H^\bullet_{\rm{amb}}(Y) \oplus \H^\bullet_{\rm{amb}}(Y)^\perp$ with respect to the Poincar\'{e} pairing, such that $\H_{\rm{amb}}^\bullet(Y)$ is closed under $\star_0$. Here, $\rm{H}_{\rm{amb}}^\bullet(Y)$ is the subspace of ambient classes, i.e. those pulled back from $\H^\bullet(\bf{P}^4)$ along the inclusion $i:Y\hookrightarrow \bf{P}^4$. By \cite{SandaShamoto}*{Lem. 7.3}, $c_1(Y)\star_0(-)$ acts by zero on $\H_{\rm{amb}}^\bullet(Y)^\perp$. On the other hand, by \cite{SandaShamoto}*{Lem. 7.5}, 
    \[
        (\H_{\rm{amb}}^\bullet(Y),\star_0) \cong \bf{C}[h]/(h^2(h^2-27))
    \]
where $h$ is the hyperplane class. Since $c_1(Y) = 2h$, it follows that the eigenvalues of $c_1(Y)\star_0(-)$ are $c_1 = T$, $c_2 = 0$, and $c_3 = -T$, where $T = 2\sqrt{6}>0$. Note that $\pm T$ are simple eigenvalues, while $T$ appears with multiplicity $2+ \dim \H_{\rm{amb}}^\bullet(Y)^\perp$.

\begin{prop}
    There exists a mutation $\DCoh(Y) = \langle \cC_1,\cC_2,\cC_3\rangle$ of the Kuznetsov decomposition \eqref{E:KuznetsovCubicthreefold} such that for all objects $E$ of $\cC_i$ with $\ch(E) \ne 0$, we have
    \[
        \norm{e^{c_i/w}\Phi_w^0(\widehat{\Gamma}_Y \Ch(E))}
        \leq O(|w|^{-m})    
    \]
    as $w\to 0$ in $\mathscr{S}(\varphi,\tfrac{\pi}{2}+\epsilon)$ for some $m\geq 0$ and $\varphi$ as in \Cref{R:smallersector}. Furthermore, $\cC_1 = \langle \cO(1)\rangle$, $\cC_3 = \langle \cO(2)\rangle$, and $\cC_2\simeq \Ku(Y)$ by the associated mutation functor. 
\end{prop}

\begin{proof}
    See the proof of \cite{SandaShamoto}*{Thm. 7.9}.
\end{proof}

\begin{cor}\label{cor_qcohZ_cubic3}
    The restriction of $\cZ_w^0$ to $\rm{K}_0(\cC_i)$ is non-trivial for all $i=1,2,3$. That is, there is an object $E\in \cC_i$ such that in the notation of \Cref{P:asymptoticsofsection}, $\int_Y\Psi(\widehat{\Gamma}_Y\Ch(E))\neq 0$. For such $E$, there exists an $m\ge 0$ such that 
    \begin{equation}
    \label{E:asymptoticestimatelogZcubic3}
    \log\cZ_w^0(E) = C_E- \frac{c_i}{w} + \left(\frac{3}{2} - m\right)\log(w) + \epsilon_E(w)
\end{equation}
such that $\epsilon_E(w)\to 0$ as $w\to 0$ in $\mathscr{S}(\varphi,\epsilon)$ with $\varphi$ and $\epsilon$ as in \Cref{R:smallersector}. Here, $C_E = \log \int_Y\Psi(\widehat{\Gamma}_Y\Ch(E)) + \tfrac{3}{2}\log(2\pi).$
\end{cor}

\begin{proof}
    First, note that $Y$ satisfies Property $\cO$ by \cite{SandaShamoto}*{Cor. 7.7}. Therefore, the discussion in \Cref{Ex:smallquantumcohomology} applies. The claim is now an immediate consequence of \Cref{C:Zasymptoticestimate}, see \eqref{E:asymptoticestimatelogZ}.
\end{proof}

Next, we consider a smooth cubic fourfold $X\subset \bf{P}^5$ with its Kuznetsov decomposition
\begin{equation}
    \label{E:KuznetsovdecompoCubicfourfold}
    \DCoh(X) = \langle \Ku(X),\cO,\cO(1),\cO(2)\rangle.
\end{equation}

As in the threefold case, we have $\H^\bullet(X) = \H^\bullet_{\rm{amb}}(X) \oplus \H^\bullet_{\rm{amb}}(X)^\perp$, where $c_1(X)\star_0(-)$ acts by zero on $\H^\bullet_{\rm{amb}}(X)^\perp$ and restricts to an endomorphism of $\H^\bullet_{\rm{amb}}(X)$. Furthermore, writing $h$ for the hyperplane class, we have 
\[
    (\H^\bullet_{\rm{amb}}(X),\star_0) \cong \bf{C}[h]/(h^2(h^3-27)).
\]
Thus, the eigenvalues of $c_1(X)\star_0(-)|_{\H^\bullet_{\rm{amb}}(X)} = 3h\cdot(-)$ are $0$ and $9, 9e^{2\pi \mathtt{i}/3},9e^{4\pi \mathtt{i}/3}$, where the latter three eigenvalues are simple. We index them as $c_1 = 9e^{4\pi \mathtt{i}/3}$, $c_2 = 9$, $c_3 = 0$, and $c_4 = 9e^{2\pi \mathtt{i}/3}$.

\begin{prop}
    There exists a mutation $\DCoh(X) = \langle \cC_1,\cC_2,\cC_3,\cC_4\rangle$ of the Kuznetsov decomposition \eqref{E:KuznetsovdecompoCubicfourfold} such that for all objects $E$ of $\cC_i$ with $\ch(E)\ne 0$, we have 
    \[
        \norm{e^{c_i/w}\Phi_w^0(\widehat{\Gamma}_X \Ch(E))}
        \leq O(|w|^{-m})    
    \]
    for $w\to 0$ in $\mathscr{S}(\varphi,\epsilon)$ with $\varphi$ and $\epsilon$ as in \Cref{R:smallersector} and some $m\geq 0$. Furthermore, $\cC_1 = \langle \cO(2)\rangle$, $\cC_2 = \langle \bf{R}_{\cO(2)}\cO\rangle$, $\cC_4 = \langle \cO(1)\rangle$, and $\cC_3\simeq \Ku(X)$ by the mutation functor.
\end{prop}

    \begin{proof}
    See the proof of \cite{SandaShamoto}*{Thm. 7.9}.
\end{proof}

\begin{cor}\label{cor_qcohZ_cubic4}
    The restriction of $\cZ_w^0$ to $\rm{K}_0(\cC_i)$ is non-trivial for all $i=1,2,3,4$. That is, there is an object $E$ of $\cC_i$ such that in the notation of \Cref{P:asymptoticsofsection}, $\int_X\Psi(\widehat{\Gamma}_X\Ch(E))\neq 0$. For such $E$, there exists an $m\ge 0$ such that
    \begin{equation}
\label{E:asymptoticestimatelogZcubic4}
    \log\cZ_w^0(E) = C_E- \frac{c_i}{w} + \left(2 - m\right)\log(w) + \epsilon_E(w)
\end{equation}
such that $\epsilon_E(w)\to 0$ as $w\to 0$ in $\mathscr{S}(\varphi,\epsilon)$ with $\varphi$ and $\epsilon$ as in \Cref{R:smallersector}. Here, $C_E = \log \int_X\Psi(\widehat{\Gamma}_X\Ch(E)) + 2\log(2\pi).$
\end{cor}

\begin{proof}
    The proof is the same as in the threefold case.
\end{proof}

\subsection{Quasi-convergent paths from semisimple quantum cohomology}

Using the results of \Cref{SS:asymptotics}, we give a criterion for obtaining quasi-convergent paths from quantum co\-homology in the semisimple case which amplifies \cite{Vanjapaper}*{Thm. 5.11}. 

We consider a full exceptional collection $\cE = (E_1,\ldots, E_N)$ in a $k$-linear Hom-finite\footnote{That is, for any objects $E$ and $F$ in $\cD$, $\Hom_{\cD}(E,F)$ is a finite dimensional $k$-vector space.} triangulated category $\cD$. By \cite{Collins_Polischuk_2010}, there is a region $\cG_{\cE} \subset \Stab(\cD)$ glued from $\langle E_1,\ldots, E_N\rangle$ and holomorphic maps $\logZ_{\cE}: \cG_{\cE} \to \bf{C}^{d}$ given by $\sigma \mapsto (\logZ_\sigma(E_1),\ldots, \logZ_\sigma(E_N))$; indeed, $\logZ_{\cE}$ is defined because $E_i$ is $\sigma$-stable for all $i=1,\ldots, N$ and $\sigma \in \cG_{\cE}$ by \cite{Collins_Polischuk_2010}*{Prop. 2.2}. Since $\cD$ is Hom-finite, there exists $m(\cE) \in \bf{N}$ such that $\Hom^{\le -m(\cE)}(E_i,E_j) = 0$ for all $1\le i,j\le N$.

\begin{lem}
\label{L:gluingregion}
    The map $\logZ_{\cE}$ restricts to a biholomorphism $\Omega_{\cE} \to \cS_{\cE}$, where $\cS_{\cE}$ is the open subset of $\bf{C}^N$ such that $\frac{1}{\pi}(\Im(z_i) - \Im(z_{i-1})) > m(\cE)$ for all $i=2,\ldots, N$ if $E_{i-1}$ and $E_{i}$ are not orthogonal and $\Omega_{\cE}$ is the connected component of $\logZ_{\cE}^{-1}(\cS_\cE)$ contained in $\cG_\cE$.
\end{lem}

\begin{proof}
    First, note that in the case where $E_i$ and $E_{i-1}$ are orthogonal, $\Stab(\langle E_{i-1},E_i\rangle) \cong \bf{C}^2$ via the map $\sigma \mapsto (\logZ_\sigma(E_{i-1}),\logZ_\sigma(E_i))$. Then, surjectivity can be verified using \Cref{C:gluingconditions} to construct $\Omega_{\cE}$ by gluing from $\cE$. That $\logZ_{\cE}$ is holomorphic can be verified in local coordinates -- see \cite{HLJR}*{Thm. 3.9}, for example.
\end{proof}

\begin{rem}
    In the case where $\cE$ is strong, $m(\cE)$ can be taken to be $1$, and in fact $\cS_{\cE}$ can be enlarged by instead imposing the condition that there exists an $\epsilon>0$ such that $\left\lceil\Im\left(\frac{z_{i-1}}{\pi}\right) +\epsilon\right\rceil < \frac{\Im(z_{i})}{\pi} + \epsilon$ for all $i$ such that $E_{i-1}$ and $E_i$ are not orthogonal --  see \cite{Vanjapaper}*{Thm. 3.4}.
\end{rem}

\begin{hyp}
\label{H:Fanoasymptotics}
    Let $X$ be a Fano variety such that $\DCoh(X)$ admits a full ex\-ceptional collection $\cE = (E_1,\ldots, E_N)$ such that for some $\tau \in B$:
    \begin{equation}
    \label{E:Fanoasymptotics}
        \log \cZ^{\tau}_w(E_j) = C_j - \frac{u_j}{w} +  \frac{\dim X}{2}\log(2\pi w) + \epsilon_j(w)
    \end{equation} 
    such that $\epsilon_j(w)\to 0$ as $w\to 0$ in a sector $\mathscr{S}\subset \bf{C}^*$, for some $C_j \in \bf{C}$. Choose a $\tau$-admissible $\varphi \in \bf{R}$ such that $\bf{R}_{>0}\cdot e^{\mathtt{i}\varphi} \subset \mathscr{S}$. We index $\sigma(\cE_\tau) = (u_1,\ldots,u_N)$ such that $i<j$ implies that $\Im(-e^{-\mathtt{i}\varphi}u_i) \le \Im(-e^{-\mathtt{i}\varphi}u_j)$ and write $E_i \approx E_j$ if $u_i = u_j$.
\end{hyp}

The asymptotic estimate in \Cref{H:Fanoasymptotics} occurs, for example, in the context of \eqref{E:logZestimatesemisimple}.

\begin{lem}
\label{L:extensionoflog}
    Assume \Cref{H:Fanoasymptotics}. Then there exists a path in $\Stab(X)$ of the form $\sigma_{t,\varphi}^\tau = (\cZ_{te^{\mathtt{i}\varphi}}^{\tau},\cP_t)$ which is quasi-convergent as $t\to 0$.
\end{lem}

\begin{proof}
    We write $\cZ_t = \cZ_{te^{\mathtt{i}\varphi}}^{\tau}$. Let $\psi$ be small enough that for all $t\in (0,\psi)$, we have $\max_{j=1}^N \{\lvert \epsilon_j(t)\rvert \} < \delta/4$. Next, choose branch cuts such that 
    \begin{equation}
    \label{E:Fanoasymptotic2}
        \log \cZ_t(E_j) \approx C_j + \frac{\dim X}{2}\left(\log(2\pi t) + \mathtt{i}\varphi\right) - \frac{u_j(\tau)}{e^{\mathtt{i}\varphi}\cdot t} + 2\pi \mathtt{i} n_j 
    \end{equation}
    for $n_j \in \bf{Z}$ and such that $n_i = n_j$ if $u_i = u_j$. Letting $z_j(t) = \log \cZ_t(E_j)$, we see that $z(t) \in \bf{C}^N$ enters $\cS_{\cE}$ for $t$ sufficiently close to zero. So, $z(t)$ determines a unique path of stability conditions $\sigma_t$ in $\Omega_{\cE}$. 

    To see that $\sigma_t$ is quasi-convergent, we check the conditions of \Cref{D:quasiconvergent}. Let $1\le i_1<\cdots< i_k\le n$ denote the indices where the value of $u_i$ changes, i.e. $u_{i_{*-1}} \ne u_{i_{*-1}+1} = \cdots = u_{i_*} \ne u_{i_*+1}$. Note that $\DCoh(X) = \langle \cD_1,\ldots, \cD_k\rangle$, where $\cD_a = \langle E_j: i_{a-1}<j\le i_a\rangle$ and that our hypotheses on $\Omega_{\cE}$ imply that all of the objects in $\cE$ are limit semistable. Then, consider any non-zero object $F$ of $\DCoh(X)$ and let $F_k\to F_{k-1}\to \cdots \to F_1 \to F_0 = F$ be its canonical filtration with respect to $\DCoh(X) = \langle \cD_1,\ldots, \cD_k\rangle$ so that $G_a = \Cone(F_a\to F_{a-1}) \in \cD_a$ for $a=1,\ldots, k$. 
    
    We claim that this is the limit Harder-Narasimhan filtration of $X$. If the non-zero $G_a$ are limit semistable, then we are done by the definition of the order $(u_1,\ldots, u_N)$. To see that each non-zero $G_a$ is limit semistable, note that all non-zero objects of $\cD_a$ are limit semistable, being iterated extensions of exceptional objects $E_j$ that obey the estimates of \eqref{E:Fanoasymptotic2}. So, $\lim_{t\to\infty} \phi_t^+(Y) - \phi_t^-(Y) = 0$ and \Cref{D:quasiconvergent}(1) follows.

    The argument of the previous paragraph shows that the limit semistable objects of $\sigma_t$ are $\bigcup_{a=1}^k (\Ob(\cD_a)\setminus \{0\})$. It is then an exercise using \eqref{E:Fanoasymptotic2} to verify \Cref{D:quasiconvergent}(2). 
\end{proof}

\begin{rem}
    In the setting of \Cref{L:extensionoflog}, the quasi-convergent path $\sigma_t$ can also be shown to converge to an admissible boundary point of the space $\Astab(X)$ of augmented stability conditions of $\DCoh(X)$ \cite{augmented}, with the underlying multi-scale line $\Sigma$ having two levels. Besides the root, there are terminal components in bijection with the categories $\cD_1,\ldots, \cD_k$.
\end{rem}

\subsection{NMMP Conjectures for Fano Varieties}
\label{SS:conjectures}
In this section, we elaborate on the NMMP of Halpern-Leistner \cite{NMMP} in the special case of $\DCoh(X)$ for $X$ a smooth Fano variety. In particular, we formulate conjectures which relate analytic properties of the quantum connection and semi\-orthogonal decompositions of $\DCoh(X)$. 

\Cref{conj:NMMPFano} below is an interpretation of \cite{NMMP}*{Proposal III} in the case where $X$ is Fano. The main simplification in this case is that one does not need to consider truncations of the quantum differential equation as in \emph{loc. cit.} At the end of this section, we explain the exact relation between our interpretation of \cite{NMMP}*{Proposal III} and the original statement.

As before, we denote by $\sigma(\cE_\tau)$ the spectrum of $\cE_\tau\star_\tau(-)\in \End(\H^\bullet(X))$, i.e. the multi-set of eigenvalues counted with multiplicity. We write $\lvert \sigma(\cE_\tau)\rvert$ for the underlying set. 

Recall that $\cZ_w^{\tau}$ denotes the quantum cohomology central charge at $\tau \in \H^\bullet(X)$, depending on $e^t = w \in \bf{C}^*$. In the next statement, we fix a norm $\lVert \:\cdot\:\rVert$ on $\H^\bullet(X)$.

\begin{conj}[Halpern-Leistner]
\label{conj:NMMPFano}
    For any Fano variety $X$, there exist $\tau \in \H^2(X)$ and a sector $\mathscr{S} \subset \bf{C}^*$ such that: for any $\tau$-admissible phase $\varphi$ with $\bf{R}_{>0}e^{\mathtt{i}\varphi}\subset \mathscr{S}$, there is a quasi-convergent path $\sigma_{t,\varphi}^\tau = (\cZ_{te^{\mathtt{i}\varphi}}^{\tau},\cP_{t,\varphi})$ in $\Stab(X)$, defined as $t\to 0$, satisfying the following spanning condition: for all $r = \Re(-\lambda e^{-\mathtt{i}\varphi})$ where $\lambda \in \lvert \sigma(\cE_\tau)\rvert$, 
    \[
        F^r\H_{\rm{alg}}^\bullet(X) := \left\{\alpha \in \H_{\rm{alg}}^\bullet(X):\log \lvert \cZ_{te^{\mathtt{i}\varphi}}^\tau(\alpha) \rvert  \le rt^{-1}+o(t^{-1})\text{ as } t\to 0\right\}
    \]
    is spanned by Chern characters of limit semistable objects for $\sigma_{t,\varphi}^\tau$.\footnote{Here, by $f\in o(t^{-1})$ we mean that $\lim_{t\to 0} \frac{\lvert f(t)\rvert}{t^{-1}} = 0$.}
\end{conj}

In \Cref{conj:NMMPFano}, it is not essential that $\tau \in \H^2(X)$. Indeed, one can state the conjecture for any $\tau \in \H^\bullet(X)$, so long as the quantum product is defined at $\tau$ and one has a canonical fundamental solution $\Phi_w^\tau$ of the quantum differential equation.

\begin{conj}\label{conj:noncommutativeGamma}\label{conj:2}
    For any smooth Fano variety $X$ there exist $\tau \in \H^\bullet(X)$, a sector $\mathscr{S}\subset \bf{C}^*$, $\delta>0$, and a holomorphic map \vspace{-2mm}
    \[
        \mathscr{S}\cap \{w:\lvert w\rvert <\delta\}\to \Stab(X), \;\; w\mapsto \sigma^{\tau}_w=(\cZ_{w}^\tau,\cP_{w})
    \]
    such that for any $\tau$-admissible phase $\varphi \in \bf{R}$ for which $\bf{R}_{>0}\cdot e^{\mathtt{i}\varphi}\subset \mathscr{S}$, the path $\sigma^{\tau}_{t,\varphi} := \sigma^{\tau}_{te^{\mathtt{i}\varphi}}$ is quasi-convergent, as $t\to 0$. Furthermore, \vspace{-2mm}
    \begin{enumerate}[label=(\Alph*)]
        \item the semiorthogonal decomposition induced by $\sigma_{t,\varphi}^\tau$ as $t\to 0$ is 
        \begin{equation}
        \label{E:Conj2SOD}
            \DCoh(X) = \langle \cD_\lambda:\lambda \in \lvert \sigma(\cE_\tau)\rvert\rangle
        \end{equation}
        where $\lambda<\mu$  if $\Im(-e^{-\mathtt{i}\varphi} \mu) > \Im(-e^{-\mathtt{i}\varphi} \lambda)$. Also, for any limit semistable object $E\in \cD_\lambda$ we have \vspace{-2mm}
        \begin{equation}
        \label{E:LSSasymptotic}
            \cZ^{\tau}_{te^{\mathtt{i}\varphi}}(E)\sim  C_E\cdot (2\pi t)^{\dim X/2}\cdot  \exp(-e^{-\mathtt{i} \varphi} \cdot \lambda\cdot t^{-1}) \text{ as }t\to 0
        \end{equation}
    for some constant $C_E\in \CC^*$.\vspace{-2mm} \label{conj2SOD}
    \item The dependence of $\sigma_{t,\varphi}^\tau$ on $(\tau,\varphi)$ is continuous, and deformations of $\tau \in B$ and $\varphi \in \bf{R}$ result in mutation equivalent semiorthogonal decompositions as long as $\varphi$ is $\tau$-admissible. \label{conj2independence}\vspace{-2mm}
    \item There exist an isomonodromic deformation $(\nabla^{u})_{u\in U}$ where $U$ is an open connected domain of the deformation space, equipped with a map $\mathscr{B}\hookrightarrow U$ where $\mathscr{B}$ is a submanifold of $\H^\bullet(X)$ containing $\tau$, a sector $\mathscr{S} \subset \bf{C}^*$, constants $\epsilon,\rho>0$, and a holomorphic map 
    \[
        U\times (\mathscr{S} \cap \{w:\lvert w\rvert <\rho + \epsilon\}) \to \Stab(X),\;\; (u,w)\mapsto \sigma^u_w=(\cZ^u_w,\cP_w^u),
    \]
    such that: \label{conj2isomonodromic} \vspace{-2mm}
    \begin{enumerate}
        \item $\sigma_w^u$ is quasi-convergent as $w\to 0$ along any ray-segment in $\mathscr{S} \cap \{w:\lvert w\rvert <\rho + \epsilon\}$;\vspace{-2mm}
        \item the semiorthogonal decompositions obtained from $\sigma_w^u$ and a choice of ray-segment in $\mathscr{S} \cap \{w:\lvert w\rvert<\rho +\epsilon\}$ are mutation equivalent, and\vspace{-2mm}
        \item $\sigma_w^u$ is geometric for all $\rho-\epsilon < \lvert w\rvert < \rho+\epsilon$.\vspace{-2mm}
    \end{enumerate}
    \end{enumerate}
\end{conj}

\begin{rem}
    \Cref{conj:2}\ref{conj2independence} predicts that  semiorthogonal decompositions coming from ``cont\-inuous deformations'' of paths should be related by mutation. In general, it might be expected that any pair of semi\-orthogonal decompositions coming from \Cref{conj:2}\ref{conj2SOD} should be related by a sequence of mut\-ations and autoequivalences of $\DCoh(X)$. However, this type of claim is out of reach at present, since it necessitates an extensive global knowledge of $\Stab(X)$, which is available at present only in several examples -- see, e.g. \cite{HKK}.
\end{rem}

For Fano threefolds, where semiorthogonal decompositions have been extensively studied, see e.g. \cite{KuznetsovFanothreefolds}, we expect that the paths in \Cref{conj:2} will recover these decompositions up to mutation. We can formulate our expectations more precisely for Fano complete intersections. Let $X$ denote a Fano smooth complete intersection in $\bf{P}^n$ of degree $d \le n$. The \emph{Kuznetsov decomposition} of $X$ is the semiorthogonal decomposition: 
\begin{equation}
\label{E:kuznetsovdecomp}
    \DCoh(X) = \langle \Ku(X), \langle\cO_X,\ldots, \cO_X(n-d)\rangle\rangle
\end{equation}
where $\Ku(X) := \langle \cO_X,\ldots, \cO_X(n-d)\rangle^\perp$ is the \emph{Kuznetsov (or residual)} component.

\begin{conj}\label{conj:compl_inter}
    For a Fano complete intersection $X\subset \bf{P}^n$, there exist $\tau\in \H^{\bullet}(X)$ and a sector $\mathscr{S}$ such that \Cref{conj:2}\ref{conj2SOD} holds and the induced semiorthogonal decomposition \eqref{E:Conj2SOD} is a refinement and mutation of the Kuznetsov decomposition \eqref{E:kuznetsovdecomp}.
\end{conj}

\begin{rem}
    We make several more comments on these conjectures:\vspace{-2mm}
    \begin{enumerate}
        \item \Cref{conj:2}\ref{conj2SOD} is motivated by the example of $\bf{P}^n$ which was proven in \cite{Vanjapaper}*{\S 5.2}; see also \cite{NMMP}*{\S3.1} for very strong results in the case of $\bf{P}^1$. In \Cref{S:examples} we verify \Cref{conj:noncommutativeGamma} for Grassmannians and quadrics. For cubic threefolds and fourfolds we verify \Cref{conj:2}\ref{conj2SOD} and \Cref{conj:compl_inter}. \vspace{-2mm}
        \item In contrast to \cite{NMMP}, \Cref{conj:noncommutativeGamma}\ref{conj2SOD} and \ref{conj2isomonodromic} rely on the flexibility to deform $\tau \in \H^2(X)$ into an element in a larger parameter space. In the case of a Fano complete intersection, if $\tau$ is constrained to lie in $\H^2(X)$ it may only be possible to lift $\cZ_w^\tau$ to a quasi-convergent path giving rise to a single representative of the mutation class of the decomposition $\DCoh(X) = \langle \Ku(X),\cO_X,\ldots, \cO_X(n-d)\rangle$. In the present work, this phenomenon is observed for cubic threefolds and fourfolds, see \Cref{SS:cubicthreefolds} and \Cref{SS:Cubicfourfolds}.\vspace{-2mm}
        \item Allowing deformations of $\tau$ to $\H^\bullet(X)$ allows greater flexibility in lifting charges. A possible interpretation of this is that restricting to $\tau \in \H^2(X)$, or even $0\in \H^2(X)$, may determine special representatives in the mutation class of the canonical semiorthogonal decomposition predicted by \cite{NMMP}.\vspace{-2mm}
        \item \Cref{conj:2}\ref{conj2isomonodromic} predicts that one can use (a deformation of) the quantum differential equation to flow from geometric regions of $\Stab(X)$ to regions glued from semiorthogonal decomp\-ositions and vice versa. This prediction is based on the example of $\bf{P}^1$ -- cf. \cite{NMMP}*{Rem. 13}. One might also hope that certain semiorthogonal decompositions are ``better'' than others, in that their corresponding regions constructed by gluing contain geometric stability conditions. This property may not be preserved by mutation.\vspace{-2mm}
    \end{enumerate}
\end{rem}

\subsubsection*{Relationships between the conjectures}
Next, we will examine the relationships between certain parts of the conjectures above. The relationship between the conjectures is summarized as follows:
\[
    \begin{tikzcd}[arrows=Rightarrow]
        \text{Gamma II}\arrow[rr] && \text{\Cref{conj:2}\ref{conj2SOD}}\arrow[rr,"\tau \in \H^2(X)\text{ simple}"] && \text{NMMP \Cref{conj:NMMPFano}}.
    \end{tikzcd}
\] 
This will allow us to deduce NMMP \Cref{conj:NMMPFano} in new cases: Grassmannians, smooth and projective quadrics, and toric Fano varieties. 

\begin{lem}
\label{L:NMMPFanofromConj2(1)}
    If $\tau \in \H^\bullet(X)$ of \Cref{conj:2}\ref{conj2SOD} can be chosen to lie in $\H^2(X)$, the elements of $\sigma(\cE_\tau)$ occur with multiplicity one, and \eqref{E:Conj2SOD} underlies a full exceptional collection, then NMMP \Cref{conj:NMMPFano} holds.
\end{lem}

\begin{proof}
    We only have to check the spanning condition. For this, choose an enumeration of $ \sigma(\cE_\tau) = \{u_1,\ldots, u_N\}$ such that $\Re(-u_1e^{-\mathtt{i}\varphi}) < \cdots < \Re(-u_Ne^{-\mathtt{i}\varphi})$. Write $\cD_i$ for the semiorthogonal factor of \eqref{E:Conj2SOD} corresponding to $u_i$. We get $\rm{K}_0^{\rm{top}}(X) = \bigoplus_{i=1}^N \rm{K}_0^{\rm{top}}(\cD_i)$ and, defining $\H_{\rm{alg}}(\cD_i) := \im(\rm{K}_0(\cD_i) \to \rm{K}_0^{\rm{top}}(\cD_i))$, we can find a generator of $\rm{H}_{\rm{alg}}(\cD_i)$ consisting of an exceptional limit semistable object $E_i$. The estimate of \Cref{conj:2}\ref{conj2SOD} implies 
    \[
        \log \cZ_{te^{\mathtt{i}\varphi}}^\tau(E_i) \approx \log C_E + \frac{\dim X}{2}\log(2\pi t) - \frac{u_i}{te^{\mathtt{i}\varphi}}
    \]
    as $t\to 0$. Letting $r_i = \Re(-u_ie^{-\mathtt{i}\varphi})$ for $i=1,\ldots, k$, we obtain a split filtration 
    \[
       F^j\H_{\rm{alg}}(X):= F^{r_j}\H_{\rm{alg}}(X) = \bigoplus_{i=1}^j \H_{\rm{alg}}(\cD_j)
    \]
    so that the spanning condition of NMMP \Cref{conj:NMMPFano} holds. Note that here, we use the fact that each $\H_{\rm{alg}}(\cD_j)$ is rank one so that any non-zero element $\alpha$ of $\rm{H}_{\rm{alg}}(\cD_j)$ has $\lvert \cZ_t(\alpha)\rvert = r_jt + o(t)$ as $t\to 0$.
\end{proof}

The next proposition is an extension of \cite{Vanjapaper}*{Thm. 5.11} to the case where the eigenvalues of $\cE_\tau$ have multiplicity. This generalization is crucial, since by \cite{CottiCoalescence} most Grassmannians exhibit a \emph{coalescence} phenomenon wherein $\cE_\tau \in \End(\H^\bullet(\Gr(k,V)))$ has repeated eigenvalues for all $\tau \in \H^2(\Gr(k,V))$.

\begin{prop}
\label{P:GammaIIimpliesConj1}
    If $X$ is a Fano variety for which the Gamma Conjecture II holds at some $\tau \in \H^\bullet(X)$, then \Cref{conj:2}\ref{conj2SOD} holds. 
\end{prop}

\begin{proof}
    Gamma II holds for $X$ in the form of \cite{Galkin_Iritani_Gamma}*{Conj. 4.9}. So, given a $\tau$-admissible phase $\varphi\in \bf{R}$, order $\sigma(\cE_\tau)$ as $(u_1,\ldots,u_N)$ such that $i<j$ implies that $\Im(-u_ie^{-\mathtt{i}\varphi}) < \Im(-u_je^{-\mathtt{i}\varphi})$. Then, there exist a small angular sector $\mathscr{S}\subset \bf{C}^*$ containing $\mathbf{R}_{>0}\cdot e^{\mathtt{i}\varphi}$ and a full exceptional collection $(E_1,\ldots, E_n)$ in $\DCoh(X)$ such that 
    \[
        \log \cZ^{\tau}_w(E_j) \approx \frac{\dim X}{2}\log(2\pi w) - \frac{u_j}{w}.
    \]
    as $w\to 0$ in $\mathscr{S}$ by \cite{GGI16}*{Prop. 2.5.1}. Note that is a unique sequence of indices $0 = j_0 < j_1<\cdots<j_k = n$ such that $j_{a-1}<p\le j_a$ if and only if $u_{p} = u_{j_a}$ for all $a=1,\ldots, k$. Letting $\cZ_t = \cZ_{te^{\mathtt{i}\varphi}}^\tau$, we have:
    \begin{equation}
    \label{E:logZestimate}
        \log \cZ_t(E_j) \approx \frac{\dim X}{2}\left(\log(2\pi t) + i \varphi\right) - \frac{u_j}{te^{\mathtt{i}\varphi}}
    \end{equation}
    for $t\in \bf{R}_{>0}$. Thus, we are in the context of \Cref{H:Fanoasymptotics}, so that to construct $\sigma_{t,\varphi}^\tau$ it suffices to verify the hypothesis of \Cref{L:extensionoflog}; however, this follows from \cite{CDGcoalescent}*{Thm. 4.5(6)}, which shows that if $u_i = u_j$ then $E_i$ and $E_j$ are orthogonal in $\DCoh(X)$. So, if $\cD_a$ is the category generated by $E_j$ for all $i_{a-1}<j\le i_a$ with $a=1,\ldots, k$, then $\Stab(\cD_a) \cong \bf{C}^{i_a-i_{a-1}}$ with coordinates $\logZ(E_j)$ for $i_{a-1}<j\le i_a$. 

    It remains only to characterize the induced semiorthogonal decomposition as in \eqref{E:Conj2SOD}. For this, observe that in the proof of \Cref{L:extensionoflog} the semiorthogonal decomposition arising from $\sigma_{t,\varphi}^\tau$ is $\DCoh(X) = \langle \cD_1,\ldots, \cD_k\rangle$ where $\cD_a = \langle E_p:j_{a-1}<p\le j_a\rangle$. Up to shift, the limit semistable objects in $\cD_a$ are sums objects in $\{E_p:j_{a-1}<p\le j_a\}$. Exponentiating the estimate \eqref{E:logZestimate} and taking sums gives the conclusion.
\end{proof}

\begin{cor}
\label{C:NMMPFanoexamples}
    The NMMP \Cref{conj:NMMPFano} holds for Grassmannians, smooth quadrics, and smooth toric Fano varieties.
\end{cor}

\begin{proof}
    This follows from \Cref{L:NMMPFanofromConj2(1)}, \Cref{P:GammaIIimpliesConj1}, and the works \cites{CDGhelix, GGI16}, \cite{Quadrics}, and \cite{FangZhoutoric}, which prove Gamma II in the respective cases.
\end{proof}

\subsubsection*{Relation to original conjecture}

Before moving on, we make several observations about the relation between \Cref{conj:NMMPFano} and \cite{NMMP}*{Proposal III}. We already mentioned above that because $X$ is Fano in the present work, we don't need to truncate the quantum differential equation as proposed in \emph{loc. cit.} The other main difference between \Cref{conj:NMMPFano} and NMMP Proposal III is that we have phrased the spanning condition differently.

First, there is freedom in \cite{NMMP} to choose the fundamental solution used to define the quantum cohomology central charge. In our framework, the integrand defining the quantum cohomology central charge is always $\Upsilon^\tau_{te^{\mathtt{i}\varphi}}(-) := \Phi^\tau_{te^{\mathtt{i}\varphi}}(\widehat{\Gamma}_X\Ch(-))$ --  see \Cref{defn_perturbed_can_sol} and compare with \cite{NMMP}*{Rem. 12}. 

In \cite{NMMP}, the spanning condition is as in \Cref{conj:NMMPFano} except that $\log \lvert \cZ_{te^{\mathtt{i}\varphi}}^\tau(\alpha)\rvert$ is replaced by $\log \lVert \Upsilon_{te^{\mathtt{i}\varphi}}^\tau(\alpha)\rVert$, where for $\alpha \in \H^\bullet(X)$ we let $\Ch(\alpha) = \sum_{d} (2\pi \mathtt{i})^{d/2}\alpha_d$. If $\int_X \Upsilon_{te^{\mathtt{i}\varphi}}^\tau (\alpha) \ne 0$, then 
\[
    \log \cZ_{te^{\mathtt{i}\varphi}}^\tau(\alpha)  = \frac{\dim X}{2}\log(2\pi te^{\mathtt{i}\varphi}) + \log \int_X \Upsilon_{te^{\mathtt{i}\varphi}}^\tau(\alpha).
\]

\begin{lem}
\label{L:spanningconditions}
    Suppose given a quasi-convergent path $\sigma_{t,\varphi}^\tau = (\cZ_{te^{\mathtt{i}\varphi}}^\tau,\cP_t)$ such that 
    \begin{equation}
    \label{E:liminfratioint}
        \inf_{0\ne E\in \cP_{\sigma_{t,\varphi}^\tau}}\left\{
        \liminf_{t\to 0} \frac{\lvert \int_X\Upsilon^\tau_{te^{\mathtt{i}\varphi}}(E) \rvert}{\lVert \Upsilon^\tau_{te^{\mathtt{i}\varphi}}(E) \rVert } \right\} = C > 0.
    \end{equation}
    Then, the spanning condition of \cite{NMMP}*{Proposal III} for $\sigma_{t,\varphi}^\tau$ is equivalent to the one in \Cref{conj:NMMPFano}.
\end{lem}

\begin{proof}
    First, define $I(\alpha) = \lvert \int_X \alpha \rvert\cdot \lVert \alpha\rVert^{-1}$ for non-zero $\alpha \in \H^\bullet(X)$. If $\alpha = \widehat{\Gamma}_X\Ch(E)$ for $0\ne E\in \cP_{\sigma_{t,\varphi}^\tau}$, write $I(E) = I(\alpha)$. First of all, since $I(\alpha)$ is a continuous function on the unit sphere $\{\alpha \in \H^\bullet(X):\lVert \alpha \rVert = 1\}$, there exists $M > 0$ such that $I(\alpha) \le M$ for all non-zero $\alpha$. Using the constants $C$ and $M$, one can prove that for any limit semistable $E$ one has 
    \[
        \log \lvert \cZ_{te^{\mathtt{i}\varphi}}(E)\rvert - \log \lVert \Upsilon_{te^{\mathtt{i}\varphi}}^\tau(E)\rVert \in o(t^{-1})
    \]
    so that the two spanning conditions are equivalent.
\end{proof}

\begin{prop}
\label{P:semisimpleimpliesliminf}
    Let $\tau \in \H^\bullet(X)$ be given at which Gamma II holds and suppose that $\sigma_{t,\varphi}^\tau = (\cZ^\tau_{te^{\mathtt{i}\varphi}},\cP_t)$ is as in \Cref{P:GammaIIimpliesConj1} as $t\to 0$ such that $\cE_\tau\star_\tau(-)$ is semisimple with distinct eigenvalues. Then, the hypothesis of \Cref{L:spanningconditions} holds.
\end{prop}

\begin{proof}
    The set of limit semistable objects is $\{E_i^{\oplus s}[t]:i=1,\ldots, N, s \in \bf{Z}_{\ge 1}, t\in \bf{Z}\}$, where $(E_1,\ldots, E_N)$ is the asymptotically exponential exceptional collection given by Gamma II at $\tau$ with phase $\varphi$. Thus, by the discussion in \Cref{SS:asymptotics} and in particular \Cref{P:asymptoticsofsection} we have that $\Upsilon_{te^{\mathtt{i}\varphi}}^\tau(E_i) \sim \exp(-u_ie^{-\mathtt{i}\varphi}t^{-1}) \cdot \int_X \Psi_\tau(e_i)$, where $\int_X\Psi_\tau(e_i)\ne 0$. It follows that 
    \[
        \lim_{t\to0} \frac{\lvert \int_X\Upsilon_{te^{\mathtt{i}\varphi}}^\tau(E_i)\rvert}{\lVert\Upsilon_{te^{\mathtt{i}\varphi}}^\tau(E_i)\rVert} = C \cdot \frac{\lvert\int_X \Psi_\tau(e_i)\rvert}{\lVert \Psi_\tau(e_i)\rVert} > 0
    \]
    and the result now follows from the description of limit semistable objects.
\end{proof}

\begin{cor}
\label{C:originalNMMPholds}
    The original NMMP \cite{NMMP}*{Proposal III} holds for Grassmannians and quadrics.
\end{cor}

\begin{proof}
    This is a consequence of \Cref{C:NMMPFanoexamples} and \Cref{P:semisimpleimpliesliminf}, noting that there exist $\tau \in B \subset \H^\bullet(X)$ in both of these cases such that $\cE_\tau \star_\tau(-) \in \End(\H^\bullet(X))$ satisfies the requisite hypotheses.
\end{proof}

\begin{rem}
    Note, however, that $\tau$ must be allowed to lie in $\H^\bullet(X)$ rather than just $\H^2(X)$ for most Grassmannians by \cite{CottiCoalescence}. However, $\tau = 0$ works for $\bf{P}^n$.
\end{rem}

The main case where the spanning condition was applied in \cite{NMMP} was the one where the quantum cohomology of $X$ is semisimple, to deduce the existence direction of Dubrovin's conjecture. Thus, at least for such considerations our rephrasing in \Cref{conj:NMMPFano} is adequate. In the present work, however, we consider the non-semisimple examples of cubic hypersurfaces, where \Cref{L:spanningconditions} can no longer be applied. It seems that in these cases, the more convenient spanning condition to use is the one involving $\log \lvert \cZ_w^\tau\rvert$.

\section{Verification of \Cref{conj:noncommutativeGamma} in some examples}
\label{S:examples}

In this section, we use \Cref{S:constructionofstab} to verify \Cref{conj:noncommutativeGamma} for Grassmannians and quadric hyper\-surfaces. We also make progress toward \Cref{conj:noncommutativeGamma} for cubic threefolds and fourfolds and explain how the full statement may follow from the existence of suitable deformations of the quantum connection in these cases.

\subsection{Grassmannians and quadrics}

\begin{figure}
\begin{center}
\begin{tikzpicture}[scale=1.5]


  \draw (-1.5,0) -- (1.5,0);
  \draw (0,-1.5) -- (0,1.5);

  \foreach \x/\y/\name/\dir/\dx/\dy in {
      1/1/{\sqrt{2}\mathtt{i}}/{above}/0/2pt,
      -1/1/{-\sqrt{2}}/{left}/-2pt/0,
      -1/-1/{-\sqrt{2}\mathtt{i}}/{below}/0/-2pt,
      1/-1/{\sqrt{2}}/{right}/2pt/0
  } {
    \pgfmathsetmacro{\xr}{\x*cos(45) - \y*sin(45)}
    \pgfmathsetmacro{\yr}{\x*sin(45) + \y*cos(45)}
    \fill[black] (\xr,\yr) circle (1pt)
      node[\dir, xshift=\dx, yshift=\dy] {\small $\name$};
  }

  \fill[black] (0,0) circle (2pt)
    node[above right, xshift=2pt, yshift=2pt]{\small $0$};


  \begin{scope}[xshift=6cm]

    \draw (-1.5,0) -- (1.5,0);
    \draw (0,-1.5) -- (0,1.5);

    \foreach \x/\y/\name in {
        1/1/{\lambda_0}, 
        -1/1/{\lambda_1}, 
        -1/-1/{\lambda_3}, 
        1/-1/{\lambda_2}
    } {
      \pgfmathsetmacro{\xr}{\x*cos(15) - \y*sin(15)}
      \pgfmathsetmacro{\yr}{\x*sin(15) + \y*cos(15)}
      \fill[black] (\xr,\yr) circle (1pt)
        node[above right, xshift=2pt, yshift=2pt] {\small $\name$};
    }

    \fill[black] (0,0) circle (2pt)
      node[above right, xshift=2pt, yshift=2pt]{\small $0$};

  \end{scope}

  \draw[->] (2,0) -- (4,0)
    node[midway, above=4pt, font=\small]{$w \mapsto e^{-\mathtt{i}\theta} w$};

\end{tikzpicture}
\end{center}
\caption{On the left side is the spectrum of $\tfrac{1}{4}\cE_\tau$. The value $0$ appears with multiplicity two. On the right side is the spectrum of $\cE_\tau/e^{\mathtt{i}\theta}$. After applying a small rotation, the eigenvalues are in general position so that $\Im(\lambda_3)<\Im(\lambda_2)<0<\Im(\lambda_1)<\Im(\lambda_0)$.}
\label{fig:spectrum}
\end{figure}
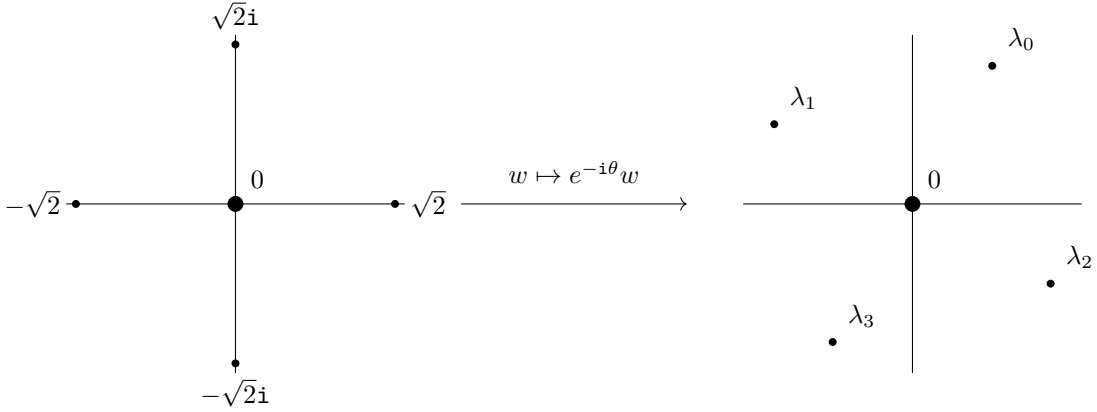

    By \Cref{C:NMMPFanoexamples}, we know that \Cref{conj:NMMPFano} holds for Grass\-mannians $\Gr(k,V)$ and smooth quadrics. We first consider \Cref{conj:noncommutativeGamma} for $\Gr(2,4)$, which is itself a quadric hypersurface in $\bf{P}^5$.
    
    For $X = \Gr(2,4)$, the spectrum of $\cE_\tau$ at $\tau = 0$ is
    \[
        \sigma(c_1(X)\star_0(-)) = \left\{ 4e^{\pi i/4}\left(\zeta_4^{i_1}+\zeta_4^{i_2}\right):0\le i_1<i_2\le 3\right\},
    \]
    taken with multiplicity, where $\zeta_4 = \mathtt{i}$ -- see \cite{GGI16}*{Rem. 6.2.9}. See \Cref{fig:spectrum} for a visualization. This example is \emph{coalescent} in that for all $\tau \in \H^2(X)$, $\cE_\tau$ has repeated eigenvalues -- see \cite{CottiCoalescence}. By \cites{CDGhelix,GGI16}, the Gamma II conjecture for $\Gr(2,4)$ has a solution at $\tau = 0$ given by the twisted Kapranov collection $\mathfrak{K}\otimes \cL = \{E_\mu:=\Sigma^\mu S \otimes \cL\}_{\mu}$ where $\mu$ ranges over the Young diagrams corresponding to $\{(0), (1,0), (2,0), (1,1), (2,1), (2,2)\}$ and $\cL = \det\Lambda^2(S^\vee)$.
    The exceptional coll\-ection $\mathfrak{K}\otimes \cL$ has a grading such that $(2,2)\prec (2,1)\prec (2,0)\sim (1,1)\prec (1,0) \prec (0)$ and norm $\nu(a,b) = a+b$. Note that $E_{(1,1)}$ and $E_{(2,0)}$ are mutually orthogonal. 
    
    By \Cref{T:geomexistence}, the glued region associated to the Kapranov collection $\mathfrak{K} = \{\Sigma^\mu S\}_\mu$ contains geom\-etric stability conditions. Since $\cO_x\otimes \cL \cong \cO_x$, the glued region associated to $\mathfrak{K}\otimes \cL$ also contains geometric stability conditions. Let $C = 4e^{\pi \mathtt{i}/4}$. The quantum cohomology central charge satisfies:
    \begin{equation}
    \label{E:estimatesGr24}
    \begin{split}
        \cZ_w(E_{(2,2)}) & \sim e^{-(1+\mathtt{i})C/w}\\
        \cZ_w(E_{(2,1)}) &\sim e^{(1-\mathtt{i})C/w}\\
        \cZ_w(E_{(1,1)}) \sim \cZ_w(E_{(2,0)})& \sim 1 \\
        \cZ_w(E_{(1,0)}) & \sim e^{(\mathtt{i}-1)C/w}\\
        \cZ_w(E_{(0)}) & = e^{(1+\mathtt{i})C/w}
    \end{split}
    \end{equation}
    as $w\to 0$ in a sector $\mathscr{S}\subset \bf{C}^*$ containing $\bf{R}_{\ge 0}$. The special feature of $\Gr(2,4)$ is that the asymptotically exponential exceptional collection is obtained from the standard Kapranov collection by twisting by $\mathcal{L}$. We first prove a version of \Cref{conj:noncommutativeGamma} in this special case.

\begin{prop}
\label{P:Gr24}
    The conclusion of \Cref{conj:noncommutativeGamma}\ref{conj2SOD} holds for $X = \Gr(2,4)$ at $\tau = 0$ with modified quantum cohom\-ology central charge
    \[
        \cZ^D_w(-) := 4\pi^2 w^2 \int_{X} \Phi_w\cdot D\cdot \hat{\Gamma}\Ch(-)
    \]
    where $D \in \GL(\cH)$ is diagonal with respect to the asymptotically exponential basis and $w\in \bf{R}_{>0}\cdot e^{\mathtt{i}\theta}$ for $\theta$ a small positive number. The resulting semiorthogonal decomposition is 
    \[
        \DCoh(X) = \langle E_{(2,2)},E_{(2,1)},\langle E_{(2,0)},E_{(1,1)}\rangle, E_{(1,0)},E_{(0)}\rangle.
    \]
\end{prop}

In particular, no isomonodromic deformation is needed to begin the path of stability conditions in the geometric region in this case, at the cost of slightly modifying the fundamental solution.

\begin{proof}
    Let $w = te^{\mathtt{i}\theta}$ for $\theta$ a small positive number and $t\in \bf{R}_{>0}$. Then, $\sigma(\cE_\tau/e^{\mathtt{i}\theta}) = e^{-\mathtt{i}\theta}\sigma(\cE_\tau)$ is in general position as in \Cref{fig:spectrum}. The argument of \cite{Vanjapaper}*{Thm. 5.11} combined with \Cref{L:extensionoflog} implies that after applying $D \in \GL(\cH)$ which is diagonal with respect to $\{\widehat{\Gamma}_{X} \Ch(E_\mu)\}_{\mu}$, there exist $t_0$ and $\epsilon>0$ such that \vspace{-2mm}
     \begin{enumerate}
         \item $\forall\:t\in (t_0-\epsilon,t_0+\epsilon)$ there is a lift $\sigma_t = (\cZ_t^D,\cP_t)$ such that all of $\{E_\mu\}$ are $\sigma_t$-stable with
        \begin{equation}
        \label{E:phaseinequality}
            \phi(E_{(0)}) < \phi(E_{(1,0)}[1]) < \mu_- < \mu_+ < \phi(E_{(2,1)}[3]) < \phi(E_{(2,2)}[4])
        \end{equation}
        where $\mu_- = \min\{\phi(E_{(2,0)}[2]),\phi(E_{(1,1)}[2])\}$ and $\mu_+ = \max \{\phi(E_{(2,0)}[2]),\phi(E_{(1,1)}[2])\}$; and \vspace{-2mm}
        \item $\phi(E_{(2,2)}[4]) - \phi(E_{(0)}) < 1$. \vspace{-2mm}
     \end{enumerate}
    Thus for some $\xi \in \bf{R}$, there is $\mathtt{i}\xi \cdot \sigma_t =: \tau_t = (e^{\mathtt{i}\xi}\cdot \cZ_t^D,\cP'_t)$ for all $t\in (t_0-\epsilon,t_0+\epsilon)$ such that
    \[
        \cP_t'(0,1] = \langle E_{(2,2)}[4],E_{(2,1)}[3],E_{(2,0)}[2],E_{(1,1)}[2], E_{(1,0)}[1],E_{(0)}\rangle_{\rm{ext}}.
    \]
    Up to modifying $D$ to rescale $\widehat{\Gamma}_X \Ch(E_{(0)})$ by a large positive real number, we may assume that $\phi(E_{(0)}) < \phi(\cO_x) < \phi(E_{(1,0)}[1])$ for all $x\in X$ in addition to \eqref{E:phaseinequality}. So, $\tau_t$ lies in the geometric region of $\Stab(X)$ for all $t$ sufficiently near $t_0$ by \Cref{P:efficientstable}. Thus, $\sigma_t$ is also geometric for all $t\in (t_0-\epsilon,t_0+\epsilon)$ close to $t_0$. 
    
    One can verify that $\sigma_t$ extends to $(0,t_0+\epsilon)$ in the glued region associated to $\mathfrak{K}\otimes \cL$ using the description of $\cS^6$ in \Cref{L:gluingregion}. Finally, since $\sigma_t$ lies in this glued region as $t\to 0$, the objects $E_\mu$ are all limit semistable. It is an exercise to verify that $\sigma_t$ is quasi-convergent as $t\to 0$ with limit semistable objects $\{E_\mu\}$ up to sums and shifts; the estimates in \eqref{E:estimatesGr24} allow one to verify that the induced semiorthogonal decomposition from \cite{HLJR}*{Thm. 2.37} is $\DCoh(X) = \langle E_{(2,2)},E_{(2,1)},\langle E_{(2,0)},E_{(1,1)}\rangle, E_{(1,0)},E_{(0)}\rangle$. 
\end{proof}

In the next theorem, $X$ is either a Grassmannian $\Gr(k,V)$ or a smooth quadric hypersurface $Q \subset \bf{P}^n$. In both cases, Kapranov \cite{Kapranov1988} has constructed full exceptional collections for $\DCoh(X)$, which we used in \Cref{SS:geomstabilityfromFEC} to construct geometric stability conditions. These collections are called \emph{Kapranov collections} in the respective cases, and denoted $\mathfrak{K}$. The Kapranov collection for $\Gr(k,V)$ is given in \Cref{Ex:basicexamples} and that of a quadric is given in the paragraph below \eqref{E:quadricres}.

Given functions $f,g:\bf{R}_{>0} \to \bf{R}$ we will write $f<g$ if $\liminf_{t\to 0} g(t)-f(t) > 0$, $f\approx g$ if $\lim_{t\to 0} g(t)-f(t) = 0$, and $f\lesssim g$ if $f<g$ or $f\approx g$. 

\begin{thm}
\label{T:NCgammaforGrandQ}
    \Cref{conj:noncommutativeGamma} holds for $X = \Gr(k,V)$ or a smooth quadric hypersurface $Q\subset \bf{P}^n$. 
\end{thm}

\begin{proof}
    In these cases, \Cref{conj:noncommutativeGamma}\ref{conj2SOD} follows from \Cref{C:NMMPFanoexamples} and \ref{conj2independence} is a consequence of Gamma II; see \cite{GGI16}*{Rem. 4.6.3}. So, we consider \ref{conj2isomonodromic}. The key point is that the results of \Cref{S:constructionofstab} only allow us to construct geometric stability conditions in glued regions coming from $\mathfrak{K}$ or its twists by $\Pic(X)$. After perturbing $\tau = 0 \in \H^2(X)$ to $\eta\in \H^{\bullet}(X)$ such that $\cE_\eta$ has distinct eigenvalues $u_1,\ldots, u_N$, we are in \Cref{S:mutationsetup}, where the asymptotically exponential collection $\mathfrak{E}$ is obtained as a mutation of $\mathfrak{K}$. Thus, by \Cref{P:isomonodromicasymptoticsolution} we can isomonodromically deform to a new connection $\nabla^{v}$ on the bundle $\H^{\bullet}(X)\times \bf{P}^1\to \bf{P}^1$ indexed by $v\in \widetilde{\mathscr{U}}_N$ such that \vspace{-2mm}
    \begin{enumerate}
        \item $v = \{v_i\}$ is ordered such that $\Im(v_i) - \Im(v_{i+1}) > 4\pi$ for all $i$; and \vspace{-2mm}
        \item $\mathfrak{K} = (K_1,\ldots, K_N)$ is asymptotically exponential and the exponent of $K_i$ is $v_i$ for each $i=1,\ldots, N$. \vspace{-2mm}
    \end{enumerate}
    It follows that $\varphi = 0$ is an admissible phase. Recall from \Cref{Ex:basicexamples} that $\mathfrak{K}$ is indexed by Young diagrams with at most $n-k$ rows and $k$ columns and is graded such that the norm $\nu$ counts the number of cells in a given Young diagram. Also, $K_N = \cO_X$.

    By \cite{GGI16}*{\S 2.5}, solutions to $\nabla^{v}_{w\partial_w} = 0$ are constructed using the Laplace dual connection $\widehat{\nabla}^{v}$, which is flat and has logarithmic singularities along the simple normal crossings divisor $D := Z(\prod_{j=1}^N(\lambda - v_j))$. For each $i=1,\ldots, N$ one constructs a local $\widehat{\nabla}$-flat section of the trivial $\H^{\bullet}(X)$-bundle over $\widetilde{\mathscr{U}}_N\times \bf{P}^1$, called $\hat{y}_i(u,\lambda)$, such that $\hat{y}_i(v,v_i) = \Psi_v(e_i)$. \emph{A priori}, it is defined on a small open neighborhood of $(v,\lambda)$ but by flatness of $\widehat{\nabla}^v$ we can uniquely extend $\hat{y}_i(u,\lambda)$ to any simply connected neighborhood of $(v,\lambda)$ in $\widetilde{\mathscr{U}}_N \times \bf{C}_\lambda$, away from $D$. We define such a simply connected neighborhood as a product $\Omega \times T_i$, where: \vspace{-2mm}
    \begin{enumerate}
        \item $\Omega \subset \widetilde{\mathscr{U}}_N$ is the connected component of the preimage of the evenly covered set $\{u\in \mathscr{U}_N: \max_{i}\lvert u_i - v_i\rvert < 2\pi\}$ containing $v$;  and \vspace{-2mm}
        \item $T_i$ is a thin tube-domain around the ray $v_i+ \bf{R}_{\ge 0}$. \vspace{-2mm}
    \end{enumerate}
    Then, the corresponding local flat section of $\nabla$ is 
    \[
        y_i(u,w) := \frac{1}{w}\int_{v_i+\bf{R}_{\ge 0}} \hat{y}_i(u,\lambda)e^{-\lambda/w}\:d\lambda.
    \]
    For any $u\in \Omega$, $\widehat{\nabla}^u$ has a regular singularity at $\infty$ so the growth of $\hat{y}_i(u,\lambda)$ is polynomial as $\lambda \to \infty$. Next, define a map $Y_u(w): \bf{C}^N \to \H^\bullet(X)$ by $Y_u(w)(e_i) = y_i(u,w)$ for each $i$. We can apply \Cref{P:Watson} and \Cref{C:estimatevectorvalued} to obtain an asymptotic estimate 
    \begin{equation}
    \label{E:asymptoticestimateforproof}
        Y_u(z)e^{U/w} \sim \Psi_u\left(\id + \sum_{k\ge 0}R_k(u)w^{k+1}\right)
    \end{equation}
    where $U = \diag(u_1,\ldots, u_N)$ and $R_k(u)$ depends analytically on $u$ such that $\lVert Y_u(w)e^{U/w} - \Psi_u\rVert \to 0$ as $\lvert w\rvert \to 0$ independently of $u\in \Omega$.

    By definition \Cref{defn_perturbed_can_sol}, the quantum cohomology central charge at $u \in \widetilde{\mathscr{U}}_N$ is 
    \[
        \cZ_w^{u}(-) = (2\pi w)^{\dim X/2}\int_X \Phi^u_w\left(\widehat{\Gamma}_X\Ch(-)\right).
    \]
    Applying the definition of asymptotically exponential collection \eqref{E:asymptoticexponentialcollection} and evaluating at $E_i$ we have $\cZ_w^u(E_i) = (2\pi w)^{\dim X/2} \int_X y_i(u,w)$. Thus, \eqref{E:asymptoticestimateforproof} gives
    \begin{equation}
    \label{E:estimatelogZdeformed}
        \log \cZ_w^u(E_i) = \frac{1}{2} \log \int_X \Psi_u(e_i) + \frac{\dim X}{2} \log(2\pi w) - \frac{u_i}{w} + \Delta(w),
    \end{equation}
    where $\Delta(w)$ is an error term which tends to zero as $\lvert w\rvert \to 0$ uniformly for $u\in \Omega$. Since $\varphi = 0$ is an admissible phase, we consider the path $\cZ_t^u(E_i)$ for $w(t) = t \in \bf{R}_{>0}$. Let a small $\delta > 0$ be given and consider $t_0>0$ small enough that $\Delta(t_0) < \delta$ for all $0<t\le t_0$. 
    
    Given an object $E$ of $\DCoh(X)$ such that $\cZ_{t_0}^u(E) \ne 0$, let $\vartheta(E) := \cZ_{t_0}^u(E)/\lvert \cZ_{t_0}^u(E)\rvert \in S^1$. Applying \Cref{L:approximationlemma} to the imaginary part of \eqref{E:estimatelogZdeformed} for $w = t$ for each $i$, we can choose $u \in \Omega$ such that:\vspace{-2mm}
    \begin{enumerate}[label=(\alph*)]
        \item $\vartheta(K_i[\nu(i)]) \in e^{\mathtt{i}(0,\pi)}$ for all $i=1,\ldots, N$ and $\vartheta(K_N[\nu(N)])< \cdots < \vartheta(K_1[\nu(1)]),$ where we order by comparing arguments in $(0,\pi)$; and \vspace{-2mm}
        \item  $\vartheta(K_N[\nu(N)]) = \vartheta(\cO_X) = \exp(\mathtt{i}\pi/2)$ and $\vartheta(K_i[\nu(i)])$ is close enough to $1$ for all $i=2,\ldots,N$ that $\vartheta(\cO_X)<\vartheta(\cO_x) <\vartheta(K_{N-1}[\nu(N-1)])<\cdots$ \vspace{-2mm}
    \end{enumerate}  
    Then, by \Cref{L:gluingregion} there is a lift of $\cZ_{t_0}^u$ to a geometric stability condition $\sigma_{t_0}^u\in \Stab(X)$ which has underlying heart $\cA = \langle K_i[\nu(i)]:i=1,\ldots,N\rangle_{\rm{ext}}$. The estimates of \eqref{E:estimatelogZdeformed} now imply that $\sigma_{t_0}^u$ extends to a path $(0,t_0+\epsilon)\to \Stab(X)$ such that $\sigma_t^u$ is geometric for all $t\in (t_0-\epsilon,t_0+\epsilon)$, by \Cref{L:gluingregion}. 
    
    By openness of the geometric region in $\Stab(X)$, we can find a thin angular sector $\mathscr{S}$ around $\bf{R}_{>0}$ such that lifts $\sigma_w^u$ of $\cZ_w^u$ to the glued region of $\mathfrak{K}$ exist for all $w\in \mathscr{S} \cap \{w:\lvert w\rvert< 1+\epsilon\}$ and such that $\sigma_w^u$ is geometric for $w\in \mathscr{S} \cap \{w:t_0-\epsilon<\lvert w\rvert < t_0+\epsilon\}$, up to shrinking $\epsilon$. The result for quadrics is proven in the same fashion, using the fact that the asymptotic\-ally exponential full exceptional collection of $\DCoh(Q)$ is constructed from the Kapranov collection by mutation in \cite{Quadrics}*{\S 6}.
\end{proof}

\begin{lem}
\label{L:approximationlemma}
    Let $\Omega \subset \bf{C}^m$ be an open domain and consider continuous functions 
    \[
        f_i(u_1,\ldots, u_m):\Omega \to \bf{R}\text{ for all }i=1,\ldots, m
    \]
    and some $u^\circ \in \Omega$. For any $(y_1,\ldots, y_m)\in \bf{R}^m$ and $\delta>0$ there exist $u'\in \Omega$ and $t\in (0,1)$ such that 
    \[
        \left\lvert f_i(u') + \Im(u_i')\cdot t^{-1} - y_i\right\rvert < \delta 
    \]
    for all $i=1,\ldots, m$. Furthermore, up to choosing a new $u'$, $t$ can be taken arbitrarily close to $0$. 
\end{lem}

\begin{proof}
    Choose $\eta>0$ small enough that $\lvert u' - u^\circ\rvert < \eta$ implies that $\lvert f_i(u')-f_i(u^\circ)\rvert < \delta$ for all $i=1,\ldots, m$. Next, let $\Delta_i :=y_i - f_i(u^\circ)- \Im(u_i^\circ)\cdot t^{-1}$. Now, choosing any $t$ sufficiently close to $0$, we can choose $u_i'$ such that $\Im(u_i'-u_i^\circ)\cdot t^{-1} = \Delta_i$. Then
    \begin{align*}
        \lvert y_i - f_i(u') - \Im(u_i')t^{-1}\rvert & \le \lvert y_i - f_i(u^\circ) - \Im(u_i')t^{-1} \rvert + \lvert f_i(u^\circ) - f_i(u')\rvert\\
        &\le 0 + \delta = \delta.
    \end{align*}
    Since the estimate for $y_i$ depends only on modifying $u_i$, we can arrange this for all $i=1,\ldots, m$.
\end{proof}

\subsection{Cubic threefolds}
\label{SS:cubicthreefolds}
Consider a smooth cubic threefold $Y\subset \bf{P}^4$. Recall from \Cref{section_QH_cubics} that $c_1(X)\star_{0}(-) \in \End(\H^\bullet(Y))$ has eigenvalues $-T,0,T$, where $T=2 \sqrt{6} > 0$. By \Cref{cor_qcohZ_cubic3}, we have a semiorthogonal decomposition:
\[
    \DCoh(Y)=\langle \ko(1), \cT,\ko(2)\rangle
\]
such that the quantum cohomology central charge $\cZ_t$ along $w=e^{\mathtt{i}\varphi}t$ for $\varphi \in \bf{R}$ as in \Cref{R:smallersector} and $t\in \bf{R}_{>0}$ has the following asymptotics as $t\to 0$ 
\begin{equation}\label{eq_asymp_cubic3}
    \begin{split}
        \log\cZ_t(\ko(1)) &= C_1 +Te^{-\mathtt{i}\varphi}/t + \left(\frac{3}{2} - m\right)\log(e^{\mathtt{i}\varphi}t) + \epsilon_{\ko(1)}(t)\\
        \log\cZ_t(\ko(2)) &= C_2 -Te^{-\mathtt{i}\varphi}/t + \left(\frac{3}{2} - m\right)\log(e^{\mathtt{i}\varphi}t) + \epsilon_{\ko(2)}(t)\\
        \log\cZ_t(V_i) &= C'_i + \left(\frac{3}{2} - m\right)\log(e^{\mathtt{i}\varphi}t) + \epsilon_i(t)\\
    \end{split}
\end{equation}
for some objects $V_1,V_2\in \Ku(Y)$ such that $(\ch(V_1),\ch(V_2))$ forms a basis of $\rm{K}_0^{\rm{top}}(\Ku(Y))_{\bf{Q}}$ and $\int_Y\Psi(\widehat{\Gamma}_Y\Ch(V_i))\neq 0$, see \Cref{P:asymptoticsofsection} for the notation.
Observe also that
\begin{equation}\label{eq_cubic3_imrepart}
    \begin{split}
        \Re(Te^{-\mathtt{i}\varphi})>0 \text{ and } \Im(Te^{-\mathtt{i}\varphi})<0.
    \end{split}
\end{equation}

\Cref{conj:2}\ref{conj2SOD} holds for cubic threefolds:

\begin{thm}
\label{T:conj2.1threefold}
    There is a quasi-convergent path $\sigma_t$ in $\Stab(Y)$ for $t\in (0,t_0]$ with quantum cohomology central charge satisfying the spanning condition of \Cref{conj:NMMPFano} such that\vspace{-2mm}
    \begin{enumerate}
        \item the induced semiorthogonal decomposition is $\langle \cO(1), \cT,\cO(2)\rangle$\vspace{-2mm}
        \item each factor corresponds to an eigenvalue of $c_1(Y)\star_{0}(-)$, i.e. the asymptotics \eqref{eq_asymp_cubic3} hold.\vspace{-2mm}
    \end{enumerate}
\end{thm}

\begin{proof}
    The relevant estimates of $\cZ_t$ are in \eqref{eq_asymp_cubic3}.
    To construct the path, see \Cref{section_gluingpaths_general} and in particular \Cref{rem_gluing_pathKu}; note that \eqref{eq_cubic3_imrepart} is crucial.
\end{proof}

See \Cref{SS:Epilogue} for a conjectural explanation of how \Cref{conj:2}\ref{conj2isomonodromic} might be deduced in this case. Next, we prove that if we allow a more general class of central charges, corresponding to arbitrary fundamental solutions of the quantum differential equation, we can obtain part of \Cref{conj:2}\ref{conj2isomonodromic}:

\begin{thm}\label{thm_pqccc_geo_cubic3}
    There exists a quasi-convergent path $\sigma_t = (\cZ^A_t,\cP_t):(0,t_0] \to \Stab(Y)$, where $\cZ^A_t:=\cZ_t \circ A$ for $A\in \GL(\H^\bullet(Y))$ satisfies the spanning condition of \Cref{conj:NMMPFano}, such that \vspace{-2mm}
    \begin{enumerate}
        \item the induced semiorthogonal decomposition is $\langle\Ku(Y),\cO,\cO(1)\rangle$; and\vspace{-2mm}
        \item $\sigma_{t_0}$ is geometric.\vspace{-2mm}
    \end{enumerate}
\end{thm}

\begin{proof}
    Let $A$ to be the change of basis matrix from 
    \[
        \ch(V_1),\ch(V_2),\ch(\cO), \ch(\cO(1))
    \]
    to 
    \[
        \ch(\cO(1)),\ch(\cO(1))+\ch(V_2), \ch(V_1),\ch(\cO(2)).
    \]
    The central charge $\cZ^A_t$ satisfies the following asymptotic estimates: 
    \begin{equation}
    \label{E:estimatesthm4.5}
    \begin{split}
        \log\cZ^A_t(V_i) &= C_1 +Te^{-\mathtt{i}\varphi}t^{-1} + \left(\tfrac{3}{2} - m\right)\log(e^{\mathtt{i}\varphi}t) + \epsilon_{\ko(1)}(t)\\
        \log\cZ^A_t(\ko(1)) &= C_2 -Te^{-\mathtt{i}\varphi}t^{-1} + \left(\tfrac{3}{2} - m\right)\log(e^{\mathtt{i}\varphi}t) + \epsilon_{\ko(2)}(t)\\
        \log\cZ^A_t(\ko) &= C'_i + \left(\tfrac{3}{2} - m\right)\log(e^{\mathtt{i}\varphi}t) + \epsilon_{\ko(1)}(t).
    \end{split}
\end{equation}
The only non-trivial case is that of $\cZ^A_t(V_2):=\cZ_t(\ko(1))+\cZ_t(V_2)$; the claimed estimate follows from the observation that $\lvert \cZ_t(\cO(1))\rvert\to \infty$ and $\lvert \cZ_t(V_2)\rvert \to 0$.

Note that ${\cZ^A_t}|_{\Ku(Y)}$ is uniquely determined by the values $\cZ_t^A(V_i)$ for $i=1,2$. Moreover, since $\Im\log\cZ_t(V_2)=\Im C'_2+(3/2-m)\varphi+\Im(\epsilon_2(t))$ converges as $t\to 0^+$, we see that for $t_0\ll1$, the $\bf{R}$-bases $(\cZ_t^A(V_1),\cZ_t^A(V_2))$ of $\bf{C}$ have the same orientation for all $t\in (0,t_0)$. Thus, for some $0<t_0\ll 1$ there exists $g_t\colon (0,t_0]\to \GL_2^+(\bf{R})$ such that $\cZ^A_t=g_t\cdot \cZ^A_{t_0}$ for all $t\in (0,t_0].$

We construct the desired quasi-convergent path by gluing stability conditions along $\DCoh(Y) = \langle \Ku(Y),\cO,\cO(1)\rangle$. By the asymptotics of $\cZ^A_t(\ko(i))$ for $i=0,1$, we have a path $\eta_t\colon (0,t_0]\to \Stab(\langle \ko,\ko(1)\rangle)$ such that $\ko,\ko(1)$ are stable for all $t$, $\phi_t(\ko[1])<\phi_t(\ko(1))$, and $\phi_{t_0}(\ko(1))-\phi_{t_0}(\ko[1])<1$; see \Cref{section_gluingpaths_general} for details. Moreover, for any lift $\widetilde{g}_t$ of $g_t$ to $\GL_2^+(\bf{R})^\sim$ and any $\tau_0 \in \Stab(\Ku(Y))$ with central charge ${\cZ^A_{t_0}}|_{\Ku(Y)}$, we can consider the path $\tau_t:=\widetilde{g}_t\cdot \tau_0, t\in (0,t_0].$

Choose $\tau_0,\eta_{t_0}$ as in \Cref{thm_glued_geom_stab_cubic3} and $t_0\ll 1$ such that the error terms of \eqref{E:estimatesthm4.5} are sufficiently small. One can use these estimates to show that the gluing condition of $\Hom$-vanishing between the hearts of $\tau_t$ and $\eta_t$ verified in the proof of \Cref{thm_glued_geom_stab_cubic3} is satisfied for all $t\in (0,t_0]$. Thus, we can glue $\sigma_t:=\tau_t*\eta_t$ such that $\sigma_{t_0}$ is geometric. We leave as an exercise the verification that $\sigma_t$ is quasi-convergent as $t\to 0$.
\end{proof}

\subsection{Cubic fourfolds}
\label{SS:Cubicfourfolds}
Consider a smooth cubic fourfold $X\subset\bf{P}^5$. By \Cref{section_QH_cubics}, $c_1(X)\star_{0}(-) \in \End(\H^\bullet(X))$ has eigenvalues $Te^{4\pi \mathtt{i}/3},T,0,Te^{2\pi \mathtt{i}/3}$, where $T = 9$. By \Cref{cor_qcohZ_cubic4}, we have a semiorthogonal decomposition
\[
    \DCoh(X)=\langle \ko(2), \bf{R}_{\ko(2)}\ko, \cT,\bf{R}_{\ko(2)}\ko(1)\rangle
\]
such that the quantum cohomology central charge $\cZ_t$ along $w=e^{\mathtt{i}\varphi}t$ for $\varphi$ as in \Cref{R:smallersector} and $t\in \bf{R}_{>0}$ has the following asymptotics as $t\to 0$ 
\begin{equation}\label{eq_asymp_cubic4}
    \begin{split}
        \log\cZ_t(\ko(2)) & = C_1 +Te^{\mathtt{i}(4\pi/3-\varphi)}t^{-1} + \left(2 - m\right)\log(e^{\mathtt{i}\varphi}t) + \epsilon_{\ko(2)}(t)\\
        \log\cZ_t(\bf{R}_{\ko(2)}\ko) & = C_2 +
        Te^{-\mathtt{i}\varphi}t^{-1} 
        + \left(2 - m\right)\log(e^{\mathtt{i}\varphi}t) + \epsilon_{\bf{R}_{\ko(2)}\ko}(t)\\
        \log\cZ_t(V_i) & = C'_i + \left(2 - m\right)\log(e^{\mathtt{i}\varphi}t) + \epsilon_{i}(t)\\
        \log\cZ_t(\bf{R}_{\ko(2)}\ko(1)) & = C_3 +Te^{\mathtt{i}(2\pi/3-\varphi)}t^{-1} + \left(2 - m\right)\log(e^{\mathtt{i}\varphi}t) + \epsilon_{\bf{R}_{\ko(2)}\ko(1)}(t)
    \end{split}
\end{equation}
for objects $V_1,V_2\in \Ku(X)$ such that $(\ch(V_1),\ch(V_2))$ form a basis of $\rm{K}_0^{\rm{top}}(\Ku(X))_{\bf{Q}}$ and such that $\int_X \widehat{\Gamma}_X\Psi(\Ch(V_i))\neq 0$; see \Cref{P:asymptoticsofsection} for the notation. Observe also that
\begin{equation}
    \begin{split}
        \Im(e^{\mathtt{i}(4\pi/3-\varphi)})<\Im(e^{-\mathtt{i}\varphi})<0< \Im(e^{\mathtt{i}(2\pi/3-\varphi)}).
    \end{split}
\end{equation}

As in the case of cubic threefolds,  \Cref{conj:2}\ref{conj2SOD} holds for smooth cubic fourfolds: 

\begin{thm}
\label{T:conj2.1fourfold}
    There is a quasi-convergent path $\sigma_t:(0,t_0] \to \Stab(X)$ with quantum cohomology central charge which satisfies the spanning condition of \Cref{conj:NMMPFano} such that \vspace{-2mm}
    \begin{enumerate}
        \item the induced semiorthogonal decomposition is $ \langle \ko(2), \bf{R}_{\ko(2)}\ko, \cT,\bf{R}_{\ko(2)}\ko(1)\rangle$ \vspace{-2mm}
        \item each factor of the semiorthogonal decomposition corresponds to an eigenvalue of $c_1(X)\star_{0}(-)$, i.e. the asymptotics \eqref{eq_asymp_cubic4} holds.\vspace{-2mm}
    \end{enumerate}
\end{thm}
\begin{proof}
    The necessary estimates of $\cZ_t$ were given in \eqref{eq_asymp_cubic4}. To construct the path see \Cref{section_gluingpaths_general} and \Cref{rem_gluing_pathKu}.
\end{proof}

We also have a version of \Cref{thm_pqccc_geo_cubic3} for cubic fourfolds:

\begin{thm}\label{thm_pqccc_geo_cubic4}
    There exists a quasi-convergent path $\sigma_t=(\cZ^A_t,\cP_t):(0,t_0]\to \Stab(X)$, where $\cZ^A_t:=\cZ_t \circ A$ for $A\in \GL(\H^\bullet(X))$, satisfying the spanning condition of \Cref{conj:NMMPFano} such that\vspace{-2mm}
    \begin{enumerate}
        \item the induced semiorthogonal decomposition is $\langle\Ku(X),\ko,\ko(1),\ko(2)\rangle$; and\vspace{-2mm}
        \item if $X$ does not contain a plane, then we can choose the quasi-convergent path in such a way that $\sigma_{t_0}$ is geometric.\vspace{-2mm}
    \end{enumerate}
\end{thm}

\begin{proof}
The proof is the same as in the threefold case.
Let $A$ be the change of basis matrix that transforms the basis
\[
    \ch(V_1),\ch(V_2),\ch(\ko), \ch(\ko(1)), \ch(\ko(2))
\]
to 
\[
    \ch(\ko(2)),\ch(\ko(2))+\ch(V_2), \ch(\bf{R}_{\ko(2)}\ko),\ch(V_1),\ch(\bf{R}_{\ko(1)}\ko).
\]
The central charge $\cZ_t^A$ has the following asymptotics
    \begin{equation}
        \begin{split}
        \log\cZ^A_t(V_i) & = C_1 +Te^{\mathtt{i}(4\pi/3-\varphi)}t^{-1} + \left(2 - m\right)\log(e^{\mathtt{i}\varphi}t) + \epsilon_{\ko(2)}(t)\\
        \log\cZ^A_t(\ko) & = C_2 +
        Te^{-\mathtt{i}\varphi}t^{-1} 
        + \left(2 - m\right)\log(e^{\mathtt{i}\varphi}t) + \epsilon_{\bf{R}_{\ko(2)}\ko}(t)\\
        \log\cZ^A_t(\ko(1)) & = C'_1 + \left(2 - m\right)\log(e^{\mathtt{i}\varphi}t) + \epsilon_1(t)\\
        \log\cZ^A_t(\ko(2)) & = C_3 +Te^{\mathtt{i}(2\pi/3-\varphi)}t^{-1} + \left(2 - m\right)\log(e^{\mathtt{i}\varphi}t) + \epsilon_{\bf{R}_{\ko(1)}\ko}(t).
    \end{split}
\end{equation}

As in the threefold case for $0<t_0\ll 1$ there is a $g_t\colon (0,t_0]\to \GL_2^+(\bf{R})$ such that $\cZ^A_t=g_t\cdot \cZ^A_{t_0}$ for all $t\in (0,t_0].$

Now we construct desired the path by gluing stability conditions on the semiorthogonal decomp\-osition in the statement. By the asymptotics of $\cZ^A_t(\cO(i))$ for $i=0,1,2$ there is a path $\eta_t: (0,t_0]\to \Stab(\langle \cO,\cO(1),\cO(2)\rangle)$ such that $\cO,\cO(1),\cO(2)$ are stable for all $t$, $\phi_t(\ko[2])<\phi_t(\ko(1)[1])<\phi_t(\ko(2))$, and $\phi_{t_0}(\ko(1)[1])-\phi_{t_0}(\ko[2]), \phi(\ko(2))-\phi(\ko(1)[1])<1$; see \Cref{section_gluingpaths_general} for details. Moreover, for any lift $\widetilde{g}_t$ of $g_t$ to $\GL_2^+(\bf{R})^\sim$ and any $\tau_0\in \Stab(\Ku(X))$ with central charge ${\cZ'_{t_0}}|_{\Ku(X)}$ we can consider the path $\tau_t:=\widetilde{g}_t\cdot \tau_0:(0,t_0]\to \Stab(X)$.

By \Cref{thm_geomstab_cubic4_noplane}, there is a glued geometric stability condition $\sigma_0=\tau_0*\eta_0$.
By the asymptotics of $\cZ^A_t$, for $t_0\ll 1$ the heart of $\eta_t$ is $\cB_t=\langle \ko[-n_0+2],\ko(1)[-n_1+1],\ko(2)[n_2]\rangle_{\mathrm{ext}}$ where $n_2\geq n_1\geq n_0\geq 0$. Similarly, the heart of $\tau_t$ is $\cP_{\tau_0}(f_t(0),f_t(1)]$ where $g_t=(M_t,f_t)$, see \eqref{eq_pathGL}. By the asymptotics of $\cZ^A_t$ we also see that $f_t(0)\leq 0$ for $t\leq t_0$, whence the vanishing 
\[
    \Hom^{\leq 0}(\cP(f_t(0),f_t(1)],\cB_t)=0
\]
for any $t\leq t_0$. Thus, by \Cref{thm:glued} we obtain a path $\sigma_t=\tau_t*\eta_t$ as claimed. We omit the verification that $\sigma_t$ is quasi-convergent.
\end{proof}

\subsection{Epilogue: Future directions}
\label{SS:Epilogue}
In this final section, we explain some future avenues of investigation suggested by the present work. 

\subsubsection*{\Cref{conj:2}\ref{conj2independence} for cubics}
We expect that \Cref{conj:2}\ref{conj2independence} should hold also for Fano hypersurfaces $X$, including the cubics considered in the present work. Much of the work has been completed in the work of Sanda--Shamoto \cite{SandaShamoto}; their framework of mutation systems dictates that as one varies the phase $\varphi$ and crosses Stokes directions\footnote{I.e. those where $\Im(e^{-\mathtt{i}\varphi}(u_i - u_j))$ changes value for some eigenvalues $u_i$ and $u_j$.} the corresponding decomp\-osition of $\DCoh(X)$ undergoes a mutation. It remains to show that this result can be lifted to the space of stability conditions. Indeed, verifying the conjecture is tantamount to showing that as one varies $\varphi\in \bf{R}$ and $\tau \in \H^2(X)$, there exist quasi-convergent lifts $\sigma_{t,\varphi}^\tau$ to $\Stab(X)$ with continuous dependence on $(\tau,\varphi,t)$ and such that the resulting semiorthogonal decompositions are all related by mutation.

\subsubsection*{Isomonodromic deformations for cubics} We explain here a possible path to complete the proof of \Cref{conj:noncommutativeGamma}\ref{conj2isomonodromic} for cubic hypersurfaces. The key idea can be explained using the case of a smooth cubic threefold $Y\subset \bf{P}^4$. In \Cref{SS:asymptotics}, we observed that for $\tau = 0$ the eigenvalues of $c_1(Y)\star_0(-)$ are $\pm 2\sqrt{6}$ and $0$. The non-zero eigenvalues occur with multiplicity one, while the eigenvalue $0$ occurs with multiplicity $2 + \dim \H^\bullet_{\rm{amb}}(Y)^\perp$. 

In particular, under the \emph{Dubrovin-type conjecture} of Sanda--Shamoto \cite{SandaShamoto}*{\S5}, the Kuznetsov component $\Ku(Y)$ corresponds to the class of limit semistable objects $E$ such that $\log \cZ_t(E)$ grows like $\log(\alpha(\varphi) \cdot t)$ for some $\alpha(\varphi) \in \bf{C}^*$, depending on the admissible phase $\varphi \in \bf{R}$. On the other hand, the category $\cD_+$ (resp. $\cD_-$) corresponding $2\sqrt{6}$ (resp. $-2\sqrt{6}$), is generated by limit semistable objects $F$ such that $\log\cZ_t(F)$ has leading order term $a(\varphi)^{-1}2\sqrt{6}t^{-1}$ (resp. $-a(\varphi)^{-1}2\sqrt{6}t^{-1}$) as $t\to 0$. 

Consequently, to lift $\cZ_t$ to a path in $\Stab(Y)$ in a glued region coming from a semiorthogonal decomposition mutation equivalent to \eqref{E:KuznetsovCubicthreefold}, the semiorthogonal decomposition must be of the form $\langle \cD_-,\cT,\cD_+\rangle$, where $\cT\simeq \Ku(Y)$ and $\cD_{\pm}$ are generated by exceptional objects $E_{\pm}$. Indeed, for any limit semistable $E\in \cT$ we have
\[
    \Im \log \cZ_t(E_-) < \Im \log\cZ_t(E) < \Im\log\cZ_t(E_+) \text{ for all }t\gg0,
\]
which determines the ordering of the categories in the resulting decomposition of $\DCoh(Y)$. In \Cref{SS:cubicthreefolds}, we have identified the corresponding decomposition as $\DCoh(Y) = \langle \cO(1),\cT,\cO(2)\rangle$, obtained from \eqref{E:KuznetsovCubicthreefold} by mutation. On the other hand, in \Cref{S:constructionofstab} we have only constructed geometric stability conditions in the glued region of $\DCoh(Y) = \langle \Ku(Y),\cO,\cO(1)\rangle$. It is unclear if this is a technical limitation, or a special feature of the Kuznetsov decomposition. 

To resolve this issue, one could take inspiration from the isomonodromic deformation of the quantum connection over $\widetilde{\mathscr{U}}_N$ studied in \Cref{SS:isomonodromicdeformation} in the semisimple case. Thus, one might expect the following:

\begin{quest}
\label{Q:isomonodromicexist?}
    Let $X$ denote a cubic hypersurface and $n$ the number of distinct eigenvalues of $c_1(X)\star_0(-) \in \End(\H^\bullet(X))$. Does there exist an isomonodromic deformation (\Cref{D:isomonodromic_deformation}) of the quantum connection $\nabla$ over a complex manifold $M$, containing $\H^2(X)$ as a locally closed subspace such that $\pi_1(M,\tau = 0) \cong \mathfrak{B}_n$?
\end{quest}

Suppose \Cref{Q:isomonodromicexist?} was answered affirmatively in the case of the cubic threefold. In this case, one could embed a small neighborhood $\mathscr{B}$ near $\tau = 0$ in $M$ into $\widetilde{M}$, and analytically continue to a flat connection $\widetilde{\nabla}$ over $\widetilde{M}$. After deforming $\tau = 0$ to a suitable value $v \in \widetilde{M}$ corresponding to the element of $\mathfrak{B}_3$ that passes $2$ over $1$, we might expect that the A-model mutation system decomposition:
\[
    \H^\bullet(X) = \bigoplus_{\lambda \in \lvert \sigma(\cE_0)\rvert} \cA_\lambda,
\]
where $\cA_\lambda$ is as in \eqref{E:Amodelmutationsystem} undergoes a mutation along the corresponding element of $\mathfrak{B}_3$. See \cite{SandaShamoto}*{\S2.5} for a precise definition. In particular, the eigenvalue $2\sqrt{6}$ appears with multiplicity $2 + \dim \H^\bullet_{\rm{amb}}(Y)^\perp$, and the eigenvalues $0$ and $-2\sqrt{6}$ appear with multiplicity one. Consequently, the deformed quantum cohomology central charge $\cZ_w^v$ (\Cref{defn_perturbed_can_sol}) should lift to a family of paths $\sigma_w^v$. Taking $w(t) = e^{\mathtt{i}\varphi}t$ for $0<\varphi \ll1$, the corresponding path $\sigma_{t,\varphi}^v$ should lie in the glued geometric region of $\langle \Ku(Y),\cO,\cO(1)\rangle$ for $t$ near some $t_0$, and recover the Kuznetsov decomposition $\DCoh(Y) = \langle \Ku(Y),\cO,\cO(1)\rangle$ as $t\to 0$.

\subsubsection*{Smooth Fano Complete Intersections} Another possible extension of the results in the present work is to the more general case of Fano complete intersections in $\bf{P}^n$. The \emph{Dubrovin-type conjecture} of \cite{SandaShamoto}*{\S 5} is proven in this level of generality, and most of the results of \Cref{S:nMMPFano} still hold in this case. Consequently, the obstruction to proving \Cref{conj:noncommutativeGamma}\ref{conj2SOD} is the existence of stability conditions on the Kuznetsov components in these examples. 

In the case of cubic fivefolds $X$, stability conditions have been shown to exist on $\Ku(X)$ in \cite{Liustabilityfivefolds}. Consequently, there should be a natural extension of our results to that case.

\appendix 

\section{Asymptotic Estimates}

For use in the proof of \Cref{T:NCgammaforGrandQ}, we prove a version of the classical Watson's lemma, used to produce asymptotic expansions of certain integral functions. For our purposes, we need a version of the classical lemma for families of holomorphic functions. 

\begin{setup}
\label{Setup:Watson}
    Consider a holomorphic function $\varphi(u,\lambda)$ where $\lambda$ varies in an open neighborhood $W$ of $\bf{R}_{\ge 0} \subset \bf{C}$ and $u\in \Omega \subset \bf{C}^N$, for $\Omega$ a compact domain. 
    We assume further that there is a $b>0$ such that for any fixed $u$, $\lvert \varphi(u,\lambda)\rvert \le \lvert p_u(\lambda)\cdot e^{b\lambda}\rvert$ as $\lambda \to \infty$, where $p_u(\lambda)$ is a polynomial depending on $u$. We will consider the Taylor series expansion
    \begin{equation}
    \label{E:taylorseries}
        \varphi(u,\lambda) = \sum_{n\ge 0}a_n(u)\lambda^n,
    \end{equation}
    valid on an open disk $\bf{D}_\delta(0)$ for all $u\in \Omega$. Here, $a_n(u)$ is an analytic function of $u$ for each $n$.
\end{setup}

We first prove an auxiliary lemma:

\begin{lem}
\label{L:watsonbound}
    In \Cref{Setup:Watson}, $p_u(\lambda)$ can be chosen to depend analytically on $u$ locally and consequently there exists a polynomial $P(\lambda)$ such that $\lvert \varphi(u,\lambda)\rvert \le \lvert P(\lambda)\rvert$ as $\lambda \to \infty$ in $\bf{R}$.\footnote{That is, there exists $\lambda_0 \in \bf{R}_{>0}$ such that for all $u \in \Omega$ and all $\lambda\ge \lambda_0$ we have $\lvert \varphi(u,\lambda)\rvert \le \lvert P(\lambda)\rvert$.}
\end{lem}

\begin{proof}
    Consider a disk $\bf{D}$ around infinity $\infty \in \bf{P}^1$ with coordinate $\lambda^{-1}$, $\bf{D}^*:=\bf{D}\setminus \{\infty\}$, and for $u\in \Omega$ an open $V \subseteq \bf{C}^N$ a neighborhood of $u$ on which $\varphi(u,\lambda)$ is holomorphic. The function $\varphi(u,\lambda)$ is analytic on $\bf{D}^* \times V$ and, up to shrinking admits, it a Laurent series expansion in $\lambda^{-1} = t$ at $(u,\infty)$, given by 
    \[
        \varphi(u,\lambda) = \sum_{k \ge -n} a_k(u)\cdot t^k
    \]
    where the $a_k(u)$ are analytic functions of $u$. In particular, $p_u(\lambda) = \sum_{k=0}^n \lvert a_k(u)\rvert \cdot \lambda^k + C$, for some $C>0$ gives an upper bound for $\lvert \varphi(u,\lambda)\rvert $ for all $\lambda \gg0$. Up to shrinking $V$ and $\bf{D}$ and taking the maximum of the coefficients, we can find one polynomial $P(\lambda)$ which bounds $\varphi(u,\lambda)$ on $V\times \bf{D}^*$. Since $\Omega$ is compact, a finite cover argument shows that we can choose one $P(\lambda)$ for all $u$.
\end{proof}

\begin{cor}
\label{C:watsonlemma}
    In \Cref{Setup:Watson}, there exist $\epsilon>0$ and $C>0$ such that $\lvert \varphi(u,\lambda)\rvert \le C\lvert e^{(b+\epsilon)\lambda}\rvert$ for all $(u,\lambda) \in \Omega \times \bf{R}_{\ge 0}$.
\end{cor}

\begin{proof}
    By \Cref{L:watsonbound}, we can find $\lambda_0>0$ and a polynomial function $P(\lambda)$ such that $\lvert \varphi(u,\lambda)\rvert \le \lvert P(\lambda)\rvert$ for all $\lambda > \lambda_0$. On the other hand, $\lvert \varphi(u,\lambda)\rvert$ is bounded on $\Omega \times [0,\lambda_0]$ by compactness by some $C>0$. Finally, choose $\epsilon>0$ so that $\lvert P(\lambda)\rvert \le e^{\epsilon\lambda}$ for all $\lambda\ge \lambda_0$. The result now follows.
\end{proof}

\begin{prop}
    [Watson's Lemma] In \Cref{Setup:Watson}, the integral function
    \[
        I(u,z) = \int_0^\infty \varphi(u,\lambda) e^{-\lambda/z}\:d\lambda
    \]
    exists for all $u\in \Omega$ and $z$ with $\Re(z)>0$ and admits an asymptotic expansion
    \[
        I(u,z) \sim \sum_{n=0}^\infty a_n(u)z^{n+1} \text{ as } z\to 0
    \]
    in a sector $\mathscr{S}(0,\tfrac{\pi}{2}+\epsilon)$ for a small fixed $\epsilon>0$. In particular, 
    \[
         \left\lvert I(u,z) - \sum_{n=0}^N a_n(u)\cdot z^n \right\rvert \le O(\lvert z\rvert^{N+1})  
    \]
    independently of $u\in \Omega$.
\label{P:Watson}
\end{prop}

\begin{proof}
    We write $z = x+\mathtt{i}y$. To deduce existence of $I(u,z)$, we note that when $\Re(z)>0$, by \Cref{L:watsonbound} we have $\lvert \varphi(u,\lambda)e^{-\lambda/z}\rvert \le \lvert P(\lambda)\rvert \cdot \exp{\left(-\lambda\left(\frac{\Re z}{|z|^2}-b\right)\right)}$ which decays exponentially. Next, let a small $\delta>0$ be given and write 
    \[
        I(u,z) = \int_0^\delta \varphi(u,\lambda)e^{-\lambda/z}\:d\lambda + \int_\delta^\infty \varphi(u,\lambda)e^{-\lambda/z}\:d\lambda.
    \]
    By \Cref{C:watsonlemma}, we can bound the second term as 
    \begin{align*}
        \left\lvert \int_\delta^\infty e^{-\lambda/z}\varphi(u,\lambda)\:d\lambda\right\rvert & \le \int_\delta^\infty \lvert e^{-\lambda/z}\rvert \cdot \lvert \varphi(u,\lambda)\rvert \:d\lambda \\
        & \le \int_\delta^\infty  C\exp{\left(\lambda\left(b-\frac{x}{|z|^2}+\epsilon\right)\right)}\:d\lambda\\
        & = C\exp{\left(\delta\left(b-\frac{x}{|z|^2}+\epsilon\right)\right)}\cdot\left(\frac{x}{\lvert z\rvert^2} -(b+\epsilon)\right)^{-1}
    \end{align*}
    Now, we consider the first term. Taylor's theorem with remainder applied to \eqref{E:taylorseries} implies that the remainder $R_N(u,\lambda) = \varphi(u,\lambda) - \sum_{n=0}^N a_n(u)\cdot \lambda^n$ can be bounded as $\lvert R_N(u,\lambda)\rvert \le C' \cdot \lvert \lambda\rvert^{N+1}$ for some $C'>0$ which does not depend on $u$. Next,
    \[
        \int_0^\delta \varphi(u,\lambda) e^{-\lambda/z} \:d\lambda = \sum_{n=0}^N a_n(u) \int_0^\delta \lambda^n e^{-\lambda/z}\:d\lambda + \int_0^\delta R_N(u,\lambda) e^{-\lambda/z}\:d\lambda.
    \]
    We estimate the remainder by
    \begin{align*}
        \left\lvert \int_0^\delta R_N(u,\lambda)e^{-\lambda/z}\:d\lambda \right\rvert & \le C' \int_0^\infty \lambda^{N+1}e^{-\frac{\lambda x}{\lvert z\rvert^2}}\: d\lambda = C' \cdot \frac{\Gamma(N+2)}{x^{N+2}}\cdot \lvert z\rvert^{2(N+2)}.
    \end{align*}
    Next, we estimate the summands of the first term as 
    \begin{align*}
        \int_0^\delta \lambda^n e^{-\lambda/z}\:d\lambda & = \int_0^\infty \lambda^ne^{-\lambda/z}\:d\lambda  -\int_\delta^\infty \lambda^ne^{-\lambda/z}\:d\lambda \\
        & = \Gamma(n+1)z^{n+1} + O(e^{-\delta/z})\text{ as } z\to 0.
    \end{align*}
    Combining everything, we obtain an asymptotic expansion 
    \[
        I(u,z) \sim \sum_{n=0}^\infty a_n(u)\cdot \Gamma(n+1)\cdot z^{n+1}.
    \]
    Indeed,
    \begin{align*}
        &\left\lvert I(u,z) - \sum_{n=0}^N a_n(u)\cdot \Gamma(n+1)\cdot z^{n+1}\right\rvert \leq \\
        &\exp{\left(\delta\left(b-\frac{x}{|z|^2}+\epsilon\right)\right)}\cdot\left(\frac{x}{\lvert z\rvert^2} -(b+\epsilon)\right)^{-1}+C' \cdot \frac{\Gamma(N+2)}{x^{N+2}}\cdot \lvert z\rvert^{2(N+2)} + O(e^{-\delta/z})
    \end{align*} 
    as $z\to 0$ in $\mathscr{S}(0,\tfrac{\pi}{2}+\epsilon)$.
\end{proof}

Next, we use \Cref{P:Watson} to derive estimates when $\varphi(u,\lambda)$ is valued in a complex vector space $V$, regarded as a section of the trivial $V$-bundle $\cV\to W\times \Omega$. Consider an analytic section $\Omega \to \Omega \times V$ given by $u\mapsto \Phi_u \in V$. Suppose that $\varphi(u,\lambda)$ is a holomorphic section of $\cV$ such that $\varphi(u,0) = \Phi_u$ for all $u\in \Omega$. 

\begin{cor}
\label{C:estimatevectorvalued}
    In the above notation, 
    \[
        I(u,z) = \int_0^\infty \varphi(u,\lambda)e^{-\lambda/z}\:d\lambda
    \]
    admits an asymptotic expansion $I(u,z) \sim \Phi_uz+\sum_{n\ge 1} a_n(u)z^{n+1}$ as $z\to 0$ such that the error of the first approximation $I(u,z) - \Phi_u \cdot z$ is in $O(\lvert z\rvert^2)$ and can be made independent of $u$.
\end{cor}

\begin{proof}
    This is immediate from \Cref{P:Watson}, applied componentwise in a trivialization $V\cong \bf{C}^N$, only noting that the Taylor expansion of $\varphi(u,\lambda)$ at $0$ reads $\varphi(u,\lambda) = \Phi_u + \sum_{n\ge 1} A_n(u)\cdot \lambda^n$.
\end{proof}

In particular, in the notation of \Cref{C:estimatevectorvalued}, $\lim_{z\to 0} \frac{1}{z} I(u,z) = \Psi_u$ and the convergence is uniform in $u\in \Omega$.

\singlespacing


\end{document}